\documentclass[reqno,11pt]{amsart}
\usepackage{amssymb,amsmath,amsthm}
\usepackage{a4wide}
\usepackage{mathrsfs}
\usepackage{color}
\usepackage{esint}
\usepackage{stmaryrd}
\usepackage[colorlinks,linkcolor=blue,anchorcolor=blue,citecolor=blue]{hyperref}

\newtheorem{theorem}{Theorem}

\newtheorem{lemma}[theorem]{Lemma}
\newtheorem{proposition}[theorem]{Proposition}
\newtheorem{remark}[theorem]{Remark}
\newcommand{\dd}{\mathrm{d}}

\newcommand{\bbR}{\mathbb{R}}

\newcommand{\1}{{\eta}}

\def\beq{\begin{equation}}
\def\eeq{\end{equation}}
\def\beqs{\begin{equation*}}
\def\eeqs{\end{equation*}}
\def\bal#1\eal{\begin{align}#1\end{align}}
\def\bals#1\eals{\begin{align*}#1\end{align*}}
\def\bsp#1\esp{\begin{split}#1\end{split}}

\def\d{{\mathrm{d}}}

\let\a=\alpha
\let\e=\varepsilon

\let\pt=\partial

\let\th=\theta

\let\b=\beta

\numberwithin{equation}{section}
\numberwithin{theorem}{section}

\allowdisplaybreaks[3]

\begin{document}
\date{}
\title[Diffusive Limit of One-species VMB for Cutoff Hard Potentials]{
\large{Diffusive Limit of the  One-species  Vlasov--Maxwell--Boltzmann System for Cutoff Hard Potentials
}}

\author{Weijun Wu$^\dagger$, Fujun Zhou$^*$, Weihua Gong$^\ddagger$ and Yuan Xu$^\S$}

\address[Weijun Wu$^\dagger$]{School of Network Security, Guangdong Police College, Guangzhou 510230, China}
\email{scutweijunwu@qq.com}

\address[Fujun Zhou$^*$, Corresponding author]{School of Mathematics, South China University of Technology, Guangzhou 510640, China}
\email{fujunht@scut.edu.cn}

\address[Weihua Gong$^\ddagger$]{School of Mathematics, Guangdong University of Education, Guangzhou 510303, China}
\email{weihuagong2020@163.com}

\address[Yuan Xu$^\S$]{School of Mathematics, South China University of Technology, Guangzhou 510640, China}
\email{yuanxu2019@163.com}

\begin{abstract}
Diffusive limit of the one-species Vlasov--Maxwell--Boltzmann system in perturbation framework still remains unsolved,
due to the weaker time decay rate compared with the two-species  Vlasov--Maxwell--Boltzmann system.
By employing the weighted energy method with two newly introduced weight functions and some novel treatments, we solve this problem
for the full range of cutoff hard potentials $0\leq \gamma \leq 1$.
Uniform estimate with respect to the Knudsen number $\varepsilon\in (0,1]$ is established globally in time, which eventually leads to
the global existence of solutions to the one-species Vlasov--Maxwell--Boltzmann system and hydrodynamic limit to the
incompressible Navier--Stokes--Fourier--Maxwell system. To the best of our knowledge, this is the first result on diffusive limit
of the one-species  Vlasov--Maxwell--Boltzmann system in perturbation framework.
 \\[3mm]
\\[1mm]
\end{abstract}

%
%
\maketitle

\tableofcontents

\section{Introduction}

\subsection{Description of the Problem}
\hspace*{\fill}

 The Vlasov--Maxwell--Boltzmann (VMB, for short) system is an important model in plasma physics to describe the time evolution of dilute charged particles
 under the influence of their self-consistent electromagnetic field. Despite its importance, diffusive limit to the fluid equations has not been completely solved so far,  such as the scaled two-species VMB system
 \begin{equation}\label{twoGG1}
\left\{
	\begin{array}{ll}
	\displaystyle \varepsilon \partial_t F^\varepsilon_{+}+v\cdot\nabla_xF^\varepsilon_{+}
     + (\varepsilon E^\varepsilon + v \times B^\varepsilon )\cdot\nabla_v F^\varepsilon_{+} =\frac{1}{\varepsilon}Q(F^\varepsilon_{+},F^\varepsilon_{+})+\frac{1}{\varepsilon}Q(F^\varepsilon_{+},F^\varepsilon_{-}),    \\[2mm]
		
	\displaystyle \varepsilon \partial_tF^\varepsilon_{-}+v\cdot\nabla_xF^\varepsilon_{-}
    - (\varepsilon E^\varepsilon + v \times B^\varepsilon )\cdot\nabla_v F^\varepsilon_{-}
    =\frac{1}{\varepsilon}Q(F^\varepsilon_{-},F^\varepsilon_{-})+\frac{1}{\varepsilon}Q(F^\varepsilon_{-},F^\varepsilon_{+}),   \\[2mm]
		
    \displaystyle \partial_t E^\varepsilon - \nabla_x\times B^\varepsilon = -\frac{1}{\varepsilon^2}\int_{\mathbb{R}^3} v (F^{\varepsilon}_{+}-F^{\varepsilon}_{-})\d v, \quad  \displaystyle \nabla_x \cdot E^\varepsilon= \frac{1}{\varepsilon}\int_{\mathbb{R}^3} (F^\varepsilon_{+}-F^\varepsilon_{-})\d v,\\ [2.5mm]
		
	\displaystyle \partial_t B^\varepsilon+\nabla_x \times E^\varepsilon=0, \qquad\qquad\qquad\qquad\qquad\;\,\nabla_x \cdot B^\varepsilon=0.
\end{array}\right.
\end{equation}
Here $F_{\pm}^{\varepsilon}(t,x,v)\geq 0$ are the density distribution functions for the ions $(+)$ and the electrons $(-)$, respectively, at time $t \geq 0$, position $x=(x_1, x_2, x_3)\in \mathbb{R}^{3}$ and velocity $v =(v_1, v_2, v_3) \in \mathbb{R}^{3}$ \cite{MR-90}.
The electromagnetic  field $[E^\e(t, x), B^\e(t,x)]$ is coupled with the density function $F^{\e}_{\pm}(t,x,v)$ through the   celebrated Maxwell system in (\ref{twoGG1}).
The positive parameter $\varepsilon\in (0,1]$ is the Knudsen number, which equals to the ratio of the mean free path to the macroscopic length scale.

 In  the physical situation, the ion mass is usually much
heavier than the electron mass, so that the electrons move faster than the ions.
Hence, the ions are often described by a fixed ion background $n_b(x)$, and only   electrons move. For this simplified situation, the two-species VMB system (\ref{twoGG1}) is reduced to the following scaled one-species VMB system
\begin{equation}\label{GG1}
\left\{
	\begin{array}{ll}
	\displaystyle \varepsilon \partial_t F^\varepsilon+v\cdot\nabla_xF^\varepsilon
   +\left(\e^2 E^\e+\e v\times B^\e\right)\cdot\nabla_v F^\varepsilon =\frac{1}{\varepsilon}Q(F^\varepsilon,F^\varepsilon), \\ [2mm]
		
    \displaystyle
    \pt_tE^\varepsilon-\nabla_x\times B^\varepsilon=
    -\frac{1}{\varepsilon}\int_{\mathbb{R}^3}v F^\varepsilon\d v, \qquad\;
		
	\nabla_x\cdot E^\varepsilon=\int_{\mathbb{R}^3}  F^\varepsilon\d v-n_b, \\ [3mm]
 \displaystyle
    \pt_tB^\varepsilon+\nabla_x\times E^\varepsilon=
    0, \qquad\qquad\qquad\qquad
		
	\nabla_x\cdot B^\varepsilon=0.
\end{array}\right.
\end{equation}
Here, $F^{\e}(t,x,v)\geq 0$ represents the density of the electrons at time $t\geq0$, velocity $v\in \mathbb{R}^{3}$
and position $x\in \mathbb{R}^{3}$, and the fixed constant ion background charge $n_b>0$  is scaled to $1$.
The initial data of the VMB system (\ref{GG1}) are imposed as
\begin{align}\label{GG1-initial}
  F^\varepsilon(0, x,v)= F_{0}^\varepsilon(x,v), \quad E^\varepsilon(0,x)=E^\varepsilon_0(x), \quad B^\varepsilon(0,x)=B^\varepsilon_0(x)
\end{align}
with the compatibility conditions
\begin{align*}
\nabla_x \cdot E_0^\varepsilon= \int_{\mathbb{R}^3} F^\varepsilon_{0}\d v-1,
    \quad \nabla_x \cdot B_0^\varepsilon=0.
\end{align*}
The Boltzmann collision operator $Q(G_1, G_2)$ in \eqref{GG1}
is defined as
\begin{align*}
  Q(G_1,G_2)(v)&:=\iint_{\bbR^3\times \mathbb{S}^2}
   |v-v_*|^{\gamma}q_0(\theta) \big[G_1(v')G_2(v_*')-G_1(v)G_2(v_*)\big] \d v_*\d \omega,
\end{align*}
where $v'=v-[(v-v_*)\cdot \omega]\omega$ and $v_*'=v_*+[(v-v_*)\cdot \omega]\omega$ by the conservation of momentum and energy
$$
v'+v_*'=v+v_*, \quad |v'|^2+|v_*'|^2=|v|^2+|v_*|^2.
$$
The collision kernel $ |v-v_*|^{\gamma}q_0(\theta)$, depending only on
$|v-v_*|$ and $\cos\theta=\omega\cdot\frac{v-v_*}{|v-v_*|}$, satisfies $-3<\gamma\leq 1$ and the Grad's
angular cutoff assumption
$$
0\leq q_0(\theta)\leq C|\cos\theta|.
$$
The exponent $\gamma$ is determined by the potential of intermolecular forces, and is classified into the
soft potential case when $-3<\gamma<0$ and the hard potential case when $0\leq \gamma\leq 1$. In particular, the hard potential
case includes the hard sphere model with $\gamma=1, q_0(\theta)=C|\cos\theta|$ and the Maxwell model with $\gamma=0$.

Much significant progress has been made on global well-posedness of the standard VMB system, including the two-species VMB system (\ref{twoGG1})  and
the one-species VMB system (\ref{GG1}) with $\e=1$. Firstly,  for the  standard two-species VMB system (\ref{twoGG1}) with $\e=1$, a significant breakthrough
was achieved by \cite{Guo2003}, wherein Guo constructed classical solutions in perturbation framework for the hard sphere case $\gamma=1$ in $\mathbb{T}^3$ by employing the robust high-order energy method. Subsequently, Strain and Duan--Strain \cite{DS2011, Strain2006CMP} established global existence and optimal time decay rate of classical solutions for  the hard sphere case $\gamma=1$  in $\mathbb{R}^3$. Ultimately, global existence for the full range of soft potentials $-3<\gamma<0$ was solved by \cite{DLYZ2017}. Secondly, for the standard one-species VMB system (\ref{GG1}) with $\e=1$, the mathematical theory is much more delicate than the two-species VMB system, due to the slow time decay rate of solutions to the one-species VMB system.
 The global existence of classical solutions in perturbation framework is obtained by Duan \cite{D2011} for the hard sphere case $\gamma=1$. Then
 the global existence and time decay estimates of classical solutions for the soft potential case $-1<\gamma<0$ was solved by \cite{LZ2016JDE}.
In addition, Li--Yang--Zhong \cite{LYZ2016-1} utilized spectral analysis method to study global solutions and optimal time decay rate of the VMB system with
the hard sphere case $\gamma=1$ for both two-species and one-species. Moreover, Ars\'{e}nio--Saint-Raymond \cite{AS2019} obtained global
renormalized  solutions  of the VMB system with general large initial data for both two-species and one-species.
Except for the progress on the cutoff VMB system described above, the standard non-cutoff VMB system also received some research,
such as \cite{DLYZ2013} for soft potentials $\max\{-3,-\frac{3}{2}-2s\}<\gamma<-2s$ with strong angular singularity $\frac{1}{2}\leq s<1$ and \cite{FLLZ2018} for soft potentials  $\max\{-3,-\frac{3}{2}-2s\}<\gamma<-2s$ with weak angular singularity $0<s<\frac{1}{2}$.

Hydrodynamic limit of the VMB system has also been the focus of extensive research efforts.
In a recent notable development, Ars\'{e}nio--Saint-Raymond \cite{AS2019} justified various limits (depending on the scalings)   towards incompressible viscous electro-magneto-hydrodynamics in the framework of renormalized solutions. Among these limits, the most two critical cases are  the two-fluid incompressible Navier--Stokes--Fourier--Maxwell (NSFM, for short) system with Ohm's law from the two-species VMB system \eqref{twoGG1} and the incompressible NSFM system from the one-species VMB system \eqref{GG1}.  This fundamental work opened up a series of subsequent progress on the study of the hydrodynamic limit of  the  VMB system, mainly for the two-species
VMB system in perturbation framework.  Jiang--Luo \cite{JL2019} justified the two-fluid incompressible NSFM   limit of \eqref{twoGG1} in perturbation framework for the cutoff hard sphere case $\gamma=1$, and also considered Hilbert expansion \cite{Ca} for the cutoff hard potential case $0\leq \gamma\leq 1$ \cite{JLZ2023ARMA}. Recently, Jiang--Lei \cite{JL2023ARXIV} investigated the two-fluid incompressible NSFM limit for the cutoff soft potentials $-1\leq \gamma<0$ and the non-cutoff soft potentials $\max\{-3,-\frac{3}{2}-2s\}<\gamma<-2s$ with strong angular singularity $\frac{1}{2}\leq s<1$. Then this two-fluid incompressible NSFM limit was extended to the non-cutoff potentials $\gamma>\max\{-3,-\frac{3}{2}-2s\}$ with full range $0<s<1$ \cite{XZGW2024ARXIV}.
Additionally, Yang--Zhong \cite{YZ24} demonstrated the above limit and the convergence rate for the cutoff hard sphere case $\gamma=1$ by spectral analysis method.
In addition to the incompressible NSFM limit of the two-species VMB system mentioned above, Jang \cite{J2009} studied the incompressible Vlasov--Navier--Stokes--Fourier limit for the cutoff hard sphere case $\gamma=1$ in $\mathbb{T}^3$.
However, for the hydrodynamic limit of the one-species VMB system in perturbation framework, the research so far have only focused on the compressible
Euler--Maxwell limit in hyperbolic regime, as presented by \cite{DYY2023MMMAS} for the cutoff hard sphere
case $\gamma= 1$ and \cite{LLXZ2023ARXIV} for the non-cutoff soft potential case $\max\{-3,-\frac{3}{2}-2s\}<\gamma<-2s$ with $0 <s<1$, respectively.
For more related works on hydrodynamic limit of relevant kinetic
equations, we refer to \cite{BGL91, BGL93, BU91, CIP, DEL, Esposito2018, GS, GZW-2021, guo2006, guo2010cmp, GJJ, GX2020, JXZ2018, LW2023,  LM, Saint, WZL2023ARXIV} and  the references cited therein.

As indicated from the aforementioned research progress,
the existing results on hydrodynamic limit of the one-species VMB system in perturbation framework  are all limited to the hyperbolic  limit toward the Euler--Maxwell system \cite{DYY2023MMMAS, LLXZ2023ARXIV},
 whereas diffusive limit of the one-species VMB system still remains open.
The purpose of this paper is to fill in this gap and investigate
diffusive limit of the one-species  VMB  system  \eqref{GG1}--\eqref{GG1-initial} for cutoff hard potentials $0\leq \gamma \leq 1$.
 In fact, diffusive limit of the one-species  VMB  system  \eqref{GG1}--\eqref{GG1-initial} is more difficult than
 the two-species VMB system, due to the weak time decay rate of solutions.
 By employing the weighted energy estimates with two newly introduced weight functions as well as some novel treatments, we
overcome the difficulty of slow time decay rate of solutions and establish the uniform estimate in $\varepsilon\in (0,1]$
globally in time, which eventually
leads to the convergence of the VMB system to the  incompressible NSFM system for the full range of  cutoff hard potentials.
This firstly justifies diffusive limit of classical solutions to the one-species Vlasov--Maxwell--Boltzmann system.

\subsection{Reformulation of the Problem}
\hspace*{\fill}

 We investigate diffusive limit of the one-species VMB system (\ref{GG1}) in perturbation framework around the global Maxwellian
$$
\mu=\mu(v) :=(2 \pi)^{-3 / 2} \mathrm{e}^{-|v|^{2} / 2}.
$$
By writing
$$
F^{\varepsilon}(t,x,v)=\mu+\varepsilon \mu^{1/2}f^{\varepsilon}(t,x,v),\quad
F^\varepsilon_{0}(x,v)=\mu+\varepsilon \mu^{1/2}f_{0}^{\varepsilon}(x,v)
$$
with the fluctuations $f^{\varepsilon}$ and $f_{0}^{\varepsilon}$,
the VMB system
(\ref{GG1})--\eqref{GG1-initial} is transformed into the following equivalent perturbed VMB system
\begin{align}\label{rVPB}
	\left\{\begin{array}{l}
\displaystyle \!\!\!\partial_{t} f^{\varepsilon}\!+\!\frac{1}{\varepsilon}v \!\cdot \nabla_{x} {f}^{\varepsilon}
\!+\!(\e E^\e\!+\!v\times B^\e)\!\cdot \nabla_{v} {f}^{\varepsilon}
\!-\! E^{\varepsilon} \!\cdot \!v \mu^{1/2}\!+\!\frac{1}{\varepsilon^{2}} L {f}^{\varepsilon}
=\!\e\displaystyle \!\frac{v}{2}\cdot E^{\varepsilon} {f}^{\varepsilon}\!+\!\frac{1}{\varepsilon} \Gamma({f}^{\varepsilon}\!,\! {f}^{\varepsilon}),  \\ [2mm]
\displaystyle\!\!\! \pt_tE^\varepsilon-\nabla_x\times B^\varepsilon=
    -\int_{\mathbb{R}^3}v \mu^{1/2}f^\varepsilon\d v,\qquad \qquad\;
	\nabla_x\cdot E^\varepsilon=\varepsilon \int_{\mathbb{R}^3}  \mu^{1/2}f^{\varepsilon}\d v, \\ [3mm]
 \displaystyle\!\!\!
    \pt_tB^\varepsilon+\nabla_x\times E^\varepsilon=
    0, \qquad\qquad\qquad\qquad	\qquad\;\;\;		
	\nabla_x\cdot B^\varepsilon=0, \\ [3mm]
\displaystyle \!\!\!f^\varepsilon(0,x,v)=f^\varepsilon_{0}(x,v),\qquad
    E^\varepsilon(0,x)=E^\varepsilon_{0}(x), \quad \; B^\varepsilon(0,x)=B^\varepsilon_{0}(x),
	\end{array}\right.
\end{align}
where  the linearized Boltzmann collision operator $Lf^{\varepsilon}$ and the nonlinear collision term $\Gamma(f^{\varepsilon}, f^{\varepsilon})$ are respectively defined by
\bals
Lf^{\varepsilon}\;&:=-\mu^{-1/2}
\left[Q(\mu,\mu^{1/2}f^\e)+Q(\mu^{1/2}f^\e,\mu)\right]\equiv\nu f -Kf,\\
\Gamma(f^{\varepsilon}, f^{\varepsilon})\;&:=\mu^{-1/2}
Q(\mu^{1/2}f^\e,\mu^{1/2}f^\e).
\eals
Here,  $\nu=\nu(v)$ is   the collision frequency and $K$ is a compact operator on $L^2(\mathbb{R}_v^3)$,   given by
\begin{equation}\label{nu-def}
\begin{split}
\nu\equiv \nu(v)&:=\frac{1}{\sqrt{\mu}}Q_-(\sqrt{\mu},\mu)=
\iint_{\mathbb{R}^3\times\mathbb{S}^2}|v-v_*|^{\gamma} q_0(\theta)\mu(v_*)\mathrm{d}\omega \mathrm{d}v_*,\\
Kf&:=\frac{1}{\sqrt{\mu}}\Big[Q_+(\mu,\sqrt{\mu}f)
+Q_+(\sqrt{\mu}f,\mu)-Q_-(\mu,\sqrt{\mu}f)\Big].
\end{split} \end{equation}

It is well-known that the operator $L$ is non-negative with null space
$$
\mathrm{N}(L) = \text{span} \Big\{ \mu^{1/2},  \ v \mu^{1/2} ,  \  \frac{|v|^{2}-3}{2} \mu^{1/2} \Big \}
$$
cf. \cite{CIP}.
For given     function $f(t,x,v)$, the following macro-micro decomposition  was introduced in \cite{Guo2004}
\begin{equation}\label{f decomposition}
f=\mathbf{P}f+(\mathbf{I}-\mathbf{P})f.
\end{equation}
Here, $\mathbf{P}$ denotes the orthogonal projection from $L^2(\mathbb{R}_{v}^3)$ to $\mathcal{N}(L)$, defined by
\begin{equation}\label{Pf define}
\mathbf{P}f:=\Big[a(t,x)+ v\cdot b(t,x) +\frac{|v|^2-3}{2}c(t,x)\Big]\mu^{1/2}.
\end{equation}
Furthermore, it is well-known that there exists a positive constant
 $\sigma_0>0$ such that
\begin{equation}\label{spectL}
\langle Lf, f\rangle_{L_{x,v}^2}\ge\sigma_0\|(\mathbf{I}-\mathbf{P})
f\|_{L_{x,v}^2(\nu)}^2.
\end{equation}

\subsection{Notations}
\hspace*{\fill}

Throughout the paper, $C$ denotes a generic positive constant independent of $\varepsilon$.
We use $X \lesssim Y$ to denote $X \leq CY$, where $C$ is a constant independent of $X$, $Y$. We also use the notation $X \approx Y$ to represent $X\lesssim Y$ and $Y\lesssim X$. The notation $X \ll 1$ means that $X$ is a positive constant small enough.

The multi-indexs $\alpha = [\alpha_1, \alpha_2, \alpha_3]\in\mathbb{N}^3$ and $\beta = [\beta_1, \beta_2, \beta_3]\in\mathbb{N}^3$ will be used to record space and velocity derivatives, respectively. We  denote the $\alpha$-th order space partial derivatives by $\partial_x^\alpha=\partial_{x_1}^{\alpha_1}\partial_{x_2}^{\alpha_2}\partial_{x_3}^{\alpha_3}$,
and the $\beta$-th order velocity partial derivatives by $\partial_v^\beta=\partial_{v_1}^{\beta_1}\partial_{v_2}^{\beta_2}\partial_{v_3}^{\beta_3}$.
In addition, $\partial^\alpha_\beta=\partial^\alpha_x\partial^\beta_v
=\partial^{\alpha_1}_{x_1}\partial^{\alpha_2}_{x_2}\partial^{\alpha_3}_{x_3}
\partial^{\beta_1}_{v_1}\partial^{\beta_2}_{v_2}\partial^{\beta_3}_{v_3}$ stands for the mixed space-velocity derivative.
The length of $\alpha$ is denoted by $|\alpha|=\alpha_1+\alpha_2+\alpha_3$.
If each component of $\theta$ is not greater than that of $\bar{\theta}$, we denote by $\theta \leq \bar{\theta}$. $\theta < \bar{\theta}$ means $\theta \leq \bar{\theta}$ and $|\theta| < |\bar{\theta}|$.
In addition, the notation $\nabla_x^k=\partial^\a_x$ will be used when $\b=0, |\a|=k$.

We  use $| \cdot |_{L^p_v}$ to denote the $L^p$ norm in $\mathbb{R}^3_v$, and  use $\|\cdot\|_{L^q_{x,v}}$  to denote the $L_{x,v}^q$-norm on the variables
$(x, v) \in \mathbb{R}^3\times \mathbb{R}^3$ for $1\leq q \leq \infty$.  For an integer $m \geq 1$, we also
employ $\|\cdot\|_{W^{m,q}_{x,v}}$ to represent  the norm of the Sobolev space
$W^{m,q}(\mathbb{R}^3\times \mathbb{R}^3)$. In particular, if $q=2$ then $\|\cdot\|_{H^{m}_{x,v}}$ stands for
the norm of $H^{m}(\mathbb{R}^3\times \mathbb{R}^3)=W^{m,2}(\mathbb{R}^3\times \mathbb{R}^3)$. Similarly,
$\|\cdot\|_{L^q_{x}}$,  $\|\cdot\|_{H^{m}_{x}}$,  $\|\cdot\|_{W^{m,q}_{x}}$ and
 $\|\cdot\|_{H^{m}_{x}L^2_v}$  stand for
 the norms of the function spaces $L^q(\mathbb{R}^3_x)$, $H^{m}(\mathbb{R}^3_x)$,
 $W^{m,q}(\mathbb{R}^3_x)$ and $L^{2}(\mathbb{R}^3_v,H^m(\mathbb{R}^3_x))$, respectively.
Besides, $\| \cdot \|_{L^p_x L^q_v}$ denotes $\| | \cdot |_{L^q_v} \|_{L^p_x}$.
The norm of a vector means the sum of the norms for all components
of this vector.

Let  $\langle\cdot, \cdot\rangle_{L^2_v}$  $\langle \cdot,\cdot\rangle_{L^2_{x}}$, $\langle \cdot,\cdot\rangle_{L^2_{x, v}}$ denote the $L^2_v$ inner product in $\mathbb{R}_v^3$,  the $L^2_x$ inner product in $\mathbb{R}_x^3$,
and  the $L^2_{x,v}$ inner product in $\mathbb{R}_{x,v}^3$,  respectively.
We also define the weighted  $L^2_{x,v}$ and $H^N_{x}L^2_{v}$  space endowed with the norms
\bals
\|g\|^2_{L^2_{x,v}(\nu)}&= \iint_{\mathbb{R}^3\times \mathbb{R}^{3}}\nu(v)|g(x,v)|^2\dd x \dd v,\quad
\|g\|^2_{H^N_{x}L^2_{ v}(\nu)}= \sum_{|\a|\leq N}\|\partial^\a_x g\|^2_{L^2_{x,v}(\nu)},
\eals
respectively, where $\nu=\nu(v)$ will be given in \eqref{nu-def}.

Finally,
we use $\Lambda^{-s}f(x)$ to denote
\begin{align*}
\Lambda^{-s} f(x):= {(2\pi)^{-3/2}} \int_{\mathbb{R}^3} |y|^{-s} \widehat{f}(y) e^{ix \cdot y} \d y,\qquad s\in(0,3)
\end{align*}
where $\widehat{f}(y) :=  {(2\pi)^{-3/2}} \int_{\mathbb{R}^3} f(x) e^{-ix \cdot y}\d x$ represents the Fourier transform of $f(x)$.

\subsection{Main Results}
\hspace*{\fill}

To state the main results of this paper, we introduce the following fundamental instant energy functional
\begin{align}
\label{without weight energy functional}
{\mathcal{E}}_{N} (t) \sim\;
&\sum_{|\alpha|\leq N} \left\|\partial^{\alpha}_xf^{\varepsilon}\right\|^2_{L^2_{x,v}}
+\sum_{|\alpha|\leq N} \left\|\partial^{\alpha}_xE^{\varepsilon}\right\|^2_{L^2_{x}}
+\sum_{|\alpha|\leq N} \left\|\partial^{\alpha}_xB^{\varepsilon}\right\|^2_{L^2_{x}},
\end{align}
and the corresponding dissipation rate functional
 \bal
 \bsp
\label{without weight dissipation functional}
{\mathcal{D}}_{N} (t)\sim \;
&\!\!\sum_{1 \leq |\alpha| \leq N} \!\!\left \| \partial^{\alpha}_x \mathbf{P}f^\varepsilon \right \|^2_{L^2_{x,v}}
\!+ \!\frac{1}{\varepsilon^2} \sum_{|\alpha| \leq N}\left\|\partial^{\alpha}_x
(\mathbf{I}\!-\!\mathbf{P})f^\varepsilon\right\|^2_{L^2_{x,v}(\nu)}
+\e^2\sum_{1\leq|\alpha| \leq N-1} \left\|\partial^{\alpha}_xE^{\varepsilon}\right\|^2_{L^2_{x}}\\
&+\e^2\sum_{2\leq|\alpha| \leq N-1} \left\|\partial^{\alpha}_x B^{\varepsilon}\right\|^2_{L^2_{x}},
\esp
\eal
 where the integer $N\in\mathbb{Z}$ will be determined later.

On the one hand, owing to  the weaker dissipation of the linearized Boltzmann operator $L$ for hard potentials rather than the hard sphere case, in order to deal with the Lorentz force term brought by the  electromagnetic field, we introduce the following time-velocity weight function
\begin{align}\label{weight function}
\overline{w}_{l_1}(\alpha,\beta)&:=
\langle v \rangle^{l_1-4|\alpha|-4|\beta|}e^{\frac{q\langle v\rangle^2}{(1+t)^\vartheta}},
\end{align}
where  $\langle v \rangle := \sqrt{1+|v|^2}$ and the constants $l_1, q,  \vartheta\geq 0$ will be determined later.  Moreover,
sometimes we write   $\overline{w}_{l_1}(|\a|,|\b|)$ to denote  $\overline{w}_{l_1}(\a,\b)$ for simplicity.
Then the  weighted instant energy functional ${\overline{\mathcal{E}}}_{N-1,l_1} (t)$
 and the corresponding dissipation rate functional ${\overline{\mathcal{D}}} _{N-1,l_1} (t)$ are defined  respectively as
\begin{align}\label{energy functional}
{\overline{\mathcal{E}}}_{N-1,l_1} (t) :=\;
& \sum_{\substack{|\alpha|+|\beta| \leq N-1 }}\left\|\overline{w}_{l_1}(\alpha, \beta) \partial_{\beta}^{\alpha} (\mathbf{I}-\mathbf{P})f^{\varepsilon}\right\|^2_{L^2_{x,v}}, \\
\label{dissipation functional}
{\overline{\mathcal{D}}} _{N-1,l_1} (t):= \;&
\frac{1}{\varepsilon^{2}} \sum_{\substack{|\alpha|+|\beta| \leq N-1 }}\left\|\overline{w}_{l_1}(\alpha, \beta) \partial_{\beta}^{\alpha}
(\mathbf{I}-\mathbf{P}) f^{\varepsilon}\right\|^2_{L^2_{x,v}(\nu)} \nonumber\\
&+ \sum_{\substack{|\alpha|+|\beta| \leq N-1 }}\frac{q\vartheta}{(1+t)^{1+\vartheta}}\left\|\langle v\rangle \overline{w}_{l_1}(\alpha, \beta) \partial_{\beta}^{\alpha}
(\mathbf{I}-\mathbf{P}) f^{\varepsilon}\right\|^2_{L^2_{x,v}}.
\end{align}

On the other hand, comparing the instant energy functional \eqref{without weight energy functional} with dissipation rate functional  \eqref{without weight dissipation functional}, one can find that the dissipations of electromagnetic field is coupled with $\e^2$, implying that the slow time decay rate of the one-species VMB system can not close the $\overline{w}_{l_1}(\alpha,\beta)$-weighted energy estimate. This urges us to introduce the second time-velocity weight function
\bal
\label{weight function 2}
\widetilde{w}_{l_2}(\alpha,\beta)\;&:=
\langle v \rangle^{l_2-4|\alpha|-\frac{7}{2}|\beta|}e^{\frac{q\langle v\rangle^2}{(1+t)^\vartheta}},
\eal
where    the constants $l_2, q,  \vartheta\geq 0$ will be determined later. Similarly,
sometimes we write   $\widetilde{w}_{l_2}(|\a|,|\b|)$ to denote  $\widetilde{w}_{l_2}(\a,\b)$ for simplicity.
Then  the weighted
instant energy functional ${\widetilde{\mathcal{E}}} _{N,l_2}(t)$ and the corresponding dissipation rate functional ${\widetilde{\mathcal{D}}} _{N,l_2} (t)$ are
respectively defined as
\begin{align}
{\widetilde{\mathcal{E}}} _{N,l_2}(t):= \;&
\sum_{0\leq|\a|+|\b|\leq N-1}
(1+t)^{-|\b|\frac{1+\vartheta}{2}} \left\|\widetilde{w}_{l_2}(\alpha, \beta) \partial_{\beta}^{\alpha} (\mathbf{I}-\mathbf{P})f^{\varepsilon}\right\|^2_{L^2_{x,v}}\nonumber\\
\;&+
(1+t)^{-\frac{1+\vartheta}{2}} \sum_{\substack{|\a|+|\b|= N,\\ |\a|\neq N}}
(1+t)^{-|\b|\frac{1+\vartheta}{2}} \left\|\widetilde{w}_{l_2}(\alpha, \beta) \partial_{\beta}^{\alpha} (\mathbf{I}-\mathbf{P})f^{\varepsilon}\right\|^2_{L^2_{x,v}}\nonumber\\
\;&
+(1+t)^{-\frac{1+\vartheta}{2}} \e\sum_{|\a|=N}
\left\|\widetilde{w}_{l_2}(\a, 0) \partial_x^{\a} f^{\varepsilon}\right\|^2_{L^2_{x,v}},\label{energy functional 2-N}\\
{\widetilde{\mathcal{D}}} _{N,l_2}(t):= \;&\sum_{0\leq|\a|+|\b|\leq N-1}
(1+t)^{-|\b|\frac{1+\vartheta}{2}}\frac{q\vartheta}{(1+t)^{1+\vartheta}}\left\|\langle v\rangle \widetilde{w}_{l_2}(\alpha, \beta) \partial_{\beta}^{\alpha}
(\mathbf{I}-\mathbf{P}) f^{\varepsilon}\right\|^2_{L^2_{x,v}}\nonumber\\
&+\sum_{0\leq|\a|+|\b|\leq N-1}(1+t)^{-|\b|\frac{1+\vartheta}{2}}\frac{1}{\varepsilon^{2}} \left\|\widetilde{w}_{l_2}(\alpha, \beta) \partial_{\beta}^{\alpha}
(\mathbf{I}-\mathbf{P})f^{\varepsilon}\right\|^2_{L^2_{x,v}(\nu)}\nonumber \\
\;&+(1+t)^{-\frac{1+\vartheta}{2}} \sum_{\substack{|\a|+|\b|\leq N,\\ |\a|\neq N}}
(1+t)^{-|\b|\frac{1+\vartheta}{2}}\frac{q\vartheta}{(1+t)^{1+\vartheta}}\left\|\langle v\rangle \widetilde{w}_{l_2}(\alpha, \beta) \partial_{\beta}^{\alpha}
(\mathbf{I}-\mathbf{P}) f^{\varepsilon}\right\|^2_{L^2_{x,v}}\nonumber\\
\;&+(1+t)^{-\frac{1+\vartheta}{2}} \sum_{\substack{|\a|+|\b|= N,\\ |\a|\neq N}}(1+t)^{-|\b|\frac{1+\vartheta}{2}}\frac{1}{\varepsilon^{2}} \left\|\widetilde{w}_{l_2}(\alpha, \beta) \partial_{\beta}^{\alpha}
(\mathbf{I}-\mathbf{P})f^{\varepsilon}\right\|^2_{L^2_{x,v}(\nu)} \nonumber\\
\;&+(1+t)^{-\frac{1+\vartheta}{2}} \sum_{|\a|= N}
\frac{\e q\vartheta}{(1+t)^{1+\vartheta}}\left\|\langle v\rangle \widetilde{w}_{l_2}(\alpha, 0)  \partial_x^{\alpha}
 f^{\varepsilon}\right\|^2_{L^2_{x,v}}\nonumber\\
\;&+(1+t)^{-\frac{1+\vartheta}{2}} \sum_{|\a|= N}\frac{1}{\varepsilon} \left\|\widetilde{w}_{l_2}(\alpha, 0) \partial_x^{\alpha}
(\mathbf{I}-\mathbf{P})f^{\varepsilon}\right\|^2_{L^2_{x,v}(\nu)}.\label{dissipation functional 2-N}
\end{align}

Moreover, a sufficient time decay rate is necessary to close the energy estimate. Due to the slow time decay rate of the linearized system of \eqref{rVPB},
the time decay estimate of the nonlinear problem can not be obtained by the Duhamel principle and semigroup estimate of the linearized problem. Thus,
we have to resort to the Sobolev space with negative index and introduce the following instant energy functional ${\mathcal{E}}_{-s} (t)$ and
the corresponding dissipation rate functional ${\mathcal{D}}_{-s} (t)$
\begin{align}
\label{negative sobolev energy}
{\mathcal{E}}_{-s} (t)\! :=\!\;
&   \left\| \Lambda^{-s}f^\varepsilon\right\|^2_{L^2_{x,v}}
\!+\!\left\| \Lambda^{-s}E^\varepsilon\right\|^2_{L^2_x}
+\left\| \Lambda^{-s}B^\varepsilon\right\|^2_{L^2_x},\\
\label{negative sobolev dissipation}
{\mathcal{D}}_{-s} (t)\! :=\!\;&
 \| \Lambda^{1-s} \mathbf{P}f^\varepsilon\|^2_{L^2_{x,v}}
\!\!+\!\frac{1}{\varepsilon^2}\| \Lambda^{-s} (\mathbf{I}\!-\!\mathbf{P})f^\varepsilon\|_{L^2_{x,v}(\nu)}^2
\!+\e^2 \| \Lambda^{1-s}E^{\varepsilon}\|^2_{L^2_{x}}
\!+\e^2 \| \Lambda^{2-s}\! B^{\varepsilon}\|^2_{L^2_{x}},
\end{align}
respectively, where the positive constant $s>0$ will be determined later.

Besides, we introduce the following energy functional with the lowest $k$th order space derivative $\mathcal{E}^k_{N_0}(t)$ and the corresponding dissipation rate functional with the lowest $k$-order space derivative ${\mathcal{D}}^k_{N_0} (t)$
\begin{align}
\label{low k energy}
\mathcal{E}^k_{N_0}(t)\sim\;
& \sum_{k \leq |\alpha| \leq N_0} \left\|\partial^{\alpha}_xf^{\varepsilon}\right\|^2_{L^2_{x,v}}
+\sum_{k \leq |\alpha| \leq N_0} \left\|\partial^{\alpha}_xE^{\varepsilon}\right\|^2_{L^2_{x}}
+\sum_{k \leq |\alpha| \leq N_0} \left\|\partial^{\alpha}_xB^{\varepsilon}\right\|^2_{L^2_{x}},  \\
\label{low k dissipation}
{\mathcal{D}}^k_{N_0} (t)\sim \;
& \sum_{k+1 \leq |\alpha| \leq N_0} \left \| \partial^{\alpha}_x \mathbf{P}f^\varepsilon \right \|^2_{L^2_{x,v}}
+ \frac{1}{\varepsilon^2} \sum_{k \leq |\alpha| \leq N_0} \left\|\partial^{\alpha}_x
(\mathbf{I}-\mathbf{P})f^\varepsilon\right\|^2_{_{L^2_{x,v}(\nu)}}\nonumber\\
&
+\sum_{k+1 \leq |\alpha| \leq N_0-1} \e^2\left\|\partial^{\alpha}_xE^{\varepsilon}\right\|^2_{L^2_{x}}
+\sum_{k+2 \leq |\alpha| \leq N_0-1}\e^2 \left\|\partial^{\alpha}_xB^{\varepsilon}\right\|^2_{L^2_{x}},
\end{align}
respectively, where $k=0,1$ and   the constant $N_0$ will be determined later.
\medskip

Our first main result gives the uniform global estimate with respect to $\varepsilon\in (0,1]$ of solutions to the one-species VMB system (\ref{rVPB}).

\begin{theorem}   \label{mainth1}
Let $0\leq\gamma\leq1$ and $0<\varepsilon \leq 1$. Introduce the following constants in sequence
\begin{equation}
\begin{split}\label{hard assumption}
 &N_0\geq 3\; (N_0\in \mathbb{Z}), \quad N=\left[\frac{5}{3}N_0\right],\quad 1 < s < 3/2, \quad \vartheta=\frac{s-1}{2}
  ,\\
 & l_1 \geq 4 N \quad\text{and}\; \quad l_2-l_1=\frac{N-1}{2}.
\end{split}
\end{equation}
If there exists a small constant $\delta_1>0$ independent of $\varepsilon$ such that the initial data
\begin{equation}\nonumber
\widetilde{\mathcal{E}}(0):=\mathcal{E}_{N}(0)+{\overline{\mathcal{E}}}_{N-1,l_1}(0)+ {\widetilde{\mathcal{E}}}_{N,l_2}(0)+\mathcal{E}_{-s}(0) \leq \delta_1,
\end{equation}
  then the one-species VMB system \eqref{rVPB} admits a unique global solution
$(f^{\varepsilon},E^\varepsilon,B^\varepsilon)$  satisfying
\begin{align}
&\e^s (1+t)^{\frac{s}{2}} \mathcal{E}_{N_0}^0(t)+\e^{1+s} (1+t)^{\frac{1+s}{2}} \mathcal{E}_{N_0}^1(t)
  \leq C \widetilde{\mathcal{E}}(0), \label{thm1 N l0 decay}\\[1mm]
&\mathcal{E}_{N}(t)+{\overline{\mathcal{E}}}_{N-1,l_1}(t)
+{\widetilde{\mathcal{E}}}_{N,l_2}(t)
+\mathcal{E}_{-s}(t) \leq C\widetilde{\mathcal{E}}(0)\label{thm1 widetilde N l0+l1 decay}
\end{align}
for any $t \geq 0$ and some positive constant $C>0$ independent of $\varepsilon$.
\end{theorem}
\medskip

The second main result is on the hydrodynamic limit from the one-species VMB system \eqref{rVPB}
to the   incompressible NSFM system.
\begin{theorem} \label{mainth3}
Let $(f^\varepsilon, E^\varepsilon,B^\varepsilon)$  be the
global solutions to the one-species VMB system \eqref{rVPB} constructed in Theorem
\ref{mainth1}  with initial data $f_0^\e={f}_0^\e(x,v), E_0^\e=E_0^\e(x),B_0^\e=B_0^\e(x)$. Suppose that there exist scalar functions $(\rho_0, \theta_0)=(\rho_0(x),  \theta_0(x))$ and vector-valued functions $(u_0, E_0, B_0)=(u_0(x),E_0(x), B_0(x))$ such that
\begin{align}
\begin{split}\label{theorem1.3 1}
&{f}_0^\varepsilon \to {f}_0\;\text{ strongly in } H^{N}_x L^2_v,
\;\; E_0^\varepsilon \to E_0 \;\text{ and }\; B_0^\varepsilon \to B_0 \; \text{ strongly in } H^{N}_x
\end{split}
\end{align}
as $\varepsilon \to 0$ and ${f}_0={f}_0(x, v)$ is of the form
\begin{align}\label{theorem1.3 2}
\begin{split}
\!\!\!f_0=\;&\Big(\rho_0+u_0\cdot v+\theta_0\frac{|v|^2-3}{2}\Big) \mu^{1/2}.
\end{split}
\end{align}

Then there hold
\begin{equation}\label{theorem1.3 4}
\begin{split}
&{f}^\varepsilon \to {f}
\; \text{ weakly}\!-\!* \text{ in }  L^\infty(\mathbb{R}^+; H^N_x L^2_v)
 \text{ and strongly in } C(\mathbb{R}^+; H^{N-1}_x L^2_v), \\
&E^\varepsilon \to E \;
\text{ weakly}\!-\!*  \text{ in }   L^\infty(\mathbb{R}^+;H^{N}_{x})
\text{ and strongly in } C(\mathbb{R}^+; H^{N-1}_x),\\
&B^\varepsilon \to B \;
\text{ weakly}\!-\!*  \text{ in }   L^\infty(\mathbb{R}^+;H^{N}_{x})
\text{ and strongly in } C(\mathbb{R}^+; H^{N-1}_x)
\end{split}
\end{equation}
as $\varepsilon \to 0$ and $f={f}(t, x, v)$ has the form
\begin{align}
\begin{split}\label{theorem1.3 5}
f=\;&\Big(\rho+u\cdot v+\theta\frac{|v|^2-3}{2}\Big){\mu}^{1/2}.
\end{split}
\end{align}
Moreover, the above mass density $\rho=\rho(t,x)$, bulk velocity $u=u(t,x)$, temperature $\theta=\theta(t,x)$,  electric field $E=E(t,x)$ and magnetic field  $B=B(t,x)$ satisfy
 \begin{align}
\begin{split}\label{jixian solution space}
&\left(\rho, u, \theta, E, B\right) \in  C(\mathbb{R}^+; H^{N}_{x})
\end{split}
\end{align}
 and the following  incompressible NSFM system
\begin{equation}\label{INSFP limit}
\left\{
\begin{array}{ll}
\displaystyle \partial_{t} u+u \cdot \nabla_{x} u-\nu \Delta_{x} u+\nabla_{x}P=E+u\times B,  &\nabla_{x} \cdot u=0,\\[2mm]
\displaystyle \partial_{t} \theta+u \cdot \nabla_{x} \theta-\kappa \Delta_{x} \theta=0, &\rho+\theta=0,\\[2mm]
\displaystyle \partial_{t} E-\nabla_{x}\times B=-u, &\nabla_{x} \cdot E=0,\\[2mm]
\displaystyle \partial_{t} B+\nabla_{x}\times E=0, &\nabla_{x} \cdot B=0,\\[2mm]
\displaystyle u(0)=\mathcal{P}u_0,\; \theta (0)=\frac{3}{5}\theta_0-\frac{2}{5}\rho_0,\;  E(0)=E_0,B(0)=B_0,&
\end{array} \right.
\end{equation}
where $\mathcal{P}$ is the Leray projection, and the viscosity coefficient $\nu$ and the heat conductivity coefficient $\kappa$ are defined in \eqref{nu define}
and \eqref{kappa define}, respectively.
\end{theorem}
\medskip

\begin{remark}
The results given in Theorems  \ref{mainth1} and  \ref{mainth3} indicate
 that
\bals
  \;&\sup_{0\leq t\leq \infty} \Big\{\Big\|F^\e-\mu-\e \Big(\rho+u\cdot v+\th\frac{|v|^2-3}{2}\Big){\mu}\Big\|_{H^{N}_{x}L^{2}_{v}}
  \Big\}=o(\e),\\
   \;&\sup_{0\leq t\leq \infty} \Big\{\big\|E^\varepsilon-E \big\|_{H^{N-1}_{x}}+\big\|B^\varepsilon-B \big\|_{H^{N-1}_{x}}\Big\}=o(1),
\eals
 which eventually justifies that the incompressible NSFM system \eqref{INSFP limit}, as the first order
 approximation, is the hydrodynamic limit of the one-species VMB system \eqref{rVPB}.
 \end{remark}
\medskip


\begin{remark}\label{remark-1.7}
It is worth pointing out the mechanism to introduce two different weight functions $\overline{w}_{l_1}(\a, \b)$  and $\widetilde{w}_{l_2}(\a, \b)$ in \eqref{weight function} and \eqref{weight function 2}. In fact,
the first weight function $\overline{w}_{l_1}(\a, \b)$ is used to overcome the velocity growth in the Lorentz force term, and is friendly to the
transport term $\frac{1}{\e}v \cdot \nabla_x f^\e$ in the sense that it does not involve any growth in velocity.
However, in this process the Lorentz force term brings a trouble term
 \bal\label{KUNNAN-1}
 \;&\sum_{|\a'|=1}\Big\langle
 -v\times \partial_x^{\a'} B^\e \cdot\nabla_v\partial_{\b}^{\a-\a'}(\mathbf{I}\!-\!\mathbf{P}) f^\e, \; \overline{w}^{2}_{l_1}(\a,\b)\partial_{\b}^\a(\mathbf{I}\!-\!\mathbf{P})f^\e
 \Big
\rangle_{L^2_{x,v}},
\eal
which  can not be controlled by the extra dissipation in \eqref{dissipation functional} brought by the first weight, because
 the  critical term  $\e^2\|\nabla_x B^\e \|_{L^\infty_x}^2$ can  only  reach  a time decay rate $(1+t)^{-1}$, less than the required decay rate $(1+t)^{-(1+\vartheta)}$, cf. \eqref{thm1 N l0 decay}.
 Thus, we have to design the second weight function $\widetilde{w}_{l_2}(\a, \b)$, which contributes another microscopic dissipation
 in \eqref{dissipation functional 2-N} and control the Lorentz force term \eqref{KUNNAN-1} under the help of the interpolation inequality, cf.
 Section \ref{The a Priori Estimate} for more details.

Note that this usage of two different weight functions is fundamentally different from \cite{DLYZ2017,JL2023ARXIV},
wherein the time decay rate is fast enough and only velocity growth need such treatment.

\end{remark}
\medskip

\begin{remark}\label{remark-1.5}
It is noteworthy that the more challenging diffusive limit of the one-species VMB system with cutoff soft potentials still remains unsolved,
although diffusive limit of the two-species VMB system with cutoff soft potentials $-1\leq \gamma<0$ has been justified by
\cite{JL2023ARXIV}. In fact, the approach, which we employ in this paper by applying two interactional weighted energy estimates, is not applicable to the one-species VMB system with cutoff soft potentials.
To illustrate this, let us define the second weight function
$
{w}_{l}(\alpha,\beta):=
\langle v \rangle^{l-m|\alpha|-(m- \frac{\gamma}{2}-1)|\beta|}e^{\frac{q\langle v\rangle^2}{(1+t)^\vartheta}}
$
and take the $L^2_{x,v}$ estimate of ${w}_{l}(\a,\b)\partial_{\b}^\a(\mathbf{I}-\mathbf{P})f^\e$. In this process,
the transport term and   the Lorentz force term will bring two trouble terms
\bal\label{KUNNAN-2}
\bsp
 \|\langle v \rangle {w}_{l}(\a+e_i,\b-e_i)\partial_{\b-e_i}^{\a+e_i} (\mathbf{I}-\mathbf{P}) f^\e\|_{L^2_{x,v}}^2
\esp
\eal
and
\bal\label{KUNNAN-3}
\;&\sum_{|\a'|=1}\e^2\|\nabla_x B^\e \|_{L^\infty_x}^2 \|\langle v\rangle^{-\gamma} w_{l}(|\a|\!-\!1,|\b|\! +\!1) \nabla_v\partial_{\b}^{\a-\a'\!}(\mathbf{I}\!-\!\mathbf{P}) f^\e \|_{L^2_{x,v}}^2,
\eal
respectively. On the one hand,  to guarantee that \eqref{KUNNAN-2} can be absorbed by the dissipation generated from
the $L^2_{x,v}$ estimate of $ {w}_{l}(\a+e_i,\b-e_i)\partial_{\b-e_i}^{\a+e_i}(\mathbf{I}-\mathbf{P})f^\e$, we have to design the iteration with   time decay factor $(1+t)^{-(1+\vartheta)}$  for one more velocity derivative.
On  the other hand,
 the time decay rate of $\|\nabla_xB^\e\|_{L^\infty_{x}}^2$ in the one-species VMB system \eqref{rVPB} can  only   reach
$\e^{-{\frac{2s+5}{2}}} (1+t)^{-\frac{2s+5}{4}},$
which is obviously much slower  than the two-species VMB system \cite{JL2023ARXIV} and has a singularity $\e^{-{\frac{2s+5}{2}}}$.  Due to
this weak time decay rate,   the term \eqref{KUNNAN-3} is out of control.  Consequently, diffusive limit of the one-species VMB system  \eqref{rVPB} with cutoff soft potentials still remains unsolved.
\end{remark}
\medskip

\subsection{Difficulties  and Innovations}
\hspace*{\fill}

In this subsection, we outline the difficulties and innovations proposed in this paper.

Our analysis is based on the uniform weighted energy estimates with respect to $\varepsilon\in (0,1]$ globally in time, which eventually leads to the incompressible NSFM limit.
To deal with the one-species VMB system \eqref{rVPB} for the hard potentials $0\leq\gamma \leq1$, we are faced with considerable difficulties induced by the weak dissipation
of the linearized Boltzmann operator and the slow time decay rate of the electromagnetic field.
In what follows, we point out the critical technical points in our treatment.
\subsubsection{Velocity weighted $H_{x,v}^N$ energy estimate}
\hspace*{\fill}

 Our velocity weighted $H_{x,v}^N$  energy estimate contains some new ingredients.
 The starting point is the following fundamental $H^N_{x}L^2_{v}$ energy estimate for the VMB system  \eqref{rVPB}
\begin{align}
\begin{split}\label{basic energy estimate-D}
 \;&\frac{\d}{\d t}\mathcal{E}_{N}(t)+\mathcal{D}_{N}(t)\\
\lesssim
 \;&
  \underbrace{\mathcal{E}_N(t)
\left(
\|\langle v\rangle\nabla_v(\mathbf{I}-\mathbf{P}) f^\varepsilon\|_{H^{N-2}_{x}L^2_v}^2
+
\|\langle v\rangle   (\mathbf{I}-\mathbf{P})f^\e\|_{H^{N-1}_{x}L^2_{v}}^2\right)}_{:=D_1}\\
  \;&+\!\underbrace{\e^3\|\nabla_x [E^{\e},\! B^{\e}]\|_{H^2_{x}}^2\Big(\e\|\langle v\rangle \nabla_x^N\!(\mathbf{I}\!-\mathbf{P})f^\varepsilon\|_{L^2_{x,v}}^2
 \!+\!\frac{1}{\e}\|\langle v\rangle \nabla_v \nabla_x^{N-1}\!(\mathbf{I}\!-\!\mathbf{P})f^\varepsilon
 \|_{L^2_{x,v}}^2\Big)}_{:=D_2}+\cdots.
\end{split}
\end{align}
The velocity growth of $D_1$ and $D_2$ in \eqref{basic energy estimate-D} urges us to design appropriate weighted energy estimate.

 On the one hand, to handle the velocity growth of $D_1$ in  \eqref{basic energy estimate-D}, we design a new time-velocity weight function $\overline{w}_{l_1}(\alpha,\beta)$ defined in \eqref{weight function} and consider the corresponding  weighted energy ${\overline{\mathcal{E}}}_{N-1,l_1} (t)$ defined in  \eqref{energy functional}.
 The advantage of this weight $\overline{w}_{l_1}(\alpha,\beta)$ is that the transport term does
not involve any growth in velocity, namely,
\bals
\bsp
\;&\Big\langle- \frac{1}{\e} \partial_{\b}^{\a}\left [v\cdot\nabla_x (\mathbf{I}-\mathbf{P}) f^\e\right] , w^{2}_{l_1}(\a,\b)\partial_{\b}^\a(\mathbf{I}-\mathbf{P})f^\e
\Big\rangle_{L_{x,v}^2}\\
=\;&\Big\langle- \frac{1}{\e} \partial_{\b- e_i}^{\a+ e_i}    (\mathbf{I}-\mathbf{P}) f^\e, w^{2}_{l_1}(\a,\b)\partial_{\b}^\a(\mathbf{I}-\mathbf{P})f^\e
\Big\rangle_{L_{x,v}^2}\\
\lesssim \;& \frac{1}{\e}\|\overline{w}_{l_1}(\a,\b)\partial_{\b}^\a(\mathbf{I}-\mathbf{P})f^\e\|_{L^2_{x,v}}
\|\overline{w}_{l_1}(\a+e_i,\b-e_i)\partial_{\b-e_i}^{\a+e_i} (\mathbf{I}-\mathbf{P}) f^\e\|_{L^2_{x,v}}
,
\esp
\eals
where we have used the notation $e_i$ to denote the multi-index with the $i$-th element unit and the rest zero, as well as the fact
$
\overline{w}_{l_1}(\alpha,\beta)=\overline{w}_{l_1}(\alpha+e_i,\beta-e_i).
$
 The usage of the weight $\overline{w}_{l_1}(\alpha,\beta)$  will generate an extra dissipation term
 \eqref{dissipation functional}, but this requires a certain time decay rate of $E^\e$ and  $B^\e$. However,   when performing
 the corresponding weighted energy estimate for Lorentz force terms, we find that for $|\a'|=1$,
 \bals
 \;&\Big\langle
 -v\times \partial_x^{\a'} B^\e \cdot\nabla_v\partial_{\b}^{\a-\a'}(\mathbf{I}\!-\!\mathbf{P}) f^\e, w^{2}_{l_1}(\a,\b)\partial_{\b}^\a(\mathbf{I}\!-\!\mathbf{P})f^\e
 \Big
\rangle_{L^2_{x,v}}\\
\lesssim \;&
\underbrace{\e^2\|\nabla_x B^\e \|_{L^\infty_x}^2 \|\langle v\rangle \overline{w}_{l_1}(|\a|\!-\!1,|\b|\! +\!1) \nabla_v\partial_{\b}^{\a-\a'\!}(\mathbf{I}\!-\!\mathbf{P}) f^\e \|_{L^2_{x,v}}^2}_{:=D_3}+\frac{\eta}{\e^2}\| \overline{w}_{l_1}(\a,\b)\partial_{\b}^{\a}(\mathbf{I}\!-\!\mathbf{P}) f^\e \|_{L^2_{x,v}}^2.
\eals
Here $D_3$ can not be controlled by the extra dissipation given in \eqref{dissipation functional}, in that the time decay rate of $\e^2\|\nabla_x B^\e \|_{L^\infty_x}^2$ can  only  reach  $(1+t)^{-1}$, less than the required decay rate $(1+t)^{-(1+\vartheta)}$, cf. \eqref{X define}. To overcome this difficulty,
our treatment is to utilize  the   interpolation method and  the Cauchy--Schwarz inequality with small $\eta$,
\bal
\bsp\label{Bdekunnan:D}
D_3
\lesssim\;&
\Big[\left(\e^{2-\varrho}\|\nabla_x B^\e \|_{L^\infty_x}\right)^{\frac{2}{\varrho}}\|\langle v\rangle^{\widetilde{\ell}} \overline{w}_{l_1}(|\a|-1,|\b| +1)\nabla_v\partial_{\b}^{\a-\a'}(\mathbf{I}-\mathbf{P}) f^\e \|_{L^2_{x,v}}^2\Big]^{\varrho}\\
\;&\times
\Big[\frac{1}{\e^2}\| \overline{w}_{l_1}(|\a|-1,|\b| +1) \nabla_v\partial_{\b}^{\a-\a'}(\mathbf{I}-\mathbf{P}) f^\e \|_{L^2_{x,v}}^2\Big]^{1-\varrho}\\
\lesssim\;& \left(\e^{2-\varrho}\|\nabla_x B^\e \|_{L^\infty_x}\right)^{\frac{2}{\varrho}}
\|\underbrace{ \langle v\rangle^{\widetilde{\ell}} \overline{w}_{l_1}(|\a|-1,|\b| +1) } \nabla_v\partial_{\b}^{\a-\a'}(\mathbf{I}-\mathbf{P}) f^\e \|_{L^2_{x,v}}^2\\
\;&+
\frac{\eta}{\e^2}\| \overline{w}_{l_1}(|\a|-1,|\b| +1) \nabla_v\partial_{\b}^{\a-\a'}(\mathbf{I}-\mathbf{P}) f^\e \|_{L^2_{x,v}}^2
\\
\lesssim\;& \left(\e^{2 }\|\nabla_x B^\e \|_{L^\infty_x}\right)^{2\widetilde{\ell}}
\frac{1}{\e^2}\|\overline{w}_{l_2}(|\a|-1,|\b| +1) \nabla_v\partial_{\b}^{\a-\a'}(\mathbf{I}-\mathbf{P}) f^\e \|_{L^2_{x,v}}^2+\cdots,
\esp
\eal
where $|\a'|=1$, $l_2=\widetilde{\ell}+l_1$ and $\varrho={\widetilde{\ell}}^{-1}$.
This indicates that once the  magnetic field $\e^{2 }\|\nabla_x B^\e \|_{L^\infty_x}$
has some decay, the total decay of $(\e^{2 }\|\nabla_x B^\e \|_{L^\infty_x})^{2\widetilde{\ell}}$ will be large enough
as the weight index $\widetilde{\ell}$ increases.
Therefore,  motivated by \cite{DLYZ2017, JL2023ARXIV}, we  design the second  weight  $\widetilde{w}_{l_2}(\a,\b)$ defined in \eqref{weight function 2} to control
the underbraced term in \eqref{Bdekunnan:D}.

The advantage of the weight $\widetilde{w}_{l_2}(\alpha, \beta)$ is that it can balance the singularity and velocity growth
in the weighted energy estimate. More precisely, for the $L^2_{x,v}$-estimate of
$
  \widetilde{w}_{l_2}(\alpha, \beta) \partial_{\beta}^{\alpha} (\mathbf{I}-\mathbf{P})f^{\varepsilon}
$ with $|\a|+|\b|=n\leq N-1$, the most trouble term induced by the Lorentz field can be controlled by
 \bal
 \bsp\label{weightfunction2D_1}
 \;&\sum_{|\a'|=1}\Big\langle
 -v\times \partial_x^{\a'} B^\e \cdot\nabla_v\partial_{\b}^{\a-\a'}(\mathbf{I}-\mathbf{P}) f^\e, \widetilde{w}^{2}_{l_2}(\a,\b)\partial_{\b}^\a(\mathbf{I}-\mathbf{P})f^\e
 \Big
\rangle_{L^2_{x,v}}\\
\lesssim \;&
\e^2\|\nabla_x B^\e \|_{L^\infty_x}^2\sum_{|\a'|=1} \|\langle v\rangle^{\frac{1}{2}} \widetilde{w}_{l_2}(|\a|-1,|\b| +1)\nabla_v\partial_{\b}^{\a-\a'}(\mathbf{I}-\mathbf{P}) f^\e \|_{L^2_{x,v}}^2
\\\;&+\frac{\eta}{\e^2}\|  \widetilde{w}_{l_2}(\a,\b)\partial_{\b}^{\a}(\mathbf{I}-\mathbf{P}) f^\e \|_{L^2_{x,v}}^2\\
\lesssim \;&
\e^3\|\nabla_x B^\e \|_{L^\infty_x}^2{(1+t)^{\frac{1+\vartheta}{2}}} {(1+t)^{-\frac{1+\vartheta}{2}}}\!\! \sum_{|\a'|=1}\!\!\|\langle v\rangle \widetilde{w}_{l_2} (|\a|\!-\!1,|\b|\! +\!1) \nabla_v\partial_{\b}^{\a\!-\!\a'}\!(\mathbf{I}\!-\!\mathbf{P}) f^\e \|_{L^2_{x,v}}\\
\;&\times \frac{1}{\e}  \| \widetilde{w}_{l_2}(|\a|-1,|\b| +1) \nabla_v\partial_{\b}^{\a-\a'}(\mathbf{I}-\mathbf{P}) f^\e \|_{L^2_{x,v}}+\cdots,
\esp
 \eal
where we have used the fact
$
  \langle v\rangle  \widetilde{w}_{l_2}(\a,\b)=  \langle v\rangle^{\frac{1}{2}}\widetilde{w}_{l_2}(\a-e_i, \b+e_i).
$
However, the shortcoming of this weight $\widetilde{w}_{l_2}(\alpha, \beta)$ lies in that the transport term involves velocity growth, that is,
\bal
\bsp\label{weightfunction2D_2}
\;&\Big\langle- \frac{1}{\e} \partial_{\b}^{\a}\left [v\cdot\nabla_x (\mathbf{I}-\mathbf{P}) f^\e\right] , \widetilde{w}^{2}_{l_2}(\a,\b)\partial_{\b}^\a(\mathbf{I}-\mathbf{P})f^\e
\Big\rangle_{L_{x,v}^2}\\
\lesssim \;& C_{\eta}\|\langle v \rangle^{\frac{1}{2}}\widetilde{w}_{l_2}(\a+e_i,\b-e_i)\partial_{\b-e_i}^{\a+e_i} (\mathbf{I}-\mathbf{P}) f^\e\|_{L^2_{x,v}}^2
+\frac{\eta}{\e^2}
\|\widetilde{w}_{l_2}(\a,\b)\partial_{\b}^\a(\mathbf{I}-\mathbf{P})f^\e\|_{L^2_{x,v}}^2
.
\esp
\eal
 In order to    absorb  the first term on the right-hand side of \eqref{weightfunction2D_2} by the microscopic dissipation induced by
  the $L^2_{x,v}$-estimate of $ \widetilde{w}_{l_2}(\a+e_i,\b-e_i)\partial_{\b-e_i}^{\a+e_i} (\mathbf{I}-\mathbf{P}) f^\e$, it is essential for us   to arrange  different time increments for the various space-velocity derivatives   of  $f^\e$  as
\bals
\sum_{|\a|+|\b|\leq N-1} (1+t)^{-|\b|\frac{1+\vartheta}{2}}\left\|\widetilde{w}_{l_2}(\alpha, \beta) \partial_{\beta}^{\alpha} (\mathbf{I}-\mathbf{P})f^{\varepsilon}\right\|^2_{L^2_{x,v}}.
\eals
Consequently, with this arrangement, the  first term on the right-hand side of \eqref{weightfunction2D_1}  can be controlled, provided that $$\e^3\|\nabla_x B^\e \|_{L^\infty_x}^2{(1+t)^{1+\vartheta}}\lesssim C,$$
cf. the proof of Proposition \ref{weighted 2:diyuN} for more details.

On the other hand,    the trouble term $D_2$  in  \eqref{basic energy estimate-D}  is associated with the  $N$-th order space-velocity derivative   of $(\mathbf{I}-\mathbf{P})f^\e$. Observing that  $D_2$ includes the factor  $\e^3\|\nabla_x [E^{\e},\! B^{\e}]\|_{H^2_{x}}^2$, which possess the time decay rate   $(1+t)^{-(1+\vartheta)}$,  we therefore attempt to  consider  weighted energy estimate of  $\widetilde{w}_{l_2}(\alpha, \beta) \partial_{\beta}^{\alpha} (\mathbf{I}-\mathbf{P})f^{\varepsilon}$ with $N$-th order spatial-velocity derivatives $|\alpha|+ |\beta|=N.$
However,
the usage of the weight function $\widetilde{w}_{l_2}(\alpha, 0)$ also generates a severe singularity when we handle the $N$-th order space derivative on the linear Boltzmann operator $L$, that is
\begin{align*}
\frac{1}{\varepsilon^2} \left\langle L \partial^\alpha_x f^\varepsilon,
\widetilde{w}^2_{l_2}(\alpha,0)\partial^\alpha_x f^\varepsilon\right\rangle_{L^2_{x,v}}
\geq\;&\frac{\sigma_0}{\varepsilon^2} \left\| \widetilde{w}_{l_2}(\alpha,0) \partial^\alpha_x f^\varepsilon\right\|^2_{L^2_{x,v}(\nu)}
-\frac{1}{\varepsilon^2}\left\|\partial^\alpha_x f^\varepsilon\right\|^2_{L^2_{x,v}},
\quad |\alpha|=N,
\end{align*}
where the last term includes the singular macroscopic quantity $\dfrac{1}{\varepsilon^2}\left\|\partial^\alpha_x \mathbf{P} f^{\varepsilon}\right\|_{L^2_{x,v}}^2$ and is out of control.
We solve this difficulty by first multiplying the first  equation in  \eqref{rVPB} with $\varepsilon$ and then making $\widetilde{w}_{l_2}^2(\alpha,0)$-weighted energy estimate $(|\alpha|=N)$, namely
\begin{align*}
&\frac{\varepsilon}{2}\frac{\d}{\d t}  \left\|  \widetilde{w}_{l_2}(\alpha,0) \partial^\alpha _x f ^\varepsilon\right\|_{L^2_{x,v}}^2 + \frac{1}{\varepsilon}\left\langle L \partial^\alpha_x f^{\varepsilon}, \widetilde{w}_{l_2}^2(\alpha,0) \partial^\alpha_x f^{\varepsilon}\right\rangle_{L^2_{x,v}}\\
=\;&\frac{\varepsilon}{2}\frac{\d}{\d t}  \left\|  \widetilde{w}_{l_2}(\alpha,0) \partial^\alpha _xf ^\varepsilon\right\|_{L^2_{x,v}}^2
+\frac{1}{\varepsilon}\left\langle L \partial^\alpha_x (\mathbf{I}-\mathbf{P})f^{\varepsilon}, \widetilde{w}_{l_2}^2(\alpha,0) \partial^\alpha_x (\mathbf{I}-\mathbf{P})f^{\varepsilon}\right\rangle_{L^2_{x,v}}\\
&+\frac{1}{\varepsilon}\left\langle L \partial^\alpha (\mathbf{I}-\mathbf{P})f^{\varepsilon}, \widetilde{w}_{l_2}^2(\alpha,0) \partial^\alpha _x \mathbf{P}f^{\varepsilon}\right\rangle_{L^2_{x,v}}\\
\gtrsim\;&\frac{\varepsilon}{2}\frac{\d}{\d t} \left\|  \widetilde{w}_{l_2}(\alpha,0) \partial^\alpha_x f ^\varepsilon\right\|_{L^2_{x,v}}^2
+ \frac{\sigma_0}{\varepsilon} \left\| \widetilde{w}_{l_2}(\alpha,0) \partial^\alpha_x (\mathbf{I}-\mathbf{P})f^\varepsilon\right\|_{L^2_{x,v}(\nu)}^2
-\frac{C}{\varepsilon^2}\left\|\partial^\alpha_x (\mathbf{I}-\mathbf{P})f^\varepsilon\right\|^2_{L^2_{x,v}(\nu)}\\
&
-\left\|\partial^\alpha_x \mathbf{P} f^{\varepsilon}\right\|_{L^2_{x,v} }^2,\qquad\qquad\qquad\qquad\qquad \qquad\qquad\qquad\qquad\qquad\qquad \qquad
\end{align*}
then the last two terms on the right-hand side can be controlled by the dissipation $\mathcal{D}_{N}(t)$.

Meanwhile, due to the loss of the highest order   derivatives of electric field $\e\|\nabla^N_x E^\e\|_{L^2_x}$ in the dissipation \eqref{without weight dissipation functional}, observing that
\bals
\sum_{|\a|=N}\Big\langle  \partial_x^\a (E^\e\cdot v \sqrt{\mu}) , \e \widetilde{w}^{2}_{l_2}(\a,0)\partial_x^\a f^\e
\Big\rangle_{L_{x,v}^2}\lesssim\e\|\nabla^N_x E^\e\|_{L^2_x}\|\nabla^N_x f^\e\|_{L^2_{x,v}},
\eals
we use a time factor $(1 + t)^{-\frac{1+\vartheta}{2} }$ to calculate it.
That is, the desired $L^2_{x,v}$ energy estimate of the $N$-th order space derivative is $(1 + t)^{-\frac{1+\vartheta}{2} }\e\| \widetilde{w}_{l_2}(\a,0)\partial_x^\a f^\e\|_{L^2_{x,v}}^2$, which implies
\bals
\sum_{|\a|=N}\Big\langle  \partial_x^\a (E^\e\cdot v \sqrt{\mu}) , (1 + t)^{-\frac{1+\vartheta}{2} }\e \widetilde{w}^{2}_{l_2}(\a,0)\partial_x^\a f^\e
\Big\rangle_{L_{x,v}^2}\lesssim\; \eta\e^2(1+t)^{-(1+\vartheta)}\|\nabla^N_x E^\e\|_{L^2_x}^2+ \mathcal{D}_N(t).
\eals
In conclusion, we eventually
design the following velocity weighted $N$-th order energy
\bals
\sum_{\substack{|\a|+|\b|= N,\\ |\a|\neq N}}\!\!\!\!\!\!
(1+t)^{-(|\b|+1)\frac{1+\vartheta}{2}} \left\|\widetilde{w}_{l_2}(\alpha, \beta) \partial_{\beta}^{\alpha} (\mathbf{I}-\mathbf{P})f^{\varepsilon}\right\|^2_{L^2_{x,v}}
+ (1+t)^{-\frac{1+\vartheta}{2}} \e\sum_{|\a|=N}\!\!\!
\left\|\widetilde{w}_{l_2}(\a, 0) \partial_x^{\a} f^{\varepsilon}\right\|^2_{L^2_{x,v}}.
\eals
For more details, please see the proof of Proposition \ref{weighted 2:N}.

\subsubsection{Time decay estimates of lower-order derivatives $\e^s\mathcal{E}^0_{N_0}(t)$ and $\e^{1+s}\mathcal{E}^1_{N_0}(t)$}
\hspace*{\fill}

To close the a priori estimates, we also need sufficient time decay rates of
$\e^s\mathcal{E}^0_{N_0}(t)$ and $\e^{1+s}\mathcal{E}^1_{N_0}(t)$. Due to the slow time decay rate of the linearized VMB system \cite{D2011} and the singularity $1/\varepsilon$ in front of the nonlinear terms, the method of combining the semigroup theory of linearized equation with the Duhamel principle is no longer applicable to the nonlinear problem. To overcome this difficulty, inspired by \cite{GW2012CPDE}, we employ the interpolation and  energy estimate in  negative Sobolev space $\|\Lambda^{-s} \big(f^\varepsilon, E^\varepsilon, B^\varepsilon\big) \|$ to obtain the time decay rate.
In contrast to the two-species VMB system  \cite{JL2023ARXIV},    one
major difference lies in the Lorentz force term
\bal\label{decay-kunnan-1}
\left\langle \partial^\a_x \left[(v\times B^\e)\cdot \nabla_v  \mathbf{P}f^\e \right], \partial^\a_x \mathbf{P}f^\e\right\rangle_{L^2_{x,v}}\neq 0
\eal
 when $|\a|> 0$. Therefore,  in the energy estimate of  $\mathcal{E}^1_{N_0}(t)$, the tricky term  \eqref{decay-kunnan-1} is
 handled by
\bals
\eqref{decay-kunnan-1}
\lesssim\;&\!
\Big[ \sum_{1\leq |\a'|\leq N_0-1} \|\partial^{\a'}_xB^\e \|_{L^3_x}\| \partial^{\a-\a'}_x \mathbf{P} f^\e\!\|_{L^6_x L^2_v}
+ \!\|\nabla^{N_0}_xB^\e \|_{L^2_x}\| \mathbf{P} f^\e\|_{L^\infty_x L^2_v}\Big]
\|\partial^\a_x \mathbf{P}f^\varepsilon\|_{L^2_{x, v}}\\
\lesssim\;&
\|\nabla_x B^\e\|_{H^{N_0-1}_x}\mathcal{D}^0_{N_0}(t),
\eals
which further implies that we can only obtain
 the energy inequality
\bal\label{decay-kunnan-2}
\frac{\d}{\d t} \mathcal{E}^1_{N_0}(t)+   \mathcal{D}^1_{N_0}(t)
\lesssim \|\nabla_x B^\e\|_{H^{N_0-1}_x}\mathcal{D}^0_{N_0}(t).
\eal
 The last nonhomogeneous source term in \eqref{decay-kunnan-2} hinders us from applying the interpolation and
 negative Sobolev space mentioned above.
Fortunately,  thanks to \eqref{0-N estimate} and \eqref{decay:estimate:guocheng:1}, we find that
\bals
\bsp
\;&\sup_{0\le\tau\le t}\left\{\e^s(1+t)^{\frac{s}{2}} \mathcal{E}^0_{N_0}(t)\right\}+
(1+t)^{-p}\int_0^t
\e^s(1+\tau)^{\frac{s}{2}+p} \mathcal{D}^0_{N_0}(\tau)\dd \tau\\
\lesssim\;&{\mathcal{E}}_{N}(0)
+\sup_{0\leq \tau\leq t}\left\{{\mathcal{E}}_{N_0+\frac{s}{2}}(\tau)+
{\mathcal{E}}_{-s}(\tau)\right\},
\esp
\eals
indicating that once $\e(1+t)^{\frac{1}{2}}\|\nabla_x B^\e\|_{H^{N_0-1}_x}\lesssim \eta$ is small enough, the last
nonhomogeneous term in \eqref{decay-kunnan-2}  will be controlled in the way
\bals
\bsp
\;&\e^{1+s} \int_0^t  (1+\tau)^{\frac{1+s}{2}+p} \|\nabla_x B^\e(\tau)\|_{H^{N_0-1}_x}\mathcal{D}^0_{N_0}(\tau)\dd \tau
\lesssim\eta\e^{s}\int_0^t  (1+\tau)^{\frac{s}{2}+p}
\mathcal{D}^0_{N_0}(\tau)\dd \tau
\esp
\eals
in the process of obtaining the time decay rate $(1+t)^{\frac{1+s}{2}}$ of $\e^{1+s}\mathcal{E}^1_{N_0}(t)$ from \eqref{decay-kunnan-2}.
Consequently, based on the analysis above, we are able to obtain the time decay estimates of $\e^s\mathcal{E}^0_{N_0}(t)$ and $\e^{1+s}\mathcal{E}^1_{N_0}(t)$,  please see the proof of Proposition \ref{k-N decay proposition} for details.
By combining the above strategies, we eventually close the global a priori estimates successfully.

\medskip

The rest of this paper is organized as follows. In Section \ref{The a Priori Estimate}, we establish the uniform a priori weighted energy estimate. In Section \ref{Global Existence}, we first obtain the time decay rate and close the a priori estimates, and then we give the proof of Theorem \ref{mainth1}. In section \ref{Limit section}, based on the uniform energy estimate with respect to $\e \in (0,1]$ globally in time, we justify
the limit of the one-species VMB system \eqref{rVPB} to the  incompressible NSFM system \eqref{INSFP limit}, that is, give the proof of Theorem \ref{mainth3}.

\section{The a Priori Estimates}\label{The a Priori Estimate}

In this section, we deduce the uniform a priori estimates  of the one-species VMB system \eqref{rVPB} with respect to $\e\in (0,1]$
globally in time.
For this purpose, we define the following time-weighted energy norm $X(t)$ by
\bal
\bsp\label{X define}
X(t):=\;&
\sup_{0 \leq \tau \leq t} \Big\{\mathcal{E}_{N}(\tau)+{\overline{\mathcal{E}}}_{N-1,l_1}(\tau)
+{\widetilde{\mathcal{E}}}_{N,l_2}(\tau)
+\mathcal{E}_{-s}(\tau)+\e^s (1+\tau)^{\frac{s}{2}} \mathcal{E}_{N_0}^0(\tau)\Big\}
\\
\;&+\sup_{0 \leq \tau \leq t}\left\{\e^{1+s} (1+\tau)^{\frac{1+s}{2}} \mathcal{E}_{N_0}^1(\tau)+\e^{3} (1+\tau)^{{1+\vartheta}} \left\|\nabla_x\left[E^\e(\tau), B^\e(\tau)\right]\right\|_{H^2_x}^2
\right\}.
\esp
\eal
Here, all the involved parameters are fixed to satisfy  \eqref{hard assumption}.

The construction of the weighted instant energy functionals $\mathcal{E}_{N}(t)$, ${\overline{\mathcal{E}}}_{N-1,l_1}(t)$,
${\widetilde{\mathcal{E}}}_{N,l_2}(t)$  and
$\mathcal{E}_{-s}(t)$   will be given in the following. Suppose that the one-species VMB system  \eqref{rVPB} admits a smooth solution $\left(f^\varepsilon,E^\varepsilon,B^\varepsilon\right)$ over $0 \leq t \leq T$ for $0 < T \leq \infty$, and  the solution
$\left(f^\varepsilon,E^\varepsilon,B^\varepsilon\right)$ also satisfies
\begin{align}\label{priori assumption}
\sup_{0 \leq t \leq T} X(t) \leq \delta_0,
\end{align}
where $\delta_0$ is a suitable small positive constant to be determined later.

\subsection{Fundamental Energy Estimate}
\hspace*{\fill}

In this subsection, we aim to establish the   uniform  energy estimate without weight for the VMB system   \eqref{rVPB}.

To be precise, we begin with decomposing the perturbation $f^\e$ into
\begin{equation}\label{fdefenjie}
f^\e =\mathbf{P}f^\e +(\mathbf{I}-\mathbf{P})f^\e,
\end{equation}
where the macro part  $\mathbf{P}f^\e$ is given by
\bal\label{defination:Pf}
\mathbf{P}f^\e=\Big[a^\e(t,x)+b^\e(t,x)\cdot v+c^\e(t,x)\frac{|v|^2-3}{2}\Big]{\mu}^{1/2}
\eal
with $a^\e(t,x), b^\e(t,x), c^\e(t,x)$ defined as
\begin{align}
a^\e(t,x):=\langle \mu^{1/2}, \mathbf{P}f^\e \rangle_{L^2_v},\quad
b^\e(t,x):=\langle v\mu^{1/2}, \mathbf{P}f^\e\rangle_{L^2_v}, \quad
c^\e(t,x):=\;& \frac{1}{3} \big\langle(|v|^2-3)\mu^{1/2},\mathbf{P}f^\e \big\rangle_{L^2_v}. \nonumber
\end{align}

We are now in a position to state   the main result of this subsection.
\begin{proposition}\label{result-basic energy estimate}Let $N, l_1,
 l_2$ be fixed parameters  stated in Theorem \ref{mainth1}.
Assume that $ (f^\e, E^\e, B^\e) $ is a solution to the VMB system  \eqref{rVPB} defined on $ [0, T] \times \mathbb{R}^3 \times \mathbb{R}^3$
and the \emph{a priori} assumption  \eqref{priori assumption} holds true for $\delta_0$ small enough.
Then, there exists an energy functional $\mathcal{E}_{N}(t)$ and the
corresponding energy dissipation functional $\mathcal{D}_{N}(t)$ which satisfy \eqref{eq:energy:estimate:result-hard-1}, \eqref{without weight dissipation functional} respectively
such that
\begin{align}\label{basic-energy-estimate-result}
\begin{split}
\frac{\d}{\d t}\mathcal{E}_{N}(t)+\mathcal{D}_{N}(t)
\lesssim\;&
\delta_0 {\overline{\mathcal{D}}}_{N-1,l_1}(t)
+\delta_0 {\widetilde{\mathcal{D}}}_{N,l_2}(t).
\end{split}
\end{align}
\end{proposition}

The main idea of  proving Proposition
\ref{result-basic energy estimate} is that by combining direct energy methods
on the perturbation VMB system \eqref{rVPB} and
the so-called macro-micro decomposition method to   acquire the missing macro dissipations.
\medskip

First of all, we   find
a dissipation structure of the macroscopic part $\mathbf{P}f^\e$. 
\begin{lemma}\label{macroscopic estimate}
Let $\left(f^\varepsilon,E^\varepsilon, B^\varepsilon\right)$ be the solution to the VMB system \eqref{rVPB} defined on $ [0, T] \times \mathbb{R}^3 \times \mathbb{R}^3$.
Then there exists an interactive functional $\mathcal{E}^N_{int}(t)$  defined in \eqref{macro energy2}  satisfying
\begin{align}
\begin{split}\nonumber
\left|\mathcal{E}^N_{int}(t)\right| \lesssim\;&
\sum_{|\alpha| \leq N} \left\| \partial^\alpha_x f^{\varepsilon}\right\|^2_{L^2_{x,v}}
+\sum_{|\alpha| \leq N} \left\| \partial^\alpha_x E^{\varepsilon}\right\|^2_{L^2_{x}}+\sum_{|\alpha| \leq N} \left\| \partial^\alpha_x B^{\varepsilon}\right\|^2_{L^2_{x}},
\end{split}
\end{align}
such that for any $t \geq 0$, there holds
\begin{align}\label{macro estimate}
\begin{split}
&\varepsilon \frac{\d}{\d t}\mathcal{E}^N_{int}(t)\!+\! \sum_{1 \leq |\alpha| \leq N}\!\!
\left\|   \partial^\alpha_x \mathbf{P}f^\e\right\|^2_{L^2_{x,v}}
+\e^2 \sum_{1 \leq|\alpha| \leq N-1} \left\| \partial^\alpha_x E^{\varepsilon} \right\|^2_{L^2_{x}}+ \e^2 \sum_{2 \leq|\alpha| \leq N-1}\!\! \left\| \partial^\alpha_x B^{\varepsilon} \right\|^2_{L^2_{x}} \\
\lesssim \;& \frac{1}{\varepsilon^2} \sum_{|\alpha| \leq N}\left\| \partial^\alpha_x (\mathbf{I}-\mathbf{P})f^{\varepsilon}\right\|^2_{L^2_{x,v}(\nu)}
+\mathcal{E}_{N}(t) \mathcal{D}_{N}(t).
\end{split}
\end{align}
\end{lemma}

\begin{proof}
We prove this lemma by an analysis of the macroscopic equations and the local conservation laws. Firstly,
plugging (\ref{fdefenjie})  and  (\ref{defination:Pf}) into  the first equation in (\ref{rVPB}), we have
\newcommand{\W}{W}
\newcommand{\M}{M}
\newcommand{\T}{T}
\newcommand{\Q}{Q}
\bals
\bsp
&\pt_t \Big(a^\e+b^\e\cdot v+c^\e\frac{|v|^2-3}{2} \Big)\sqrt{\mu}+\frac{1}{\e}\sqrt{\mu}v\cdot \nabla_x \Big(a^\e+b^\e\cdot v+c^\e\frac{|v|^2-3}{2} \Big) -\sqrt{\mu}v\cdot E^\e\\
 = &-\pt_t(\mathbf{I}-\mathbf{P}) f^\e +g^\varepsilon(f^\e, E^\e, B^\e)
\esp
\eals
with $g^\varepsilon(f^\e, E^\e, B^\e)$ given by
\bal
\bsp\label{g define}
g^\varepsilon:= -\frac{1}{\e}v\cdot\nabla_x (\mathbf{I}-\mathbf{P}) f^\e-\frac{1}{\e^2}L f^\e
 - \frac{\e E^\e\cdot\nabla_v(\sqrt{\mu} f^\e)}{\sqrt{\mu}}
 -v\times B^\e\cdot\nabla_vf^\e
 +\frac{1}{\e}\Gamma(f^\e,f^\e).
 \esp
\eal
Fixing $(t,x)$   and comparing the coefficients on both sides,
we  thereby obtain the so-called macroscopic equations
\newcommand{\J}{\mathcal{G}} 
\beq\label{qkj:esti:diss:2}
 \left\{
\begin{array}{ll}
\sqrt{\mu}:  &\qquad         \displaystyle \pt_t\Big(a^\e-\frac{3}{2}c^\e\Big) = \J_1 , \,\quad\qquad\qquad     \\[2mm]
v_i\sqrt{\mu}:&\qquad   \displaystyle \pt_tb_i^\e+\frac{1}{\e}\pt_i\Big(a^\e-\frac{3}{2}c^\e\Big)-E^\e=\J _2 ^{i}  , \quad\qquad\qquad\\[2mm]
v_i^2\sqrt{\mu}:&\qquad      \displaystyle   \frac{1}{2}\pt_t c^\e+\frac{1}{\e}\pt_ib_i^\e= \J_3 ^{i}  , \quad\qquad\qquad \\[2mm]
v_iv_j\sqrt{\mu}(i\ne j):&\qquad    \displaystyle   \frac{1}{\e}\left(\pt_ib_j^\e+\pt_jb_i^\e\right)= \J _4 ^{i,j}, \,\,\qquad\qquad \\[2mm]
v_i|v|^2\sqrt{\mu}:&\qquad  \displaystyle   \frac{1}{\e}\pt_ic^\e= \J _{5}^i. \quad\qquad\qquad
\end{array}
\right.
\eeq
Here the right-hand terms $ \J_1,\J _2 ^{i},\J_3 ^{i},\J _4 ^{i,j}$ and $\J _5^i (i,j=1,2,3) $ are all of the form
\bals
\bsp
\big\langle-\pt_t(\mathbf{I}-\mathbf{P}) f^\e  + g^\varepsilon(f^\e, E^\e, B^\e), \zeta \big\rangle_{L^2_v},
\esp
\eals
where $\zeta$ is a (different) linear combination of the basis $\left[\sqrt{\mu},\,v_i\sqrt{\mu},\,v_iv_j\sqrt{\mu},\,v_i|v|^2\sqrt{\mu} \right]$.

The second set of equations we consider are the local conservation laws satisfied
by $(a^\e\!,\! b^\e\!, \!c^\e)$. To derive them, multiplying the first equation in (\ref{rVPB}) by the collision
invariant $(1, v, \frac{|v|^2-3}{2})\sqrt{\mu}$ and integrating  the resulting  identity over $\mathbb{R}^3_v$,  we derive by
using  (\ref{defination:Pf}) that $a^\e$, $b^\e$ and $c^\e$ obey the local macroscopic balance laws of mass, moment and energy
\beq\label{qkj:esti:f:L^2:12}
 \left\{
\begin{array}{ll}
 \displaystyle\pt_t a^\e+\frac{1}{ \e}\nabla_x\cdot b^\e=0,\\[2mm]
 \displaystyle \pt_t b^\e+\frac{1}{\e}\nabla_x(a^\e+c^\e)-E^\e=\mathcal{X}_{b^\e},\\[2mm]
 \displaystyle\pt_t c^\e+\frac{2}{3\e}\nabla_x\cdot b^\e
 =\mathcal{X}_{c^\e}.
\end{array}
\right.
\eeq
Here $\mathcal{X}_{b^\e}$ and $\mathcal{X}_{c^\e}$ are given by
\bal
\bsp\label{qkj:esti:f:L^2:12-1}
\mathcal{X}_{b^\e}&:=
\e a^\e E^\e+b^\e\times B^\e
 - \frac{1}{\e}\int_{\mathbb{R}^3}v\cdot \nabla_x (\mathbf{I}-\mathbf{P})f^\e v\sqrt{\mu} \d v,\\
\mathcal{X}_{c^\e}&:= \frac{2}{3 }\e E^\e\cdot b^\e
-\frac{1}{3\e}\int_{\bbR^3}v\cdot\nabla_x(\mathbf{I}-\mathbf{P})f^\e|v|^2\sqrt{\mu}\d v.
\esp
\eal

Next, applying  the analogous  arguments employed in \cite{GZW-2021}, which is based on the analysis of the macro fluid-type system \eqref{qkj:esti:diss:2} and the local conservation
laws \eqref{qkj:esti:f:L^2:12},
 we can  obtain that
 \begin{align} \label{macro estimate 1}
\begin{split}
&\varepsilon \frac{\d}{\d t} \mathcal{E}^{(\alpha)}_{a^\varepsilon,b^\varepsilon,c^\varepsilon}(t)
+\left\|\nabla_x (a^\varepsilon,b^\varepsilon,c^\varepsilon)\right\|^2_{L^2_x} \\
\lesssim \;& \frac{1}{\e^2}\left\|  \nabla_x (\mathbf{I}-\mathbf{P}) f^\varepsilon \right\|^2_{L^2_{x,v}(\nu)}
+\left\|  (\mathbf{I}-\mathbf{P}) f^\varepsilon \right\|_{L^2_{x,v}(\nu)}
\left\| \nabla_x E^\varepsilon \right\|_{L^2_{x}}
\\
\;&
+\varepsilon^2 \sum_{i=1}^3\sum_{j\neq i}\left\|\big \langle  (\J _3^i +\J _3^j +\J _4 ^{i,j}+\J _5^i +\mathcal{X}_{b^\e}), \zeta\big\rangle_{L^2_v} \right\|^2_{L^2_x}\quad\text {for}\quad |\a|=0,
\end{split}
\end{align}
and
 \begin{align} \label{macro estimate 2}
\begin{split}
&\varepsilon \frac{\d}{\d t} \mathcal{E}^{(\alpha)}_{a^\varepsilon,b^\varepsilon,c^\varepsilon}(t)
+\left\|\partial^\alpha_x \nabla_x (a^\varepsilon,b^\varepsilon,c^\varepsilon)\right\|^2_{L^2_x} \\
\lesssim \;& \frac{1}{\e^2}\left\|\partial^\alpha_x \nabla_x (\mathbf{I}-\mathbf{P}) f^\varepsilon \right\|^2_{L^2_{x,v}(\nu)}
+\left\|\partial^\alpha_x \nabla_x  (\mathbf{I}-\mathbf{P}) f^\varepsilon \right\|_{L^2_{x,v}(\nu)}
\left\|\partial^\alpha_x  E^\varepsilon \right\|_{L^2_{x}}
\\
\;&
+\varepsilon^2 \sum_{i=1}^3\sum_{j\neq i}\left\|\big\langle \partial^\alpha_x (\J _3^i +\J _3^j +\J _4 ^{i,j}+\J _5^i +\mathcal{X}_{b^\e}), \zeta\big\rangle_{L^2_v} \right\|^2_{L^2_x}
\quad\text{ for}\quad 1\leq|\a|\leq N-1.
\end{split}
\end{align}
The details are omitted for
 simplicity. Here, we only point out the representation of
$\mathcal{E}^\mathbf{(\a)}_{a^\e, b^\e, c^\e}(t)$ as
\bals
\bsp
\mathcal{E}^\mathbf{(\a)}_{a^\e, b^\e, c^\e}(t)=\;& \kappa_1 \big\langle-   \nabla_x\cdot \partial_x^\a b ,\, \partial_x^\a a\big\rangle_{L^2_x}
+  \sum_{i=1}^3\big\langle \partial_x^\a (\mathbf{I}-\mathbf{P}) f^\e ,\,\chi_{a^\e}^{i}\pt_i \partial_x^\a a^\e \big\rangle_{L^2_{x,v}}\\
\;&+ \sum_{i=1}^3\sum_{j\neq i}\big\langle \partial_x^\a (\mathbf{I}-\mathbf{P}) f^\e ,\,\chi_{b^\e}^{i,j}\pt_j \partial_x^\a b_i^\e \big\rangle_{L^2_{x,v}}
+ \sum_{i=1}^3\big\langle \partial_x^\a (\mathbf{I}-\mathbf{P}) f^\e ,\,\chi_{b^\e}^{i}\pt_i \partial_x^\a b_i^\e \big\rangle_{L^2_{x,v}}
\\
\;&
+ \sum_{i=1}^3\big\langle  \partial_x^\a(\mathbf{I}-\mathbf{P}) f^\e ,\,\chi_{c^\e}^{i}\pt_i \partial_x^\a c^\e \big\rangle_{L^2_{x,v}}
\quad\text{ for}\quad 0\leq|\a|\leq N-1,
\esp
\eals
where $\chi_{a^\e}^{i}$, $ \chi_{b^\e}^{i,j}$, $\chi_{b^\e}^{i}$ and $\chi_{c^\e}^{i}$ are dissimilar linear combinations of
$\left[\sqrt{\mu},v_i\sqrt{\mu}, v_iv_j\sqrt{\mu}, v_i|v|^2\sqrt{\mu} \right]$
and  $\kappa_1> 0$ is a small   constant.

To estimate   $E^\e$,     using the Maxwell system in (\ref{rVPB}), \eqref{qkj:esti:f:L^2:12} and  the Cauchy--Schwarz inequality with $\eta$, one can  deduce that, for any $1\leq|\a|\leq N-1$, there holds
\bal
\bsp\label{E1}
&\e^2\|\pt^\alpha_x E^\e\|_{L_x^2}^2=\e^2\big\langle \pt^\alpha_x E^\e, \pt^\alpha_x E^\e  \big\rangle_{L_x^2} \\
=\;&\e^2 \Big\langle \pt^\alpha_x E^\e, \pt^\alpha_x \Big[\pt_t b^\e+\frac{1}{\varepsilon}\nabla_x(a^\e+c^\e)-\mathcal{X}_{b^\e}\Big] \Big\rangle_{L_x^2}\\
=\;&\frac{\mathrm d}{\mathrm d t}\!\e^2\big\langle \pt^\alpha_x E^\e, \pt^\alpha_x b^\e  \big\rangle_{L_x^2}\! -\!\e^2\big\langle \pt^\alpha_x \pt_t E^\e ,\pt^\alpha_x b^\e \big\rangle_{L_x^2}
+\eta\e^2\|\pt^\alpha_x E^\e\|_{L_x^2}^2+C_{\eta}\|\pt^\alpha_x \nabla_x (a^\e+c^\e) \|_{L_x^2}^2\\
 \;&+C_{\eta}
\e^2\|\pt^\alpha_x \mathcal{X}_{b^\e}\|_{L_x^2}^2.
\esp
\eal
Further, for the second term on the right-hand side of  (\ref{E1}), applying (\ref{rVPB})$_2$ gives rise to
\begin{align}
\begin{split}\label{E2}
\e^2\big\langle \pt^\alpha_x \pt_t E^\e , \pt^\alpha_x b^\e \big\rangle_{L_x^2}
&=\e^2\big\langle \pt^\alpha_x (\nabla_x\times B^\e-b^\e), \pt^\alpha_x b^\e  \big\rangle_{L_x^2} \\
&=\e^2\big\langle \pt^\alpha_x \nabla_x\times B^\e,  \pt^\alpha_x b^\e \big\rangle_{L_x^2}
-\e^2\|\pt^\alpha_x b^\e\|_{L_x^2}^2\\
&=\e^2\big\langle \pt^\alpha_x B^\e,  \pt^\alpha_x\nabla_x\times  b^\e \big\rangle_{L_x^2}
-\e^2\|\pt^\alpha_x b^\e\|_{L_x^2}^2.
\end{split}
\end{align}
Hence, putting \eqref{E2} into \eqref{E1} and choosing $\eta$ small enough, we find that
\begin{align}
\begin{split}\label{E4}
&\e\frac{\mathrm d}{\mathrm d t}{\mathcal{E}}_{E^\e}^{(\a)}(t)+\e^2\|\pt^\alpha_x E^\e\|_{L_x^2}^2\\
\lesssim\;&\e^4\|\pt^\alpha_x \nabla_x\times B^\e\|_{L_x^2}^2+\|\pt^\alpha_x \mathbf{ P}f^\e \|_{L_x^2}^2
+\|\pt^\alpha_x \nabla_x\mathbf{ P}f^\e \|_{L_x^2}^2+\e^2\|\pt^\alpha_x \mathcal{X}_{b^\e}\|_{L_x^2}^2 \quad\text{for} \quad|\a|=1,
\end{split}
\end{align}
and
\begin{align}
\begin{split}\label{E5}
\e\frac{\mathrm d}{\mathrm d t}{\mathcal{E}}_{E^\e}^{(\a)}(t)+\e^2\|\pt^\alpha_x E^\e\|_{L_x^2}^2
\lesssim\;&\e^4\|\pt^\alpha_x  B^\e\|_{L_x^2}^2
+\|\pt^\alpha_x \nabla_x\mathbf{ P}f^\e \|_{L_x^2}^2+\e^2\|\pt^\alpha_x \mathcal{X}_{b^\e}\|_{L_x^2}^2
\end{split}
\end{align}
for   $2\leq|\a|\leq N-1$, where ${\mathcal{E}}_{E^\e}^{(\a)}(t)$ is defined as
$${\mathcal{E}}_{E^\e}^{(\a)}(t):=-\e\big\langle \pt^\alpha_x E^\e,  \pt^\alpha_x b^\e \big\rangle_{L_x^2}\quad\text{ for}\quad 1\leq|\a|\leq N-1. $$

To obtain the dissipation of $B^\e$,  by   using the Maxwell system in (\ref{rVPB}) and  the Cauchy--Schwarz inequality with $\eta$, we have
\bals
\bsp
&\e^2\|\pt^\alpha_x \nabla_x\times B^\e\|_{L_x^2}^2\\
=\;&\e^2 \big\langle  \pt^\alpha_x\nabla_x\times B^\e, \pt^\alpha_x(\pt_t E^\e+b^\e)  \big\rangle_{L_x^2} \\
=\;&\frac{\mathrm{d}}{\mathrm{d}t}\e^2 \big\langle  \pt^\alpha_x\nabla_x\times B^\e, \pt^\alpha_x E^\e  \big\rangle_{L_x^2}
-\e^2\big\langle  \pt^\alpha_x\nabla_x\times \pt_t B^\e, \pt^\alpha_x E^\e \big\rangle_{L_x^2}
+\e^2 \big\langle  \pt^\alpha_x\nabla_x\times B^\e ~,~ \pt^\alpha_x b^\e  \big\rangle_{L_x^2}\\
\leq\;&\frac{\mathrm{d}}{\mathrm{d}t}\e^2 \big\langle  \pt^\alpha_x\nabla_x\times B^\e, \pt^\alpha_x E^\e  \big\rangle_{L_x^2}
+\e^2\|\pt^{{\alpha}}_x \nabla_x E^\e\|_{L_x^2}^2+\eta\e^2\|\pt^\alpha_x \nabla_x\times B^\e\|_{L_x^2}^2+ C_{\eta}\e^2\|\pt^\alpha_x b^\e\|_{L_x^2}^2.
\esp
\eals
Thereby, we derive that for any $1\leq |\a|\leq N-2$, there holds
\begin{equation}\label{B3}
\begin{split}
  &\e\frac{\mathrm{d}}{\mathrm{d}t}{\mathcal{E}}_{B^\e}^{(\a)}(t)+\e^2\|\pt^\alpha_x \nabla_x\times B^\e\|_{L_x^2}^2
  \lesssim\e^2 \|\pt^\alpha_x \nabla_x \times E^\e\|_{L_x^2}^2
  + \e^2\|\pt^\alpha_x b^\e\|_{L_x^2}^2,
\end{split}
\end{equation}
where ${\mathcal{E}}_{B^\e}^{(\a)}(t)$ is defined as
$${\mathcal{E}}_{B^\e}^{(\a)}(t):= -\e \big\langle  \pt^\alpha_x\nabla_x\times B^\e, \pt^\alpha_x E^\e  \big\rangle_{L_x^2}\quad\text{ for}\quad 0\leq|\a|\leq N-2. $$

Finally, set
\begin{align}\label{macro energy2}
\mathcal{E}^{N}_{int}(t):= \sum_{0\leq|\alpha| \leq N-1}\mathcal{E}^{(\alpha)}_{a^\varepsilon, b^\varepsilon,c^\varepsilon}(t)+\kappa_2\sum_{1\leq|\alpha| \leq N-1} \mathcal{E} ^{(\alpha)}_{E^\varepsilon}(t)
+\kappa_3\sum_{1\leq|\alpha| \leq N-2} {\mathcal{E}}_{B^\e}^{(\a)}(t)
\end{align}
with  $ 0 <\kappa_3< \kappa_2 \ll 1$.
It follows from \eqref{macro estimate 1}, \eqref{macro estimate 2},
\eqref{E4}, \eqref{E5} and \eqref{B3} that
\bal
&\varepsilon \frac{\d}{\d t} \mathcal{E}^{N}_{int}(t)
+ \!\!\!\sum_{0 \leq |\alpha| \leq N-1}\!\!\!
\left\|   \partial^\alpha_x\nabla_x \mathbf{P}f^\e\right\|^2_{L^2_{x,v}}+\e^2\!\! \sum_{1 \leq|\alpha| \leq N-1}\!\! \left\| \partial^\alpha_x E^{\varepsilon} \right\|^2_{L^2_{x}}+ \e^2 \!\! \sum_{2 \leq|\alpha| \leq N-1}\!\! \left\| \partial^\alpha_x B^{\varepsilon} \right\|^2_{L^2_{x}}\nonumber \\
\lesssim \;&\frac{1}{\varepsilon^2} \sum_{0\leq|\alpha| \leq N}\left\| \partial^\alpha_x (\mathbf{I}-\mathbf{P})f^{\varepsilon}\right\|^2_{L^2_{x,v}(\nu)}
+\sum_{0\leq|\alpha| \leq N-1}\e^2\|\pt^\alpha_x \mathcal{X}_{b^\e}\|_{L_{x,v}^2}^2\label{macro energy1}\\
\;&
  +\varepsilon^2\!\!\sum_{0\leq|\alpha| \leq N-1} \sum_{i=1}^3\sum_{j\neq i}\left\|  \big\langle \partial^\alpha_x( \J _3^i +\J _3^j +\J _4 ^{i,j}+\J _5^i ),\, \zeta\big\rangle_{L^2_{v}}\right\|^2_{L^2_{x}}.\nonumber
\eal
By referring back to \eqref{g define}, \eqref{qkj:esti:f:L^2:12-1} for $g^\varepsilon$ and $\mathcal{X}_{b^\e}$, and employing \eqref{hard gamma}, we obtain
\begin{align}
\bsp\label{g estimate}
&\sum_{0\leq|\alpha| \leq N-1}\!\!\e^2\|\pt^\alpha_x \mathcal{X}_{b^\e}\|_{L_{x,v}^2}^2
  +\varepsilon^2\sum_{0\leq|\alpha| \leq N-1} \sum_{i=1}^3\sum_{j\neq i}\left\|  \big\langle \partial^\alpha_x( \J _3^i +\J _3^j +\J _4 ^{i,j}+\J _5^i ),\, \zeta\big\rangle_{L^2_{v}}\right\|^2_{L^2_{x}}\\
\lesssim\;&\sum_{|\alpha| \leq N-1}\left\| \partial^\alpha_x\nabla_x (\mathbf{I}-\mathbf{P})f^{\varepsilon}\right\|^2_{L^2_{x,v}(\nu)}+
\frac{1}{\varepsilon^2} \sum_{ |\alpha| \leq N-1}\left\| \partial^\alpha_x (\mathbf{I}-\mathbf{P})f^{\varepsilon}\right\|^2_{L^2_{x,v}(\nu)}+ \mathcal{E}_{N}(t)\mathcal{D}_{N}(t).
\esp
\end{align}
Hence, gathering \eqref{macro energy1} and \eqref{g estimate} leads to the desired result \eqref{macro estimate}.
This completes the proof of Lemma \ref{macroscopic estimate}.
\end{proof}
Next, we   perform the following
 $H^N_{x}L^2_{v}$  energy estimate for the VMB system   \eqref{rVPB}.
\begin{lemma}\label{energy estimate without weight}Let $ (f^\e, E^\e, B^\e) $ be the solution of the VMB system  \eqref{rVPB} defined on $ [0, T] \times \mathbb{R}^3 \times \mathbb{R}^3$ and  the \emph{a priori} assumption  \eqref{priori assumption} hold true for $\delta_0$ small enough. Then, there holds
\begin{align}\label{estimate without weight}
\begin{split}
&\frac{\d}{\d t}\left[\left\| f^\varepsilon\right\|^2_{H^N_xL^2_v}+\left\| (E^\varepsilon, \! B^\varepsilon)\right\|^2_{H^N_x}
\right]
 \!+ \!\frac{\sigma_0}{\varepsilon^2}
\left\| (\mathbf{I} \!- \!\mathbf{P})f^\varepsilon\right\|^2_{H^N_xL^2_v(\nu)}
+\e\frac{\d}{\d t}  \int_{\mathbb{R}^3}|b^\e|^2(a^\e+c^\e)\dd x
\\
\lesssim\;&\delta_0
\mathcal{D}_{N}(t)
+\delta_0\left(\left\|\langle v\rangle \nabla_v(\mathbf{I}-\mathbf{P})f^\varepsilon\right\|_{H^{N-2}_xL^2_v}^2
+\left\|\langle v\rangle (\mathbf{I}-\mathbf{P})f^\varepsilon\right\|_{H^{N-1}_xL^2_v}^2\right)\\
 \;&
 +\frac{\delta_0}{(1+t)^{1+\vartheta}}\Big(\e\left\|\langle v\rangle \nabla_x^N(\mathbf{I}-\mathbf{P})f^\varepsilon\right\|_{L^2_{x,v}}^2
 +\frac{1}{\e}\left\|\langle v\rangle \nabla_v \nabla_x^{N-1}(\mathbf{I}-\mathbf{P})f^\varepsilon\right\|_{L^2_{x,v}}^2\Big).
\end{split}
\end{align}
\end{lemma}
\begin{proof}
The proof  of (\ref{estimate without weight}) is divided into
two steps as follows.

\emph{Step 1. $L_{x,v}^2$-estimate of $f^\e$.}\;
Taking the  $L^2(\bbR^3_{x}\times\mathbb{R}^3_{v})$ inner product of the first equation of \eqref{rVPB} with $f^\e$ and then employing (\ref{spectL}), we derive
\bal
\bsp\label{without weight-basic}
&\frac{1}{2}\frac{\d}{\d t}\left\| f^\varepsilon\right\|^2_{L^2_{x,v} } +\underbrace{\Big\langle-E^\e\cdot v \mu^{1/2}, f^\varepsilon\Big\rangle_{L^2_{x,v} } }_{\mathcal{X}_1}
+\frac{\sigma_0}{\varepsilon^2}\left\| (\mathbf{I}-\mathbf{P})f^\varepsilon\right\|^2_{L^2_{x,v}(\nu)}  \\
\lesssim\;& \underbrace{\Big\langle\frac{\e}{2}( v\cdot E^\e f^\varepsilon),  f^\varepsilon\Big\rangle_{L^2_{x,v} }}_{\mathcal{X}_2}
+\underbrace{\Big\langle\frac{1}{\varepsilon}
 \Gamma(f^\varepsilon, f^\varepsilon), f^\varepsilon\Big\rangle_{L^2_{x,v} }}_{\mathcal{X}_3} .
\esp
\eal
For the term ${\mathcal{X}_1} $, recalling the second and third equation of \eqref{rVPB}, we have
\begin{align}\nonumber
\mathcal{X}_{1}=\Big\langle E^\e, \pt_t E^\varepsilon- \nabla_x\times B^\varepsilon \Big\rangle_{L^2_{x} }
=\frac{1}{2}\frac{\d}{\d t}\left(\left\| E^\varepsilon\right\|^2_{L^2_x}+\left\| B^\varepsilon\right\|^2_{L^2_x}\right)
.
\end{align}
By the  macro-micro decomposition \eqref{fdefenjie}, we derive
\bals
\bsp
\mathcal{X}_2=
 \;&\underbrace{ \frac{\e}{2} \langle v \cdot E^\e \mathbf{P}f^\e,\mathbf{P}f^\e \rangle_{L^2_{x,v}} }_{\mathcal{X}_{2,1}}
  +\underbrace{\e \langle v \cdot E^\e \mathbf{P}f^\e,(\mathbf{I}\!-\! \mathbf{P})f^\e \rangle_{L^2_{x,v}}}_{\mathcal{X}_{2,2}}+\underbrace{\frac{\e}{2} \langle v \cdot E^\e (\mathbf{I}\!-\!\mathbf{P})f^\e,(\mathbf{I}-\mathbf{P})f^\e \rangle_{L^2_{x,v}} }_{\mathcal{X}_{2,3}}.
\esp
\eals
Then we derive from   \eqref{defination:Pf}, (\ref{qkj:esti:f:L^2:12}), \eqref{qkj:esti:f:L^2:12-1}  and the H\"{o}lder inequality that
\bals
\bsp
\mathcal{X}_{2,1}
=\;&-\frac{\e}{2}\iint_{\mathbb{R}^3\times\mathbb{R}^3}    v \cdot
E^\e \Big[\Big(a^\e+b^\e\cdot v+c^\e\frac{|v|^2-3}{2}\Big)\sqrt{\mu}\Big]^2  \d x\d v \\
=\;&-\e\int_{\bbR^3}    (a^\e+c^\e)b^\e\cdot E^\e    \d x \\
\le\;&\int_{\bbR^3}  \e  (a^\e+c^\e)b^\e\cdot\pt_t b^\e\d +\int_{\bbR^3}\left| (a^\e+c^\e)b^\e\cdot\left[\nabla_x(a^\e+c^\e)+\e\mathcal{X}_{b^\e}\right]\right|\d x\\
 \le\;&-\frac{\e}{2}\frac{\d}{\d t}\int_{\bbR^3}   (a^\e+c^\e)|b^\e|^2\d x+C\int_{\bbR^3}  \e \Big(|\mathcal{X}_{c^\e}|+\frac{1}{ \e}|\nabla_x\cdot b^\e|\Big)  |{b^\e}|^2\d x\\
\;& +\int_{\bbR^3}\left|  (a^\e+c^\e)b^\e\cdot\nabla_x(a^\e+c^\e)\right|\d x
+\int_{\bbR^3} \left| \e (a^\e+c^\e)b^\e\cdot\mathcal{X}_{b^\e}\right|\d x\\
 \leq\;& -\frac{\e}{2}\frac{\d}{\d t}\int_{\bbR^3}    (a^\e+c^\e)|b^\e|^2\d x
 +C\big[\mathcal{E}_N(t)\big]^{\frac{1}{2}} \mathcal{D}_N(t)+C \mathcal{E}_N(t) \mathcal{D}_N(t),
\esp
\eals
where the fourth inequality   above is derived from the integrating by parts with respect to $t$.
Further, thanks to the H\"{o}lder inequality, the Sobolev embedding inequality and the Cauchy--Schwarz inequality with $\eta$, we find
\bals
\bsp
\;&\mathcal{X}_{2,2}+\mathcal{X}_{2,3}\\
\lesssim\;&\e\|E^\e\|_{L_x^3} \|\mathbf{P}f^\e\|_{L^6_xL^2_v}\|(\mathbf{I}-\mathbf{P})f^\e\|_{L_{x,v}^2(\nu)}+\e\|E^\e\|_{L^\infty_x}
\|\langle v\rangle (\mathbf{I}-\mathbf{P})f^\e\|_{L^2_{x,v}}
\| (\mathbf{I}-\mathbf{P})f^\e\|_{L^2_{x,v}}\\
\lesssim\;&\e\big[\mathcal{E}_N(t)\big]^{\frac{1}{2}}\mathcal{D}_N(t)
+\e^4\|\nabla_xE^\e\|_{H^1_x}^2
\|\langle v\rangle (\mathbf{I}-\mathbf{P})f^\e\|_{L^2_{x,v}}^2
+\frac{\eta}{\e^2}\| (\mathbf{I}-\mathbf{P})f^\e\|_{L^2_{x,v}}^2,
\esp
\eals
where here and below $\1>0$ is a sufficiently small constant.
Here we have used  the Nirenberg inequality
\begin{equation}\label{NGinequality}
\| f \|_{L_x^\infty}\lesssim \|\nabla_x f \|_{L_x^2}^{\frac{1}{2}}\|\nabla_x^2 f \|_{L_x^2}^{\frac{1}{2}}\lesssim \|\nabla_x f \|_{H_x^1}.
\end{equation}
Therefore, gathering the above bounds of $\mathcal{X}_{2,1}$,  $\mathcal{X}_{2,2}$ and  $\mathcal{X}_{2,3}$  together yields
\bals
\bsp
\mathcal{X}_{2}\leq\;&-\frac{\e}{2}\frac{\d}{\d t}\int_{\bbR^3}    (a^\e+c^\e)|b^\e|^2\d x
 +C\big[\mathcal{E}_N(t)\big]^{\frac{1}{2}} \mathcal{D}_N(t)+C \mathcal{E}_N(t) \mathcal{D}_N(t)
 \\
\;&
+C\e^4\|\nabla_xE^\e\|_{H^1_x}^2\|\langle v\rangle (\mathbf{I}-\mathbf{P})f^\e\|_{L^2_{x,v}}^2
+\frac{C\eta}{\e^2}\| (\mathbf{I}-\mathbf{P})f^\e\|_{L^2_{x,v}(\nu)}^2.
\esp
\eals
For the term $\mathcal{X}_3$,
the collision symmetry of $\Gamma $ and \eqref{hard gamma} implies
\bals
\bsp
\mathcal{X}_3=\Big\langle\frac{1}{\e}\Gamma(f^\e,f^\e), (\mathbf{I}-\mathbf{P})f^\e \Big\rangle_{L_{x,v}^2}
\lesssim &\;
\frac{1}{\e}
\|f^\e\|_{L^\infty_x L^2_v(\nu)}
\|f^\e\|_{L^2_{x,v}}\|(\mathbf{I}-\mathbf{P})f^\e\|_{L^2_{x,v}(\nu)}\\
\lesssim &\;
\big[\mathcal{E}_N(t)\big]^{\frac{1}{2}}\mathcal{D}_N(t).
\esp
\eals

Consequently, putting   the above estimates of $\mathcal{X}_1$, $\mathcal{X}_2$ and $\mathcal{X}_3$ into \eqref{without weight-basic} and  choosing $\eta$ small enough, we  conclude
\bal
\bsp\label{without weight-basic-result}
\;&\frac{\d}{\d t}\left[\left\| f^\varepsilon\right\|^2_{L^2_{x,v} } +\left\| E^\varepsilon\right\|^2_{L^2_x}+\left\| B^\varepsilon\right\|^2_{L^2_x}\right]+\e\frac{\d}{\d t}\int_{\bbR^3}    (a^\e\!+\!c^\e)|b^\e|^2\d x+\frac{\sigma_0}{\varepsilon^2}\left\| (\mathbf{I}\!-\!\mathbf{P})f^\varepsilon\right\|^2_{L^2_{x,v}(\nu)}\\
\lesssim \;&\big[\mathcal{E}_N(t)\big]^{\frac{1}{2}} \mathcal{D}_N(t)+ \mathcal{E}_N(t) \mathcal{D}_N(t)
+\e^4\|\nabla_xE^\e\|_{H^1_x}^2
\|\langle v\rangle (\mathbf{I}-\mathbf{P})f^\e\|_{L^2_{x,v}}^2.
\esp
\eal

\emph{Step 2. $L_{x,v}^2$-estimate of the pure space derivative  $\partial_x^\a f^\e$ $(1\leq|\a|\leq N).$\;}\;
Applying $\partial^\alpha_x$ with $1\leq | \alpha | \leq N $ to  the first equation of \eqref{rVPB}  and taking the inner product with $  \partial^\alpha_xf^\varepsilon $ over $\mathbb{R}_x^3$ $\times$  $\mathbb{R}_v^3$, we obtain by employing the coercivity of $ L$ in (\ref{spectL}) that
\bal
\bsp\label{without weight}
&\frac{1}{2}\frac{\d}{\d t}\left[\left\| \partial^\a_xf^\varepsilon\right\|^2_{L^2_{x,v} } +
\left\| \partial^\alpha_xE^\varepsilon\right\|^2_{L^2_x}+\left\| \partial^\alpha_xB^\varepsilon\right\|^2_{L^2_x}\right]
+\frac{\sigma_0}{\varepsilon^2}\left\|\partial^\alpha_x (\mathbf{I}-\mathbf{P})f^\varepsilon\right\|^2_{L^2_{x,v}(\nu)}  \\
\lesssim\;& \underbrace{ \frac{\e}{2}\sum_{0\leq|\a'|\leq|\a|}\!\!\! \Big\langle v \cdot\partial^{\alpha'}_x E^\e \partial^{\alpha-\alpha'}_x f^\e,\partial^\alpha_x f^\e \Big\rangle_{L^2_{x,v}}}_{\mathcal{X}_4}
+\underbrace{\sum_{1\leq|\a'|\leq|\a|}\!\!\! \Big\langle\e\partial^{\alpha'}_x E^{\e} \cdot \nabla_v\partial^{\alpha-\alpha'}_x f^\varepsilon,  \partial^\alpha_x f^\varepsilon\Big\rangle_{L^2_{x,v} }}_{\mathcal{X}_5}
  \\
&+\underbrace{\sum_{1\leq|\a'|\leq|\a|} \Big\langle v\times \partial^{\alpha'}_xB^\e \cdot \nabla_v \partial^{\alpha-\alpha'}_x f^\varepsilon,  \partial^\alpha_x f^\varepsilon\Big\rangle_{L^2_{x,v} }}_{\mathcal{X}_6}+\underbrace{\Big\langle\frac{1}{\varepsilon}
\partial^\alpha_x \Gamma(f^\varepsilon, f^\varepsilon),\partial^\alpha_x f^\varepsilon\Big\rangle_{L^2_{x,v} }}_{\mathcal{X}_7} .
\esp
\eal
For the term ${\mathcal{X}_4} $, recalling the macro-micro decomposition \eqref{fdefenjie}, we have
\bals
\bsp
\mathcal{X}_4
 =\;&\underbrace{ \frac{\e}{2}\sum_{0\leq|\a'|\leq|\a|} \Big\langle v \cdot\partial^{\alpha'}_x E^\e \partial^{\alpha-\alpha'}_x \mathbf{P}f^\e,\partial^\alpha_x f^\e \Big\rangle_{L^2_{x,v}} }_{\mathcal{X}_{4,1}}\\
 \;&+
  \underbrace{\frac{\e}{2}\sum_{0\leq|\a'|\leq|\a|} \Big\langle v \cdot\partial^{\alpha'}_x E^\e \partial^{\alpha-\alpha'}_x (\mathbf{I}-\mathbf{P})f^\e,\partial^\alpha_x \mathbf{P}f^\e \Big\rangle_{L^2_{x,v}}}_{\mathcal{X}_{4,2}}\\
              \;&+\underbrace{\frac{\e}{2}\sum_{0\leq|\a'|\leq|\a|} \Big\langle v \cdot\partial^{\alpha'}_x E^\e \partial^{\alpha-\alpha'}_x (\mathbf{I}-\mathbf{P})f^\e,\partial^\alpha_x (\mathbf{I}-\mathbf{P})f^\e \Big\rangle_{L^2_{x,v}}}_{\mathcal{X}_{4,3}}.
\esp
\eals
Utilizing the H\"{o}lder inequality, \eqref{NGinequality} and the Sobolev embedding $H^1(\mathbb{R}^3)\hookrightarrow L^3(\mathbb{R}^3)$, we arrive at
\bals
\bsp
\mathcal{X}_{4,1}\lesssim\;&
\e\left(\| E^\e\|_{L^\infty_x}
\| \partial^{\alpha}_x  \mathbf{P} f^\e\|_{L^2_{x,v}}
+\|\partial^{\alpha}_x E^\e\|_{L^2_x}
\|  \mathbf{P}f^\e\|_{L^\infty_{x}L^2_{v}}\right)
\| \partial^{\alpha}_xf^\e\|_{L^2_{x,v}}\\
\;&
+
\e\sum_{1\leq|\a'|\leq |\a|-1} \|\partial^{\alpha'}_x E^\e\|_{L^3_x}
\| \partial^{\alpha-\alpha'}_x  \mathbf{P} f^\e\|_{L^6_{x}L^2_{v}}
\| \partial^{\alpha}_xf^\e\|_{L^2_{x,v}}\\
\lesssim\;&
\e\| \nabla_xE^\e\|_{H^{N-1}_x}
\|  \nabla_x \mathbf{P}f^\e\|_{H^{N-1}_{x}L^2_{v}}
\| \partial^{\alpha}_xf^\e\|_{L^2_{x,v}}\\
\lesssim\;&\big[\mathcal{E}_N(t)\big]^{\frac{1}{2}} \mathcal{D}_N(t).
\esp
\eals
Applying the analogous argument of the estimate of $\mathcal{X}_{4,1}$, we have
\bals
\bsp
\mathcal{X}_{4,2}
\lesssim
\e\| \nabla_xE^\e\|_{H^{N-1}_x}
\|  \nabla_x (\mathbf{I}-\mathbf{P})f^\e\|_{H^{N-1}_{x}L^2_{v}(\nu)}
\| \partial^{\alpha}_x\mathbf{P}f^\e\|_{L^2_{x,v}}
\lesssim\big[\mathcal{E}_N(t)\big]^{\frac{1}{2}} \mathcal{D}_N(t).
\esp
\eals
For the term $\mathcal{X}_{4,3}$, we deduce from
the H\"{o}lder inequality, \eqref{NGinequality}, the Sobolev embeddings $H^1(\mathbb{R}^3)$$\hookrightarrow L^3(\mathbb{R}^3)$ and $H^1(\mathbb{R}^3)\hookrightarrow L^6(\mathbb{R}^3)$ that
\bals
\bsp
\mathcal{X}_{4,3}\lesssim\;&
\e
\Big[\sum_{0\leq|\a'|\leq 1}  \|\partial^{\alpha'}_x E^\e\|_{L^\infty_x}
\|\langle v\rangle \partial^{\alpha-\alpha'}_x  (\mathbf{I}-\mathbf{P})f^\e\|_{L^2_{x,v}}
+\|\partial^{\alpha}_x E^\e\|_{L^2_x}
\|\langle v\rangle  (\mathbf{I}-\mathbf{P})f^\e\|_{L^\infty_{x}L^2_{v}}
\\
\;&\quad+\sum_{2\leq|\a'|\leq |\a|-1} \|\partial^{\alpha'}_x E^\e\|_{L^6_x}
\|\langle v\rangle \partial^{\alpha-\alpha'}_x  (\mathbf{I}-\mathbf{P})f^\e\|_{L^3_{x}L^2_{v}}\Big]
\| \partial^{\alpha}_x(\mathbf{I}-\mathbf{P})f^\e\|_{L^2_{x,v}}\\
\lesssim\;& \e^4\|\nabla_xE^\e\|_{H^1_x}^2
\|\langle v\rangle \partial_x^\a  (\mathbf{I}-\mathbf{P})f^\e\|_{L^2_{x,v}}^2
+\mathcal{E}_N(t)
\|\langle v\rangle   (\mathbf{I}-\mathbf{P})f^\e\|_{H^{N-1}_{x}L^2_{v}}^2
\\
\;&+\frac{\eta}{\e^2}
\| \partial^{\alpha}_x(\mathbf{I}-\mathbf{P})f^\e\|_{L^2_{x,v}}^2.
\esp
\eals
Via the summarization of the bounds of $\mathcal{X}_{4,i}(i=1,2,3)$, we find
\bals
\bsp
\mathcal{X}_{4}\lesssim\;&\big[\mathcal{E}_N(t)\big]^{\frac{1}{2}} \mathcal{D}_N(t)+
\e^4\|\nabla_xE^\e\|_{H^1_x}^2
\|\langle v\rangle \partial_x^\a  (\mathbf{I}-\mathbf{P})f^\e\|_{L^2_{x,v}}^2
+\mathcal{E}_N(t)
\|\langle v\rangle   (\mathbf{I}-\mathbf{P})f^\e\|_{H^{N-1}_{x}L^2_{v}}^2\\
\;&
+\frac{\eta}{\e^2}
\| \partial^{\alpha}_x(\mathbf{I}-\mathbf{P})f^\e\|_{L^2_{x,v}}^2.
\esp
\eals

For the term ${\mathcal{X}_5} $, utilizing \eqref{fdefenjie}, we have
\bals
\bsp
\mathcal{X}_5
=\;&\underbrace{\!\!\!\sum_{1\leq|\a'|\leq|\a|}\!\!\! \Big\langle\e\partial^{\alpha'}_x E^{\e} \cdot \nabla_v\partial^{\alpha-\alpha'}_x f^\varepsilon,  \partial^\alpha_x (\mathbf{I}\!-\!\mathbf{P}) f^\varepsilon \Big\rangle_{L^2_{x,v} }}_{\mathcal{X}_{5,1}}
\underbrace{ \!- \! \!\!\!\sum_{1\leq|\a'|\leq|\a|}\!\!\!\left\langle\e\partial^{\alpha'}_x E^{\e}  \partial^{\alpha-\alpha'}_x f^\varepsilon,  \nabla_v\partial^\alpha_x \mathbf{P} f^\varepsilon\right\rangle_{L^2_{x,v} }}_{\mathcal{X}_{5,2}}.
\esp
\eals
It is derived from  the H\"{o}lder inequality, \eqref{NGinequality}, the Sobolev embeddings $H^1(\mathbb{R}^3)\hookrightarrow L^3(\mathbb{R}^3)$ and $H^1(\mathbb{R}^3)\hookrightarrow L^6(\mathbb{R}^3)$,  \eqref{fdefenjie} and the Cauchy--Schwarz inequality with $\eta$  that
\bals
\bsp
\mathcal{X}_{5,1}
\lesssim\;&\e\Big[\sum_{1\leq|\a'|\leq2}\|\partial^{\alpha'}_x E^{\e}\|_{L^\infty_{x}} \cdot \|\nabla_v\partial^{\alpha-\alpha'}_x f^\varepsilon\|_{L^2_{x,v}}+\|\partial_x^\a E^{\e}\|_{L^2_{x}} \cdot \|\nabla_vf^\varepsilon\|_{L^\infty_{x}L^2_{v}}\\
\;&
\quad+
\sum_{3\leq|\a'|\leq|\a|-1} \|\partial^{\alpha'}_xE^{\e}\|_{L^6_{x}} \cdot \|\nabla_v\partial^{\alpha-\a'}_x f^\varepsilon\|_{L^3_{x}L^2_{v}}
\Big]\| \partial^{\alpha}_x(\mathbf{I}-\mathbf{P})f^\e\|_{L^2_{x,v}}\\
\lesssim\;&\sum_{|\a'|=1}\e\|\nabla_x E^{\e}\|_{H^2_{x}}
\left(
\|\nabla_v\partial^{\alpha-\a'}_x(\mathbf{I}-\mathbf{P}) f^\varepsilon\|_{L^2_{x,v}}+
\|\nabla_v\partial^{\alpha-\a'}_x\mathbf{P}f^\varepsilon\|_{L^2_{x,v}}\right)
\| \partial^{\alpha}_x(\mathbf{I}-\mathbf{P})f^\e\|_{L^2_{x,v}}\\
\;&
+\e\| E^{\e}\|_{H^N_{x}}\left(
\|\nabla_v\nabla_x(\mathbf{I}-\mathbf{P}) f^\varepsilon\|_{H^{N-3}_{x}L^2_v}
+\|\nabla_v\nabla_x\mathbf{P} f^\varepsilon\|_{H^{N-3}_{x}L^2_v}
\right)
\| \partial^{\alpha}_x(\mathbf{I}-\mathbf{P}) f^\e\|_{L^2_{x,v}}\\
\lesssim\;&\e^4\|\nabla_x E^{\e}\|_{H^2_{x}}^2\sum_{|\a'|=1}\|\nabla_v\partial^{\alpha-\a'}_x(\mathbf{I}-\mathbf{P}) f^\varepsilon\|_{L^2_{x,v}}^2
+\e^4\| E^{\e}\|_{H^N_{x}}^2
\|\nabla_v(\mathbf{I}-\mathbf{P}) f^\varepsilon\|_{H^{N-2}_{x}L^2_v}^2\\
\;&+\frac{\eta}{\e^2}
\| \partial^{\alpha}_x(\mathbf{I}-\mathbf{P})f^\e\|_{L^2_{x,v}}^2+\big[\mathcal{E}_N(t)\big]^{\frac{1}{2}} \mathcal{D}_N(t).
\esp
\eals
Similarly,  the quantity $\mathcal{X}_{5,2}$ is controlled by
\bals
\bsp
\mathcal{X}_{5,2}
\lesssim\;&\e\Big[\sum_{1\leq|\a'|\leq|\a|-1}\|\partial^{\alpha'}_x E^{\e}\|_{L^3_{x}} \cdot \|\partial^{\alpha-\alpha'}_x f^\varepsilon\|_{L^6_{x}L^2_{v}}
+\|\partial_x^\a E^{\e}\|_{L^2_{x}} \cdot \|f^\varepsilon\|_{L^\infty_{x}L^2_{v}}\Big]\| \partial^{\alpha}_x\mathbf{P}f^\e\|_{L^2_{x,v}}\\
\lesssim\;&\e\| E^{\e}\|_{H^N_{x}}
\|\nabla_x f^\varepsilon\|_{H^{N-1}_{x}L^{2}_{v}}\| \partial^{\alpha}_x\mathbf{P}f^\e\|_{L^2_{x,v}}
\\
\lesssim\;&\big[\mathcal{E}_N(t)\big]^{\frac{1}{2}} \mathcal{D}_N(t).
\esp
\eals
Thus, collecting the above bounds of $\mathcal{X}_{5,1}$ and $\mathcal{X}_{5,2}$   shows us  that
\bals
\bsp
\mathcal{X}_{5}
\lesssim\;&\e^4\|\nabla_x E^{\e}\|_{H^2_{x}}^2\sum_{|\a'|=1}\|\nabla_v\partial^{\alpha-\a'}_x(\mathbf{I}-\mathbf{P}) f^\varepsilon\|_{L^2_{x,v}}^2
+\e^4\| E^{\e}\|_{H^N_{x}}^2
\|\nabla_v(\mathbf{I}-\mathbf{P}) f^\varepsilon\|_{H^{N-2}_{x}L^2_v}^2\\
\;&+\frac{\eta}{\e^2}
\| \partial^{\alpha}_x(\mathbf{I}-\mathbf{P})f^\e\|_{L^2_{x,v}}^2+\big[\mathcal{E}_N(t)\big]^{\frac{1}{2}} \mathcal{D}_N(t).
\esp
\eals

For the term ${\mathcal{X}_6} $, using \eqref{fdefenjie} gives rise to
\bals
\bsp
\mathcal{X}_6
=\;&\underbrace{\sum_{1\leq|\a'|\leq|\a|} \Big\langle v\times \partial^{\alpha'}_xB^\e \cdot \nabla_v \partial^{\alpha-\alpha'}_x f^\varepsilon,  \partial^\alpha_x(\mathbf{I}-\mathbf{P}) f^\varepsilon\Big\rangle_{L^2_{x,v} }}_{\mathcal{X}_{6,1}}
\\
\;&\underbrace{-
\sum_{1\leq|\a'|\leq|\a|} \Big\langle v\times \partial^{\alpha'}_xB^\e \cdot \partial^{\alpha-\alpha'}_x f^\varepsilon,  \partial^\alpha_x\nabla_v \mathbf{P} f^\varepsilon\Big\rangle_{L^2_{x,v} }.}_{\mathcal{X}_{6,2}}
\esp
\eals
Applying the analogous arguments  of the estimates of ${\mathcal{X}_{5,1}}$ and ${\mathcal{X}_{5,2}}$, we have
\bals
\bsp
\mathcal{X}_{6,1}
\lesssim\;&\Big[\sum_{1\leq|\a'|\leq2}\|\partial^{\alpha'}_x B^{\e}\|_{L^\infty_{x}} \cdot \|\langle v\rangle\nabla_v\partial^{\alpha-\alpha'}_x f^\varepsilon\|_{L^2_{x,v}}+\|\partial_x^\a B^{\e}\|_{L^2_{x}} \cdot \|\langle v\rangle\nabla_vf^\varepsilon\|_{L^\infty_{x}L^2_{v}}\\
\;&
+
\sum_{3\leq|\a'|\leq|\a|-1} \|\partial^{\alpha'}_xB^{\e}\|_{L^6_{x}} \cdot \|\langle v\rangle\nabla_v\partial^{\alpha-\a'}_x f^\varepsilon\|_{L^3_{x}L^2_{v}}
\Big]\| \partial^{\alpha}_x(\mathbf{I}-\mathbf{P})f^\e\|_{L^2_{x,v}}\\
\lesssim\;&\|\nabla_x B^{\e}\|_{H^2_{x}}
\sum_{|\a'|=1}\!\!\left(
\|\langle v\rangle\nabla_v\partial^{\alpha-\a'}_x(\mathbf{I}-\mathbf{P}) f^\varepsilon\|_{L^2_{x,v}}+
\|\partial^{\alpha-\a'}_x\mathbf{P}f^\varepsilon\|_{L^2_{x,v}}\right)
\| \partial^{\alpha}_x(\mathbf{I}-\mathbf{P})f^\e\|_{L^2_{x,v}}\\
\;&
+\| B^{\e}\|_{H^N_{x}}\left(
\|\langle v\rangle\nabla_v\nabla_x(\mathbf{I}-\mathbf{P}) f^\varepsilon\|_{H^{N-3}_{x}L^2_v}
+\|\nabla_x\mathbf{P} f^\varepsilon\|_{H^{N-3}_{x}L^2_v}
\right)
\| \partial^{\alpha}_x(\mathbf{I}-\mathbf{P}) f^\e\|_{L^2_{x,v}}\\
\lesssim\;&\e^2 \|\nabla_x B^{\e}\|_{H^2_{x}}^2\sum_{|\a'|=1}
\|\langle v\rangle\nabla_v\partial^{\alpha-\a'}_x(\mathbf{I}-\mathbf{P}) f^\varepsilon\|_{L^2_{x,v}}^2
+\e^2\| B^{\e}\|_{H^N_{x}}^2
\|\langle v\rangle\nabla_v(\mathbf{I}-\mathbf{P}) f^\varepsilon\|_{H^{N-2}_{x}L^2_v}^2\\
\;&+\frac{\eta}{\e^2}
\| \partial^{\alpha}_x(\mathbf{I}-\mathbf{P})f^\e\|_{L^2_{x,v}}^2+\big[\mathcal{E}_N(t)\big]^{\frac{1}{2}} \mathcal{D}_N(t),
\esp
\eals
and
\bals
\bsp
\mathcal{X}_{6,2}
\lesssim\;&\| B^{\e}\|_{H^N_{x}}
\|\nabla_x f^\varepsilon\|_{H^{N-1}_{x}L^{2}_{v}}\| \partial^{\alpha}_x\mathbf{P}f^\e\|_{L^2_{x,v}}
\lesssim\big[\mathcal{E}_N(t)\big]^{\frac{1}{2}} \mathcal{D}_N(t).
\esp
\eals
As a result, collecting  the above estimates of $\mathcal{X}_{6,1}$ and $\mathcal{X}_{6,2}$   leads  to
\bals
\bsp
\mathcal{X}_{6}
\lesssim\;&\e^2\|\nabla_x B^{\e}\|_{H^2_{x}}^2\sum_{|\a'|=1}
\|\langle v\rangle\nabla_v\partial^{\alpha-\a'}_x(\mathbf{I}-\mathbf{P}) f^\varepsilon\|_{L^2_{x,v}}^2
+\e^2\| B^{\e}\|_{H^N_{x}}^2
\|\langle v\rangle\nabla_v(\mathbf{I}-\mathbf{P}) f^\varepsilon\|_{H^{N-2}_{x}L^2_v}^2\\
\;&+\frac{\eta}{\e^2}
\| \partial^{\alpha}_x(\mathbf{I}-\mathbf{P})f^\e\|_{L^2_{x,v}}^2+\big[\mathcal{E}_N(t)\big]^{\frac{1}{2}} \mathcal{D}_N(t).
\esp
\eals
For the term $\mathcal{X}_7$,
employing \eqref{hard gamma} and the Sobolev embeddings theory  implies
\bals
\bsp
\mathcal{X}_7
\lesssim &\;
\frac{1}{\e}\Big[
\|f^\e\|_{L^\infty_x L^2_v(\nu)}
\|\partial^{\alpha}_xf^\e\|_{L^2_{x,v}}
+
\sum_{1\leq|\a'|\leq|\a|-1} \|\partial^{\alpha'}_xf^\e\|_{L^6_x L^2_v(\nu)}
\|\partial^{\alpha-\a'}_xf^\e\|_{L^3_{x}L^2_{v}}\\
&\;\quad
+\|\partial^{\alpha}_xf^\e\|_{L^2_{x,v}(\nu)}
\|f^\e\|_{L^\infty_{x}L^2_{v}}\Big]
\|\partial^{\alpha}_x(\mathbf{I}-\mathbf{P})f^\e\|_{L^2_{x,v}(\nu)}\\
\lesssim\;&
\big[\mathcal{E}_N(t)\big]^{\frac{1}{2}}\mathcal{D}_N(t).
\esp
\eals

Consequently, substituting all the estimates of $\mathcal{X}_4\sim\mathcal{X}_7$
into \eqref{without weight} and choosing $\eta$ small enough and then  summing  the resulting inequality over $1\leq|\alpha| \leq N$, we obtain
\bal
\bsp\label{without weight-higher}
&\frac{\d}{\d t}\Big[\!\!\!\sum_{
1\leq|\a|\leq N}\!\!\!\left\| \partial^\a_xf^\varepsilon\right\|^2_{L^2_{x,v} } +\!\!\!
\sum_{1\leq|\a|\leq N}\!\!\!
\left\| \partial^\alpha_x(E^\varepsilon,B^\e)\right\|^2_{L^2_x}\Big]
+\frac{\sigma_0}{\varepsilon^2}\sum_{1\leq|\a|\leq N}\left\|\partial^\alpha_x (\mathbf{I}-\mathbf{P})f^\varepsilon\right\|^2_{L^2_{x,v}(\nu)}  \\
\lesssim\;&
\big[\mathcal{E}_N(t)\big]^{\frac{1}{2}} \mathcal{D}_N(t)
+\mathcal{E}_N(t)
\left(
\|\langle v\rangle\nabla_v(\mathbf{I}-\mathbf{P}) f^\varepsilon\|_{H^{N-2}_{x}L^2_v}^2
+
\|\langle v\rangle   (\mathbf{I}-\mathbf{P})f^\e\|_{H^{N-1}_{x}L^2_{v}}^2\right)\\
\;&
+\!\e^3\|\nabla_x [E^{\e}, B^{\e}]\|_{H^2_{x}}^2\Big(\e\left\|\langle v\rangle \nabla_x^N(\mathbf{I}-\mathbf{P})f^\varepsilon\right\|_{L^2_{x,v}}^2
 +\frac{1}{\e}\left\|\langle v\rangle \nabla_v \nabla_x^{N-1}(\mathbf{I}-\mathbf{P})f^\varepsilon\right
 \|_{L^2_{x,v}}^2\Big)
.
\esp
\eal
In summary, adding
\eqref{without weight-basic-result} and  \eqref{without weight-higher} together and  employing the \emph{a priori} assumption  \eqref{priori assumption} lead  to
\eqref{estimate without weight}. This completes the proof of Lemma \ref{energy estimate without weight}.
\end{proof}

With  Lemma  \ref{macroscopic estimate}  and Lemma \ref{energy estimate without weight} in hand,   we are now in the position to complete the proof of Proposition \ref{result-basic energy estimate}.
\begin{proof}[\textbf{Proof of Proposition \ref{result-basic energy estimate}}]
Choosing $0<\eta_1\ll 1$ and $\delta_0$    small enough and taking the proper linear combination of  \eqref{macro estimate}$\times \eta_1+$(\ref{estimate without weight}) induce  that
\begin{align}
\begin{split}\label{basic energy estimate}
\frac{\d}{\d t}\mathcal{E}_{N}(t)\;&+\mathcal{D}_{N}(t)
\lesssim\delta_0\left( \left\|\langle v\rangle \nabla_v(\mathbf{I}-\mathbf{P})f^\varepsilon\right\|_{H^{N-2}_xL^2_v}^2
+\left\|\langle v\rangle (\mathbf{I}-\mathbf{P})f^\varepsilon\right\|_{H^{N-1}_xL^2_v}^2\right)\\
 \;&
 +\frac{\delta_0}{(1+t)^{1+\vartheta}}\Big(\e\left\|\langle v\rangle \nabla_x^N(\mathbf{I}-\mathbf{P})f^\varepsilon\right\|_{L^2_{x,v}}^2
 +\frac{1}{\e}\left\|\langle v\rangle \nabla_v \nabla_x^{N-1}(\mathbf{I}-\mathbf{P})f^\varepsilon\right\|_{L^2_{x,v}}^2\Big),
\end{split}
\end{align}
where ${\mathcal{E}}_N (t)$ comes from the energy given in  \eqref{macro energy2} and   (\ref{estimate without weight}), as follows
\bal
\bsp\label{eq:energy:estimate:result-hard-1}
{\mathcal{E}}_N (t):=\;&\left\| f^\varepsilon\right\|^2_{H^N_xL^2_v}+\left\| E^\varepsilon \right\|^2_{H^N_x}
+\left\| B^\varepsilon\right\|^2_{H^N_x}+\eta_1\e{\mathcal{E}}_{int}^N (t).
\esp
\eal
By referring back to  \eqref{dissipation functional},
\eqref{dissipation functional 2-N}
 and \eqref{hard assumption}, we can easily find that
\bal
\bsp\label{basic energy estimate-1}
&\left\|\langle v\rangle \nabla_v(\mathbf{I}-\mathbf{P})f^\varepsilon\right\|_{H^{N-2}_xL^2_v}^2
+\left\|\langle v\rangle (\mathbf{I}-\mathbf{P})f^\varepsilon\right\|_{H^{N-1}_xL^2_v}^2
\leq
\overline{\mathcal{D}}_{N-1,4N}(t)\leq
{\overline{\mathcal{D}}}_{N-1,l_1}(t)
,\\
&\e\left\|\langle v\rangle \nabla_x^N(\mathbf{I}-\mathbf{P})f^\varepsilon\right\|_{L^2_{x,v}}^2
\leq
(1+t)^{\frac{1+\vartheta}{2}}{\widetilde{\mathcal{D}}}_{N,l_2}(t),\\
&\frac{1}{\e}\left\|\langle v\rangle \nabla_v \nabla_x^{N-1}(\mathbf{I}-\mathbf{P})
 f^\varepsilon\right\|_{L^2_{x,v}}^2
\leq\e
(1+t)^{1+\vartheta}{\widetilde{\mathcal{D}}}_{N,l_2}(t).
\esp
\eal
Hence, putting \eqref{basic energy estimate-1} into \eqref{basic energy estimate} yields the desired estimate \eqref{basic-energy-estimate-result}.
This completes the proof of  Proposition \ref{result-basic energy estimate}.
\end{proof}

\subsection{Weighted Energy Estimate}
\hspace*{\fill}

In this subsection, we establish  the  energy estimates  for the VMB system   \eqref{rVPB}  with different weight functions $\overline{w}_{l_1}(\alpha, \beta)$ and $\widetilde{w}_{l_2}(\alpha, \beta)$.
 First of all, we establish the following  the  energy estimate    with the weight  $\overline{w}_{l_1}(\alpha, \beta)$ for $|\a|+|\beta|\leq N-1$.
\begin{proposition}\label{weighted 1}Let $N, l_1,
 l_2$ be fixed parameters  stated in Theorem \ref{mainth1}.
 Assume that $ (f^\e, E^\e, B^\e) $ is a solution to the VMB system  \eqref{rVPB} defined on $ [0, T] \times \mathbb{R}^3 \times \mathbb{R}^3$
and the \emph{a priori} assumption  \eqref{priori assumption} holds true for $\delta_0$ small enough. Then, there holds
\begin{align}\label{weight estimate1}
\begin{split}
&\frac{\d}{\d t}{\overline{\mathcal{E}}}_{N-1,l_1}(t)+{\overline{\mathcal{D}}}_{N-1,l_1}(t)
\\
\lesssim\;&
 \delta_0(1+t)^{-{\frac{1+s}{2}}(l_2-l_1)}\frac{1}{\e^2}
 \sum_{{|\a|+|\b|\leq N-1,|\a|\geq 1}}
 \!\!\!\!
\big\|\widetilde{w}_{l_2}(|\a|-1,|\b|+1)\nabla_v\partial_{\b}^{\a-\a'}(\mathbf{I}-\mathbf{P})
f^\varepsilon\big\|_{L^2_{x,v}}^2\\
\;&+(1+\delta_0)\mathcal{D}_{N}(t)
 .
\end{split}
\end{align}
Here, ${\overline{\mathcal{E}}}_{N-1,l_1}(t)$ and ${\overline{\mathcal{D}}}_{N-1,l_1}(t)$ are defined in \eqref{energy functional} and \eqref{dissipation functional}, respectively.
\end{proposition}

\begin{proof}[\textbf{Proof of Proposition \ref{weighted 1}}]
Applying microscopic projection $(\mathbf{I}-\mathbf{P})$ to   the first equation in (\ref{rVPB}) and using  $(\mathbf{I}-\mathbf{P})( v\sqrt{\mu}\cdot E^\e)=0$, we have
\bal
\bsp\label{rVMBweiguan-1}
&\;\pt_t(\mathbf{I}-\mathbf{P})f^\e+\frac{1}{\e}v\cdot \nabla_x (\mathbf{I}-\mathbf{P})f^\e+\frac{1}{\e^2}L((\mathbf{I}-\mathbf{P})f^\e)\\
=&\;-(\e E^\e+v \times B^\e)\cdot \nabla_v (\mathbf{I}-\mathbf{P}) f^\e
+\frac{\e}{2} v \cdot E^\e(\mathbf{I}-\mathbf{P}) f^\e
+\frac{1}{\e}\Gamma(f^\e,f^\e)
+[[\mathbf{P},\mathcal{A}]]f^\e.
\esp
\eal
Here $ [[\mathbf{P},\mathcal{A}]]
:=\mathbf{P}\mathcal{A}-\mathcal{A}\mathbf{P}$ denotes the commutator of two operators $\mathbf{P}$ and $\mathcal{A}$  given by
$$\mathcal{A}:=\frac{1}{\e}v\cdot \nabla_x + (\e E^\e+v\times B^\e)\cdot \nabla_v-\frac{
\e}{2} v \cdot E^\e.$$
Taking $\partial_{\b}^\a$  with $|\beta|=m$ and $|\alpha|+|\beta| \leq N-1$ to (\ref{rVMBweiguan-1}) and then integrating the resulting identity over $\bbR^3_x\times\bbR^3_v$ by multiplying $\overline{w}^{2}_{l_1}(\a,\b)\partial_{\b}^\a(\mathbf{I}-\mathbf{P})f^\e$, we  obtain
\bal
\bsp\label{diyigeweight:L^2:2}
& \frac{1}{2}\frac{\d }{\d t} \|\overline{w}_{l_1}(\a,\b)\partial_{\b}^\a(\mathbf{I}-\mathbf{P})f^\e\|_{L_{x,v}^2}^2
+\frac{q\vartheta}{(1+t)^{1+\vartheta}}
\|\langle v\rangle \overline{w}_{l_1}(\a,\b)
\partial_{\b}^\a(\mathbf{I}-\mathbf{P})f^\e\|_{L_{x,v}^2}^2\\
\;&
+\frac{\sigma_0}{\e^2}\| \overline{w}_{l_1}(\a,\b)\partial_{\b}^\a(\mathbf{I}-\mathbf{P})f^\e\|_{L_{x,v}^2(\nu)}^2
\\
\le
&\;\frac{\eta}{\e^2}\sum_{|\b'|\leq|\b|}\| \overline{w}_{l_1}(\a,\b')\partial_{\b'}^\a(\mathbf{I}-\mathbf{P})f^\e\|_{L_{x,v}^2(\nu)}^2
+\frac{C_{\eta}}{\e^2}\| \partial^\a_x(\mathbf{I}-\mathbf{P})f^\e\|_{L_{x,v}^2(\nu)}^2\\
&\;+\Big\langle- \frac{1}{\e}C_{\b}^{e_i}\partial_{\b-e_i}^{\a+e_i} (\mathbf{I}-\mathbf{P}) f^\e , \overline{w}^{2}_{l_1}(\a,\b)\partial_{\b}^\a(\mathbf{I}-\mathbf{P})f^\e
\Big\rangle_{L^2_{x,v}}\\
&\;
+\Big\langle
\partial_{\b}^\a\Big(-v\times B^\e \cdot\nabla_v(\mathbf{I}-\mathbf{P}) f^\e
\Big) , \overline{w}^{2}_{l_1}(\a,\b)\partial_{\b}^\a(\mathbf{I}-\mathbf{P})f^\e\Big
\rangle_{L^2_{x,v}}\\
&\;+\Big\langle
\partial_{\b}^\a\Big(-\e E^\e \cdot\nabla_v(\mathbf{I}-\mathbf{P}) f^\e
\Big) , \overline{w}^{2}_{l_1}(\a,\b)\partial_{\b}^\a(\mathbf{I}-\mathbf{P})f^\e\Big
\rangle_{L^2_{x,v}}
\\
&\;+\Big\langle
\partial_{\b}^\a\Big(\frac{\e}{2}v\cdot E^\e (\mathbf{I}-\mathbf{P}) f^\e\Big) , \overline{w}^{2}_{l_1}(\a,\b)\partial_{\b}^\a(\mathbf{I}-\mathbf{P})f^\e\Big
\rangle_{L^2_{x,v}}\\
&\;\!+\!\Big\langle\frac{1}{\e}\! \partial_{\b}^\a\Gamma(f^\e,f^\e), \overline{w}^{2}_{l_1}(\a,\b)\partial_{\b}^\a(\mathbf{I}\!-\!\mathbf{P})f^\e\!\Big
\rangle_{L^2_{x,v}}
\!\!\!\!\!\!\!+\!\Big\langle \!\partial_{\b}^\a[[\mathbf{P},\mathcal{A}]]f^\e,  \overline{w}^{2}_{l_1}(\a,\b)\partial_{\b}^\a(\mathbf{I}\!-\!\mathbf{P})f^\e\!\Big \rangle_{L^2_{x,v}}.
\esp
\eal
Here, we    have used the spectral inequality \eqref{L coercive3}.
Now   we classify the last six terms on the right-hand side  of \eqref{diyigeweight:L^2:2} as  $\mathcal{Y}_1$ to $\mathcal{Y}_6$ and estimate them term by term.
The H\"{o}lder inequality and   the Cauchy--Schwarz inequality with $\eta$  give
\bals
\bsp
\mathcal{Y}_1
\lesssim \;&C_{\eta}
\|\overline{w}_{l_1}(\a+e_i,\b-e_i)\partial_{\b-e_i}^{\a+e_i} (\mathbf{I}-\mathbf{P}) f^\e\|_{L^2_{x,v}}^2
+ \frac{\eta}{\e^2}\|\overline{w}_{l_1}(\a,\b)\partial_{\b}^\a(\mathbf{I}-\mathbf{P})f^\e\|_{L^2_{x,v}}^2,
\esp
\eals
where we have used the fact  $\overline{w}_{l_1}(\a+e_i,\b-e_i)=\overline{w}_{l_1}(\a,\b)$.

For the term $\mathcal{Y}_2$, direct calculation shows us that
\bals
\bsp
\mathcal{Y}_2=
\;&
\sum_{1\leq|\a'|\leq|\a|}\Big\langle
 -v\times \partial_x^{\a'} B^\e \cdot\nabla_v\partial_{\b}^{\a-\a'}(\mathbf{I}-\mathbf{P}) f^\e, \overline{w}^{2}_{l_1}(\a,\b)\partial_{\b}^\a(\mathbf{I}-\mathbf{P})f^\e
 \Big
\rangle_{L^2_{x,v}}
\\
\;&
+\sum_{|\a'|\leq|\a|,|\b_1|=1}\Big\langle
 -\partial_{\b_1}v\times \partial_x^{\a'} B^\e \cdot\nabla_v\partial_{\b-\b_1}^{\a-\a'}(\mathbf{I}-\mathbf{P}) f^\e, \overline{w}^{2}_{l_1}(\a,\b)\partial_{\b}^\a(\mathbf{I}-\mathbf{P})f^\e
 \Big
\rangle_{L^2_{x,v}} \\
\equiv\;& \mathcal{Y}_{2,1}+\mathcal{Y}_{2,2}.
\esp
\eals
Applying the H\"{o}lder inequality, \eqref{NGinequality}  and the Sobolev embedding inequality,
we   deduce
\bal
\bsp\label{Bdekunnan:1}
\mathcal{Y}_{2,1}
\lesssim\;&\Big[\|\nabla_x B^\e \|_{L^\infty_x} \sum_{|\a'|=1}\|\langle v\rangle \overline{w}_{l_1}(|\a|-1,|\b|+1)\nabla_v\partial_{\b}^{\a-\a'}(\mathbf{I}-\mathbf{P}) f^\e \|_{L^2_{x,v}}
\\
\;&\;+\sum_{2\leq|\a'|\leq\max\{|\a|-1,2\}}\|\partial_x^{\a'} B^\e \|_{L^\infty_x} \|\overline{w}_{l_1}(|\a-\a'|,|\b|+1)\nabla_v\partial_{\b}^{\a-\a'}(\mathbf{I}-\mathbf{P}) f^\e \|_{L^2_{x,v}}\\
\;&\;+\|\partial_x^{\a} B^\e \|_{L^3_x} \|\overline{w}_{l_1}(1,|\b|+1)\nabla_v\partial_{\b}(\mathbf{I}-\mathbf{P}) f^\e \|_{L^6_{x}L^2_{v}}\Big]\| \overline{w}_{l_1}(\a,\b)\partial_{\b}^{\a}(\mathbf{I}-\mathbf{P}) f^\e\|_{L^2_{x,v}}\\
\lesssim \;&\e^2\|\nabla_x B^\e \|_{L^\infty_x}^2\sum_{|\a'|=1} \|\langle v\rangle \overline{w}_{l_1}(|\a|-1,|\b|+1)\nabla_v\partial_{\b}^{\a-\a'}(\mathbf{I}-\mathbf{P}) f^\e \|_{L^2_{x,v}}^2\\
\;&+
\eta{\overline{\mathcal{D}}}_{N-1,l_1}(t)+ \e^2\|\nabla_xB^\e\|_{H^{N-1}_x}{\overline{\mathcal{D}}}_{N-1,l_1}(t).
\esp
\eal
Here, we have used the facts that
\bals
\bsp
\;&\langle v\rangle \overline{w}_{l_1}(\a,\b)\lesssim \overline{w}_{l_1}(|\a-\a'|,|\b|+1)\quad \text{in the case of}\quad
2\leq|\a'|\leq\max\{|\a|-1,2\},\\
\;&
\langle v\rangle \overline{w}_{l_1}(\a,\b)\lesssim \overline{w}_{l_1}(1,|\b|+1)\quad\quad \;\quad \;\; \text{in the case of}\quad  |\a|\geq3.
\esp
\eals
Furthermore,  it is derived  from $  l_2-l_1=\frac{N-1}{2}:=\widetilde{\ell} $ and  the Cauchy--Schwarz inequality  that
\bal
\bsp\label{Bdekunnan:2}
\;&\e^2\|\nabla_x B^\e \|_{L^\infty_x}^2 \sum_{|\a'|=1}\|\langle v\rangle \overline{w}_{l_1}(|\a|-1,|\b|+1)\nabla_v\partial_{\b}^{\a-\a'}(\mathbf{I}-\mathbf{P}) f^\e \|_{L^2_{x,v}}^2\\
\lesssim\;&
\Big[(\e^{2-\varrho}\|\nabla_x B^\e \|_{L^\infty_x})^{\frac{2}{\varrho}}\sum_{|\a'|=1}\|\langle v\rangle^{\widetilde{\ell}} \overline{w}_{l_1}(|\a|-1,|\b|+1)\nabla_v\partial_{\b}^{\a-\a'}(\mathbf{I}-\mathbf{P}) f^\e \|_{L^2_{x,v}}^2\Big]^{\varrho}\\
\;&\times
\Big[\frac{1}{\e^2}\sum_{|\a'|=1}\| \overline{w}_{l_1}(|\a|-1,|\b|+1)\nabla_v\partial_{\b}^{\a-\a'}(\mathbf{I}-\mathbf{P}) f^\e \|_{L^2_{x,v}}^2\Big]^{1-\varrho}\\
\lesssim\;& (\e^{2-\varrho}\|\nabla_x B^\e \|_{L^\infty_x})^{\frac{2}{\varrho}}
\sum_{|\a'|=1}\|\langle v\rangle^{\widetilde{\ell}} \overline{w}_{l_1}(|\a|-1,|\b|+1)\nabla_v\partial_{\b}^{\a-\a'}(\mathbf{I}-\mathbf{P}) f^\e \|_{L^2_{x,v}}^2\\
\;&+
\frac{\eta}{\e^2}\sum_{|\a'|=1}\| \overline{w}_{l_1}(|\a|-1,|\b|+1)\nabla_v\partial_{\b}^{\a-\a'}(\mathbf{I}-\mathbf{P}) f^\e \|_{L^2_{x,v}}^2\\
\lesssim\;& (\e^{2}\|\nabla_x B^\e \|_{L^\infty_x})^{\frac{2}{\varrho}}
\!\!\sum_{|\a'|=1}\!\!\frac{1}{\e^{2}}\| \widetilde{w}_{l_2}(|\a|-1,|\b|+1)\nabla_v\partial_{\b}^{\a-\a'}(\mathbf{I}-\mathbf{P}) f^\e \|_{L^2_{x,v}}^2+
\eta{\overline{\mathcal{D}}}_{N-1,l_1}(t)
,
\esp
\eal
where we have taken $\varrho$  as
$
\varrho={\widetilde{\ell}}^{-1}.
$
Putting the bound \eqref{Bdekunnan:2} into \eqref{Bdekunnan:1}, we obtain
\bals
\bsp
\mathcal{Y}_{2,1}\lesssim\;& (\e^{2}\|\nabla_x B^\e \|_{L^\infty_x})^{\frac{2}{\varrho}}
\sum_{|\a'|=1}\frac{1}{\e^{2}}\| \widetilde{w}_{l_2}(|\a|-1,|\b|+1)\nabla_v\partial_{\b}^{\a-\a'}(\mathbf{I}-\mathbf{P}) f^\e \|_{L^2_{x,v}}^2+
\eta{\overline{\mathcal{D}}}_{N-1,l_1}(t)\\
\;&+\e^2\|\nabla_xB^\e\|_{H^{N-1}_x}{\overline{\mathcal{D}}}_{N-1,l_1}(t).
\esp
\eals
For the term $\mathcal{Y}_{2,2}$,  we have
\bals
\bsp
\mathcal{Y}_{2,2}
\lesssim\;&\!\!\sum_{|\a'|\leq|\a|,|\b_1|=1}\!\!\!\!\!\!\|\partial_x^{\a'}\!\! B^\e \|_{L^\infty_x} \|\overline{w}_{l_1}(\a-\a',\b)\nabla_v\partial_{\b-\b_1}^{\a-\a'}(\mathbf{I}-\mathbf{P}) f^\e \|_{L^2_{x,v}}\| \overline{w}_{l_1}(\a,\b)\partial_{\b}^{\a}(\mathbf{I}-\mathbf{P}) f^\e\|_{L^2_{x,v}}\\
\lesssim \;& \e^2\|\nabla_xB^\e\|_{H^{N-1}_x}{\overline{\mathcal{D}}}_{N-1,l_1}(t)
\esp
\eals
by the H\"{o}lder inequality and \eqref{fdefenjie}.
Here, we also have   used the fact $|\a|\leq N-2$ because of  $|\b|\geq1$ and $|\a|+|\b|\leq N-1$ in this case.
As a result,  collecting the above bounds $\mathcal{Y}_{2,1}$ and $\mathcal{Y}_{2,2}$ and utilizing   the \emph{a priori} assumption  \eqref{priori assumption} lead  to
\bals
\bsp
\mathcal{Y}_{2}\lesssim\;& \delta_0(1+t)^{-\frac{1+s}{2}\widetilde{\ell}}
\frac{1}{\e^{2}}\sum_{|\a'|=1}\| \widetilde{w}_{l_2}(|\a|-1,|\b|+1)\nabla_v\partial_{\b}^{\a\!-\!\a'}(\mathbf{I}\!-\!\mathbf{P}) f^\e \|_{L^2_{x,v}}^2+
(\eta \!+\!\delta_0){\overline{\mathcal{D}}}_{N-1,l_1}(t).
\esp
\eals

For the term $\mathcal{Y}_3$, we have
\bals
\bsp
\mathcal{Y}_3=\;&
\Big\langle
 -\e E^\e \cdot\nabla_v\partial_{\b}^{\a}(\mathbf{I}-\mathbf{P}) f^\e
,  \overline{w}^{2}_{l_1}(\a,\b)\partial_{\b}^\a(\mathbf{I}-\mathbf{P})f^\e\Big
\rangle_{L^2_{x,v}}\\
\;&+\sum_{1\leq|\a'|\leq|\a| }\Big\langle
 - \e\partial_x^{\a'} E^\e \cdot\nabla_v\partial_{\b}^{\a-\a'}(\mathbf{I}-\mathbf{P}) f^\e
,  \overline{w}^{2}_{l_1}(\a,\b)\partial_{\b}^\a(\mathbf{I}-\mathbf{P})f^\e\Big
\rangle_{L^2_{x,v}}\\
\equiv\;&\mathcal{Y}_{3,2}+\mathcal{Y}_{3,2}.
\esp
\eals
Notice that
$$ \nabla_v \overline{w}_{l_1}(\a,\b)=
(l_1-|\a|-|\b|)\langle v\rangle ^{l_1-|\a|-|\b|-1}e^{\frac{q\langle v\rangle^2}{(1+t)^\vartheta}}\frac{v}{|v|}
+\langle v\rangle ^{l_1-|\a|-|\b|}
e^{\frac{q\langle v\rangle^2}{(1+t)^\vartheta}}
{\frac{2q}{(1+t)^\vartheta}}\langle v\rangle,$$
which means $\nabla_v \overline{w}_{l_1}(\a,\b)\lesssim \langle v\rangle \overline{w}_{l_1}(\a,\b)$,
and hence, we have
\bals
\bsp
\mathcal{Y}_{3,1}=
\;&\Big\langle
 -\e E^\e \cdot\frac{\nabla_v \overline{w}_{l_1}(\a,\b)}{\overline{w}_{l_1}(\a,\b)}\partial_{\b}^{\a}(\mathbf{I}-\mathbf{P}) f^\e
,
\overline{w}^{2}_{l_1}(\a,\b)\partial_{\b}^\a(\mathbf{I}-\mathbf{P})f^\e\Big
\rangle_{L^2_{x,v}}\\
\lesssim\;&\e\| E^\e \|_{L^\infty_x}
\| \langle v\rangle \overline{w}_{l_1}(\a,\b)\partial_{\b}^{\a}(\mathbf{I}-\mathbf{P}) f^\e\|_{L^2_{x,v}}
\| \overline{w}_{l_1}(\a,\b)\partial_{\b}^{\a}(\mathbf{I}-\mathbf{P}) f^\e\|_{L^2_{x,v}}\\
\lesssim\;&\e^2\|\nabla_xE^\e\|_{H^1_x}(1+t)^{\frac{1+\vartheta}{2}}
 {\overline{\mathcal{D}}}_{N-1,l_1}(t).
\esp
\eals
Taking  the H\"{o}lder inequality and   the Sobolev embedding theory gives rise to
\bals
\bsp
\mathcal{Y}_{3,2}\lesssim
\;&\e\Big[\sum_{1\leq|\a'|\leq\max\{|\a|-1,1\}}\|\partial_x^{\a'} E^\e \|_{L^\infty_x} \|\overline{w}_{l_1}(|\a-\a'|,|\b|+1)\nabla_v\partial_{\b}^{\a-\a'}(\mathbf{I}-\mathbf{P}) f^\e \|_{L^2_{x,v}}\\
\;&\quad+\|\partial_x^{\a} E^\e \|_{L^3_x} \|\overline{w}_{l_1}(1,|\b|+1)\nabla_v\partial_{\b}(\mathbf{I}-\mathbf{P}) f^\e \|_{L^6_{x}L^2_{v}}\Big]\| \overline{w}_{l_1}(\a,\b)\partial_{\b}^{\a}(\mathbf{I}-\mathbf{P}) f^\e\|_{L^2_{x,v}}\\
\lesssim \;& \e^3\|\nabla_xE^\e\|_{H^{N-1}_x}{\overline{\mathcal{D}}}_{N-1,l_1}(t),
\esp
\eals
where we have used the facts that
\bals
\bsp
\;& \overline{w}_{l_1}(\a,\b)\lesssim \overline{w}_{l_1}(|\a-\a'|,|\b|+1)\quad \,\text{in the case of}\quad
1\leq|\a'|\leq\max\{|\a|-1,1\},\\
\;&
 \overline{w}_{l_1}(\a,\b)\lesssim \overline{w}_{l_1}(1,|\b|+1)\quad\quad \qquad  \text{in the case of}\quad  |\a|\geq2.
\esp
\eals
Thus, the  above  estimates of $\mathcal{Y}_{3,1}$ and $\mathcal{Y}_{3,2}$
and   the \emph{a priori} assumption  \eqref{priori assumption}  give us that
\bals
\bsp
\mathcal{Y}_{3}
\lesssim\e^2\|\nabla_xE^\e\|_{H^1_x}(1+t)^{\frac{1+\vartheta}{2}}
 {\overline{\mathcal{D}}}_{N-1,l_1}(t)+
\e^3 \|\nabla_xE^\e\|_{H^{N-1}_x}{\overline{\mathcal{D}}}_{N-1,l_1}(t)\
\lesssim\delta_0{\overline{\mathcal{D}}}_{N-1,l_1}(t).
\esp
\eals

By employing the similar argument in the estimate of $\mathcal{Y}_{2}$, we deduce from    the \emph{a priori} assumption  \eqref{priori assumption}    that
\bals
\bsp
\mathcal{Y}_{4}
\lesssim\e^2\|\nabla_xE^\e\|_{H^1_x}(1+t)^{\frac{1+\vartheta}{2}}
 {\overline{\mathcal{D}}}_{N-1,l_1}(t)+
\e^3 \|\nabla_xE^\e\|_{H^{N-1}_x}{\overline{\mathcal{D}}}_{N-1,l_1}(t)\
\lesssim\delta_0{\overline{\mathcal{D}}}_{N-1,l_1}(t).
\esp
\eals
Further,  employing  \eqref{hard gamma1}, the \emph{a priori} assumption  \eqref{priori assumption}  and   the H\"{o}lder inequality   yields
\bals
\bsp
\mathcal{Y}_{5}
\lesssim
\sqrt{{\overline{\mathcal{E}}}_{N-1,l_1}(t)}{\overline{\mathcal{D}}}_{N-1,l_1}(t)
\lesssim\delta_0{\overline{\mathcal{D}}}_{N-1,l_1}(t).
\esp
\eals
Owing to the    Maxwellian structure for  the macro projection $\mathbf{P}$ defined in \eqref{defination:Pf}, we directly deduce from
the H\"{o}lder inequality and the  Sobolev embedding  theory that
\bals
\bsp
\mathcal{Y}_{6}
\lesssim\;&
\Big[\frac{1}{\e}\|\partial^{\a}_x \nabla_x f^\e\|_{L^2_{x,v}}
+\left(\e\|E^\e\|_{H^{N}_x}+\|B^\e\|_{H^{N}_x}\right)
\left(\|\nabla_x\nabla_v\mathbf{P}f^\e\|_{H^{N-2}_{x}L^2_{v}}
+\|\nabla_xf^\e\|_{H^{N-2}_{x}L^2_{v}}\right)
\\
\;&
+\e\|\nabla_xE^\e\|_{H^{N-2}_x}
\left(\|\nabla_x\mathbf{P}f^\e\|_{H^{N-2}_{x}L^2_{v}}
+\|\nabla_xf^\e\|_{H^{N-2}_{x}L^2_{v}}\right)\Big]
\|\partial_{\b}^\a(\mathbf{I}-\mathbf{P} )f^\e\|_{L^2_{x,v}}\\
\lesssim\;&\mathcal{D}_{N}(t)+\e\mathcal{E}_{N}(t) \mathcal{D}_{N}(t).
\esp
\eals

Therefore,  plugging  all the  estimates of $\mathcal{Y}_{1}\sim\mathcal{Y}_{6}$
into \eqref{diyigeweight:L^2:2},
taking the summation over $\left\{ |\beta|=m,|\alpha|+|\beta| \leq N-1\right\}$ for each given $0 \leq m \leq N-1$, and then taking combination of those $N-1$ estimates with properly chosen constant $C_m > 0 $ $(0 \leq m \leq N-1)$ and $\eta$, $\delta_0$ small enough,
we obtain
\eqref{weight estimate1}. This completes the proof of Proposition \ref{weighted 1}.
\end{proof}

Next,  we turn to establishing    the   weighted energy estimate    with the weight   $\overline{w}_{l_2}(\alpha, \beta)$ as follows.
\begin{proposition}\label{weighted 2:zongguji}
Let $N,
 l_2$ be fixed parameters  stated in Theorem \ref{mainth1}.
   Assume that the \emph{a priori} assumption  \eqref{priori assumption} holds true for $\delta_0$ small enough.
Then, there holds
\bal
\bsp\label{estimate-weighted-2:zong}
\frac{\d}{\d t}\widetilde{\mathcal{E}}_{N,l_2}(t)
+\widetilde{\mathcal{D}}_{N,l_2}(t)
\lesssim
  \overline{\mathcal{D}}_{ N-1,l_1}(t)
+\mathcal{D}_{N}(t)+\eta\e^2(1+t)^{-(1+\vartheta)}\|\nabla_x^N E^{\e}\|_{L^2_x}^2.
\esp
\eal
Here, $\widetilde{\mathcal{E}}_{N,l_2}(t)$ and  $\widetilde{\mathcal{D}}_{N,l_2}(t)$ are defined in \eqref{energy functional 2-N} and \eqref{dissipation functional 2-N}, respectively.
\end{proposition}

To prove Proposition \ref{weighted 2:zongguji}, the construction of
$\widetilde{\mathcal{E}}_{N,l_2}(t)$ given in \eqref{energy functional 2-N}    heavily depends on the  different order   weighted  energy estimate of
$\widetilde{w}_{l_2}(\alpha, \beta)$  matched  with different order time decay factor.
Therefore, to make it clear,   we introduce the  $n$-th weighted
instant energy functional $\widetilde{{\mathcal{E}}}_{l_2}^{(n)}(t)$  of $\widetilde{\mathcal{E}}_{N,l_2}(t)$ and the corresponding dissipation functional $\widetilde{{\mathcal{D}}}_{l_2}^{(n)}(t)$  of $\widetilde{\mathcal{D}}_{N,l_2}(t)$
  defined as
 \begin{align}\label{energy functional 2-N-new}
\widetilde{{\mathcal{E}}}_{l_2}^{(n)}(t):=
\sum_{0\leq j\leq n}
{\mathcal{E}}_{l_2}^{n,j}(t),\quad
\widetilde{{\mathcal{D}}}_{l_2}^{(n)}(t):= \;&
\sum_{{0\leq j\leq n} }
{\mathcal{D}}_{l_2}^{n,j}(t) \quad\text{for }\;0\leq n\leq N.
\end{align}
Here, in the case of $0\leq n\leq N-1$, ${\mathcal{E}}_{l_2}^{n,j}(t)$ and ${\mathcal{D}}_{l_2}^{n,j}(t)$ are given by
\begin{align*}
{\mathcal{E}}_{l_2}^{n,j}(t) :=\;
&
 (1+t)^{-\sigma_{n,j}}\sum_{|\a|+|\b|=n, |\b|=j}\left\|\widetilde{w}_{l_2}(\alpha, \beta) \partial_{\beta}^{\alpha} (\mathbf{I}-\mathbf{P})f^{\varepsilon}\right\|^2_{L^2_{x,v}}\quad  \text{for }\;0\leq j\leq n,\\
{\mathcal{D}}_{l_2}^{{n,j}}(t):= \;&
(1+t)^{-\sigma_{n,j}}\frac{q\vartheta}{(1+t)^{1+\vartheta}}\sum_{|\a|+|\b|=n, |\b|=j}\left\|\langle v\rangle \widetilde{w}_{l_2}(\alpha, \beta) \partial_{\beta}^{\alpha}
(\mathbf{I}-\mathbf{P}) f^{\varepsilon}\right\|^2_{L^2_{x,v}}\nonumber\\
&+(1+t)^{-\sigma_{n,j}}\sum_{|\a|+|\b|=n, |\b|=j}\frac{1}{\varepsilon^{2}} \left\|\widetilde{w}_{l_2}(\alpha, \beta) \partial_{\beta}^{\alpha}
(\mathbf{I}-\mathbf{P})f^{\varepsilon}\right\|^2_{L^2_{x,v}(\nu)} \quad \text{for}\;0\leq j\leq n,
\end{align*}
and in the case of $n=N$, ${\mathcal{E}}_{l_2}^{N,j}(t)$ and ${\mathcal{D}}_{l_2}^{N,j}(t)$ are given by
\bals
{{\mathcal{E}}}_{l_2}^{{N,j}}(t) :=\;
&
(1+t)^{-\sigma_{N,j}}\sum_{|\a|+|\b|=N, |\b|=j}\left\|\widetilde{w}_{l_2}(\alpha, \beta) \partial_{\beta}^{\alpha} (\mathbf{I}-\mathbf{P})f^{\varepsilon}\right\|^2_{L^2_{x,v}}\quad \quad\quad\quad\text{for }\;1\leq j \leq N,\\
{{\mathcal{E}}}_{l_2}^{{N,0}}(t):=\; & (1+t)^{-\sigma_{N,0}}\sum_{|\a|=N}\e\left\|\widetilde{w}_{l_2}(\a, 0) \partial_x^{\a} f^{\varepsilon}\right\|^2_{L^2_{x,v}},\\
{{\mathcal{D}}}_{l_2}^{{N,j}}(t):=\;&
 (1+t)^{-\sigma_{N,j}}\sum_{|\a|+|\b|=N, |\b|=j}\frac{q\vartheta}{(1+t)^{1+\vartheta}}\left\|\langle v\rangle \widetilde{w}_{l_2}(\alpha, \beta) \partial_{\beta}^{\alpha}
(\mathbf{I}-\mathbf{P})f^{\varepsilon}\right\|^2_{L^2_{x,v}} \\
\;&+(1+t)^{-\sigma_{N,j}}\!\!\!\!\sum_{|\a|+|\b|=N, |\b|=j}\frac{1}{\varepsilon^{2}} \left\|\widetilde{w}_{l_2}(\alpha, \beta) \partial_{\beta}^{\alpha}
(\mathbf{I}-\mathbf{P})f^{\varepsilon}\right\|^2_{L^2_{x,v}(\nu)}
\;\quad\text{for}\; 1\leq j\leq N,
\\
{{\mathcal{D}}}_{l_2}^{{N,0}}(t):=
\;&  (1+t)^{-\sigma_{N,0}}\sum_{|\a|=N}
\frac{\e q\vartheta}{(1+t)^{1+\vartheta}}\left\|\langle v\rangle \widetilde{w}_{l_2}(\a, 0) \partial_{x}^{\a} f^{\varepsilon}\right\|^2_{L^2_{x,v}} \nonumber\\
\;&+(1+t)^{-\sigma_{N,0}}\frac{1}{\varepsilon} \sum_{|\a|=N}\left\|\widetilde{w}_{l_2}(\a, 0) \partial_{x}^{\a}(\mathbf{I}-\mathbf{P})
f^{\varepsilon}\right\|^2_{L^2_{x,v}(\nu)}
\eals
with $\sigma_{n,|\b|}$  given by
\bal
\bsp\label{sigma:def}
\;& \sigma_{n,0}=0, \qquad\qquad\sigma_{n,|\b|+1}-\sigma_{n,|\b|}=\frac{1+\vartheta}{2} \qquad \text{for }\; n\leq N-1,\\
\;&\sigma_{N,0}=\frac{1+\vartheta}{2} ,\quad\;\;\sigma_{N,|\b|+1}-\sigma_{N,|\b|}=\frac{1+\vartheta}{2}.
\esp
\eal

Then  we   provide the weighted energy estimate   with the weight function $\widetilde{w}_{l_2}(\alpha, \beta)$ for $|\a|+|\beta|= N$  as follows.
\begin{lemma}\label{weighted 2:N}
Assume that  the \emph{a priori} assumption  \eqref{priori assumption} holds true for $\delta_0$ small enough.
Then, there holds
\bal
\bsp\label{estimate-weighted-2:N}
\frac{\d}{\d t}\widetilde{\mathcal{E}}_{l_2}^{(N)}(t)
+\widetilde{\mathcal{D}}_{l_2}^{(N)}(t)
\lesssim
\delta_0 \widetilde{\mathcal{D}}_{ N,l_2}(t)
+\mathcal{D}_{N}(t)+\eta\e^2(1+t)^{-(1+\vartheta)}\|\nabla_x^N E^{\e}\|_{L^2_x}^2.
\esp
\eal
Here, $\widetilde{\mathcal{E}}_{l_2}^{(N)}(t)$ and  $\widetilde{\mathcal{D}}_{l_2}^{(N)}(t)$ are defined in \eqref{energy functional 2-N-new}.
\end{lemma}

\begin{proof}
Our proof  of (\ref{estimate-weighted-2:N}) is divided into
two steps as follows.

\medskip
 \emph{Step 1. Weighted $L_{x,v}^2$-estimate of  $\partial^\a_x f^\e$ with the pure spatial derivative $|\a|=N$.}\;
More precisely, we can establish
\bal
\bsp\label{weighted estimate2}
&\frac{\d}{\d t} \left[(1+t)^{\sigma_{N,0}}  \mathcal{E}_{l_2}^{N,0}(t)\right]
+(1+t)^{\sigma_{N,0}}\mathcal{D}_{l_2}^{N,0}(t)\\
\lesssim\;&\eta\e^2(1+t)^{-\sigma_{N,0}}\|\nabla^N_x E^\e\|_{L^2_x}^2+(1+t)^{\sigma_{N,0}}\mathcal{D}_N(t)
+\mathcal{D}_N(t)
+\delta_0\sum_{n\leq N}(1+t)^{\sigma_{n,0}}\mathcal{D}_{l_2}^{n,0}(t)\\
\;&
+\delta_0\sum_{n\leq N}(1+t)^{\sigma_{n,1}}\mathcal{D}_{l_2}^{n,1}(t).
\esp
\eal

In fact,
applying $\partial^\alpha_x $ with $|\alpha|=N$  to   the first equation in (\ref{rVPB}), then integrating the resulting identity over $\bbR^3_x\times\bbR^3_v$ by multiplying
 $\varepsilon w^2_{l_2}(\alpha,0)\partial^\alpha_x f^\varepsilon$, we have
\bal
\bsp\label{with weight 2}
&\!\frac{\varepsilon}{2} \frac{\d}{\d t}   \left\| \widetilde{w}_{l_2}(\alpha,\!0) \partial^\alpha_x\! f^\varepsilon \right\|^2
\!+\!
 \frac{\e q\vartheta}{(1+t)^{1+\vartheta}}\left\|\langle v\rangle \widetilde{w}_{l_2}(\alpha,0) \partial^\alpha_x \!f^\varepsilon \right\|^2
 \!+\!\frac{\sigma_0}{\varepsilon}  \left\| \widetilde{w}_{l_2}(\alpha,0)\partial^\alpha_x (\mathbf{I}\!-\!\mathbf{P}) f^\varepsilon\!\right\|_{{\nu}}^2 \\
=& {\Big\langle  \partial_x^\a (E^\e\cdot v \sqrt{\mu}) , \e \widetilde{w}^{2}_{l_2}(\a,0)\partial_x^\a f^\e
\Big\rangle_{L_{x,v}^2}}
+{\Big\langle
\partial_x^\a \left(-v\times B^\e \cdot\nabla_v f^\e
\right) , \e \widetilde{w}^{2}_{l_2}(\a,0)\partial_x^\a f^\e\Big
\rangle_{L_{x,v}^2}}\\
&\;
+{\Big\langle
\partial_x^\a\left(-\e E^\e\cdot\nabla_v f^\e
\right) , \e \widetilde{w}^{2}_{l_2}(\a,0)\partial_x^\a f^\e\Big
\rangle_{L_{x,v}^2}}+{\Big\langle
\partial_x^\a\Big(
\frac{\e}{2}v\cdot E^\e  f^\e\Big) , \e \widetilde{w}^{2}_{l_2}(\a,0)\partial_x^\a f^\e\Big
\rangle_{L_{x,v}^2}}\\
&\;
+{\Big\langle\frac{1}{\e} \partial_x^\a\Gamma(f^\e,f^\e), \e \widetilde{w}^{2}_{l_2}(\a,0)\partial_x^\a f^\e\Big
\rangle_{L_{x,v}^2}}+\frac{1}{\varepsilon}  \left\| \partial^\alpha_x (\mathbf{I} - \mathbf{P}) f^\varepsilon\right\|_{L^2_{x,v}\!(\nu)}^2
+\mathcal{D}_N(t).
\esp
\eal
Here, we    have used the spectral inequality \eqref{L coercive2}.
Now   we classify the first five terms on the right-hand side  of \eqref{diyigeweight:L^2:2} as $\mathcal{J}_1$ to $\mathcal{J}_5$ and estimate them term by term.
For the term $\mathcal{J}_1$, it is derived from the H\"{o}lder inequality and the Cauchy--Schwarz inequality with $\eta$ that
\bals
\bsp
\mathcal{J}_1
\lesssim\; \e\|\nabla^N_x E^\e\|_{L^2_x}\|\nabla^N_x f^\e\|_{L^2_{x,v}}
\lesssim\; \eta\e^2(1+t)^{-\sigma_{N,0}}\|\nabla^N_x E^\e\|_{L^2_x}^2+(1+t)^{\sigma_{N,0}}\mathcal{D}_N(t).
\esp
\eals
From the macro-micro decomposition \eqref{fdefenjie},  we deduce
\bals
\bsp
\mathcal{J}_2
=\;&-\sum_{1\leq|\a'|\leq N}\Big\langle
  v\times  \partial_x^{\a'} B^\e \cdot\nabla_v \partial_x^{\a-\a'} (\mathbf{I}-\mathbf{P})f^\e, \e \widetilde{w}^{2}_{l_2}(N,0)\partial_x^\a (\mathbf{I}-\mathbf{P}) f^\e\Big
\rangle_{L_{x,v}^2}\\
\;&-\sum_{1\leq|\a'|\leq N}\Big\langle
  v\times  \partial_x^{\a'} B^\e \cdot\nabla_v \partial_x^{\a-\a'} (\mathbf{I}-\mathbf{P})f^\e, \e \widetilde{w}^{2}_{l_2}(N,0)\partial_x^\a \mathbf{P}f^\e\Big
\rangle_{L_{x,v}^2}\\[0.1mm]
\;&-\sum_{1\leq|\a'|\leq N}\Big\langle
  v\times  \partial_x^{\a'} B^\e \cdot\nabla_v \partial_x^{\a-\a'}  \mathbf{P}f^\e, \e \widetilde{w}^{2}_{l_2}(N,0)\partial_x^\a f^\e\Big
\rangle_{L_{x,v}^2}\\
\equiv\;&\mathcal{J}_{2,1}+\mathcal{J}_{2,2}+\mathcal{J}_{2,3}.
\esp
\eals
From the H\"{o}lder inequality, the Cauchy--Schwarz inequality with $\eta$, the interpolation inequality, \eqref{NGinequality}  and the Sobolev embedding $H^1(\mathbb{R}^3)\hookrightarrow L^3(\mathbb{R}^3)$, $H^1(\mathbb{R}^3)\hookrightarrow L^6(\mathbb{R}^3)$, we deduce
\bals
\bsp
\mathcal{J}_{2,1}\lesssim \;&{\e}\|\widetilde{w}_{l_2}(N,0)\partial_x^{\a}(\mathbf{I}-\mathbf{P})f^\e \|_{L^2_{x,v}}\Big[\|\nabla_xB^\e\|_{L^\infty_x}
\|\langle v\rangle^{\frac{1}{2}}\widetilde{w}_{l_2}(N-1,1)
\nabla_v\partial_x^{N-1}(\mathbf{I}-\mathbf{P})f^\e \|_{L^2_{x,v}}
\\
\;&\qquad\qquad\qquad
+\sum_{|\a'|=2}\|\partial_x^{\a'}B^\e\|_{L^\infty_x}
\|\widetilde{w}_{l_2}(N-2,1)
\nabla_v\partial_x^{\a-\a'}(\mathbf{I}-\mathbf{P})f^\e \|_{L^2_{x,v}}
\\
\;&\qquad\qquad\qquad
+\sum_{|\a'|=3}\|\partial_x^{\a'}B^\e\|_{L^3_x}
\|\widetilde{w}_{l_2}(N-2,1)
\nabla_v\partial_x^{\a-\a'}(\mathbf{I}-\mathbf{P})f^\e \|_{L^6_{x}L^2_{v}}
\\
\;&\qquad\qquad\qquad
+\sum_{4\leq|\a'|\leq N}\|\partial_x^{\a'}B^\e\|_{L^2_x}
\|\widetilde{w}_{l_2}(N-2,1)
\nabla_v\partial_x^{\a-\a'}(\mathbf{I}-\mathbf{P})f^\e \|_{L^\infty_{x}
L^2_{v}}
\Big]
\\
\lesssim \;&\frac{\eta}{\e}
\|\widetilde{w}_{l_2}(\a,0)\partial_x^\a(\mathbf{I}-\mathbf{P})f^\e\|_{L^2_{x,v}}^2
+\e^3\|\nabla_xB^\e\|_{L^\infty_x}^2\|\langle v\rangle^{\frac{1}{2}}\widetilde{w}_{l_2}(N-1,1)
\nabla_v\partial_x^{N-1}(\mathbf{I}-\mathbf{P})f^\e \|_{L^2_{x,v}}^2
\\\;&+
\e^5\|\nabla_x B^\e\|_{H^{N-1}_x}^2\sum_{n\leq N-1}(1+t)^{\sigma_{n,1}}\mathcal{D}_{l_2}^{n,1}(t)
\\
\lesssim \;&
\eta(1+t)^{\sigma_{N,0}}\mathcal{D}_{l_2}^{N,0}(t)
+\delta_0(1+t)^{\sigma_{N,1}}
\mathcal{D}_{l_2}^{N,1}(t)+ \delta_0\sum_{n\leq N-1}(1+t)^{\sigma_{n,1}}\mathcal{D}_{l_2}^{n,1}(t).
\esp
\eals
Here, we have used \eqref{priori assumption} and the facts  that
\bals
\bsp
\;& \langle v\rangle  \widetilde{w}_{l_2}(N,0)\lesssim  \langle v\rangle^{\frac{1}{2}}\widetilde{w}_{l_2}(N-1, 1),\quad
 \langle v\rangle  \widetilde{w}_{l_2}(N,0)\lesssim \widetilde{w}_{l_2}(N-2, 1)\langle v\rangle^{-\frac{7}{2}}.
\esp
\eals
The terms $\mathcal{J}_{2,2}$ and $\mathcal{J}_{2,3}$ can be estimated as
\bals
\bsp
\;&\mathcal{J}_{2,2}+\mathcal{J}_{2,3}\\
\lesssim \;&
\e\Big[\!\|\nabla_x B^\e\|_{L^\infty_x}\!\!\!\sum_{|\a'|=1}\!\!
\|\partial_x^{\a-\a'}\!(\mathbf{I}\!-\!\mathbf{P})f^\e \|_{L^2_{x,v}}
\!+\!\!\!\sum_{2\leq|\a'|\leq N}\!\!\!\|\partial_x^{\a'}B^\e\|_{L^3_x}
\|\partial_x^{\a-\a'}\!(\mathbf{I}\!-\!\mathbf{P})f^\e \|_{L^\infty_{x}L^2_{v}}\!\Big]
\| \partial_x^{\a}\mathbf{P} f^\e\! \|_{L^2_{x,v}}\\
\;&+\e\Big[\|\nabla_x B^\e\|_{L^\infty_x}
\!\sum_{|\a'|=1}\!\|\nabla_v\partial_x^{\a-\a'}\mathbf{P}f^\e \|_{L^2_{x,v}}
\!\!+\!\!\!\!\sum_{2\leq|\a'|\leq N}\!\!\|\partial_x^{\a'}B^\e\|_{L^2_x}
\|\nabla_v\partial_x^{\a-\a'}\mathbf{P}f^\e \|_{L^\infty_{x}L^2_{v}}\Big]
\|\partial_x^{\a} f^\e \|_{L^2_{x,v}}\\
\lesssim \;&\e\sqrt{\mathcal {E}_N(t)}\mathcal{D}_N(t)
\esp
\eals
by making use of the  H\"{o}lder inequality  and the Sobolev embedding theory.
Therefore, collecting the previous  estimates of the  terms $\mathcal{J}_{2,1}$, $\mathcal{J}_{2,2}$ and $\mathcal{J}_{2,3}$, we have
\bals
\bsp
\mathcal{J}_{2}\lesssim \;&
\eta(1+t)^{\sigma_{N,0}}\mathcal{D}_{l_2}^{N,0}(t)
+\delta_0\sum_{n\leq N}(1+t)^{\sigma_{n,1}}\mathcal{D}_{l_2}^{n,1}(t)+
\delta_0\mathcal{D}_N(t).
\esp
\eals

Direct calculation gives us that
\bals
\bsp
\mathcal{J}_3=\;&\underbrace{\Big\langle
-\e E^\e\cdot\nabla_v \partial_x^\a f^\e, \e \widetilde{w}^{2}_{l_2}(\a,0)\partial_x^\a f^\e\Big
\rangle_{L_{x,v}^2}}_{\mathcal{J}_{3,1}}\\
\;&+\underbrace{\e^2
\sum_{1\leq|\a'|\leq N}\Big\langle
- \partial_x^{\a'} E^\e\cdot\nabla_v \partial_x^{\a-\a'}(\mathbf{I}-\mathbf{P})f^\e, \widetilde{w}^{2}_{l_2}(\a,0)\partial_x^\a(\mathbf{I}-\mathbf{P}) f^\e\Big
\rangle_{L_{x,v}^2}}_{\mathcal{J}_{3,2}}\\
\;&+
\underbrace{\e^2
\sum_{1\leq|\a'|\leq N}\Big\langle
- \partial_x^{\a'} E^\e\cdot\nabla_v \partial_x^{\a-\a'}\mathbf{P}f^\e, \widetilde{w}^{2}_{l_2}(\a,0)\partial_x^\a(\mathbf{I}-\mathbf{P}) f^\e\Big
\rangle_{L_{x,v}^2}}_{\mathcal{J}_{3,3}}\\
\;&
+\underbrace{\e^2 \sum_{1\leq|\a'|\leq N}\Big\langle
  \partial_x^{\a'} E^\e \partial_x^{\a-\a'}f^\e,  \nabla_v\left( \widetilde{w}^{2}_{l_2}(\a,0)\partial_x^\a\mathbf{P}f^\e\right)\Big
\rangle_{L_{x,v}^2}}_{\mathcal{J}_{3,4}}.
\esp
\eals
Notice that $\nabla_v \widetilde{w}_{l_2}(\a,0)\lesssim \langle v\rangle \widetilde{w}_{l_2}(\a,0)$,
and hence, we have
\bals
\bsp
\mathcal{J}_{3,1}
\lesssim \;& \e^2\|E^\e\|_{L^\infty_x}\Big[
\|\langle v\rangle \widetilde{w}_{l_2}(\a,0)\partial_x^{\a}f^\e \|_{L^2_{x,v}}\| \widetilde{w}_{l_2}(\a,0)\partial_x^{\a} (\mathbf{I}\!-\!\mathbf{P})f^\e \|_{L^2_{x,v}}
\!+\!\|\partial_x^{\a}f^\e \|_{L^2_{x,v}}\!\| \partial_x^{\a}\mathbf{P}f^\e \|_{L^2_{x,v}}\Big]\\
\lesssim \;&\e^2\|E^\e\|_{L^\infty_x}(1+t)^{\frac{1+\vartheta}{2}}(1+t)^{\sigma_{N,0}}\mathcal{D}_{l_2}^{N,0}(t)+\e^2\sqrt{\mathcal {E}_N(t)}\mathcal{D}_N(t).
\esp
\eals
Then we  deduce from    the H\"{o}lder inequality,  the Cauchy--Schwarz inequality with $\eta$ and  the Sobolev embedding \eqref{NGinequality} that
\bals
\bsp
\mathcal{J}_{3,2}
\lesssim \;&\frac{\eta}{\e}
\|\widetilde{w}_{l_2}(\a,\!0)\partial_x^\a(\mathbf{I}\!-\!\mathbf{P}) f^\e\|_{L^2_{x,v}}^2+\e^5\!\!
\sum_{1\leq|\a'|\leq 2}\!\!\|\partial_x^{\a'}\! E^\e\|_{L^\infty_x}^2
\|\widetilde{w}_{l_2}(|\a\!-\!\a'|,1)\nabla_v \partial_x^{\a-\a'}\!(\mathbf{I}\!-\!\mathbf{P})f^\e\!\|_{L^2_{x,\!v}}^2\\
\;&+\e^5\!\!\!
 \sum_{3\leq|\a'|\leq N}\!\!\|\partial_x^{\a'} E^\e\|_{L^2_x}^2
\|\widetilde{w}_{l_2}(|\a-\a'|+2,1)\nabla_v \partial_x^{\a-\a'}(\mathbf{I}\!-\!\mathbf{P})f^\e\|_{L^\infty_x L^2_v}^2\\
\lesssim \;&\frac{\eta}{\e}
\|\widetilde{w}_{l_2}(\a,0)\partial_x^\a(\mathbf{I}-\mathbf{P}) f^\e\|_{L^2_{x,v}}^2+\e^7\|\nabla_xE^\e\|_{H^{N-1}_x}^2\sum_{n\leq N}(1+t)^{\sigma_{n,1}}\mathcal{D}_{l_2}^{n,1}(t)
\\
\lesssim \;&\eta(1+t)^{\sigma_{N,0}}\mathcal{D}_{l_2}^{N,0}(t)+ \e^7\|\nabla_xE^\e\|_{H^{N-1}_x}^2\sum_{n\leq N}(1+t)^{\sigma_{n,1}}\mathcal{D}_{l_2}^{n,1}(t)
.
\esp
\eals
Via the  H\"{o}lder inequality  and the Sobolev embedding theory, we have
\bals
\bsp
\;&\mathcal{J}_{3,3}+\mathcal{J}_{3,4}\\
\lesssim\;&\e^2\Big[\|\nabla_x E^\e\|_{L^\infty_x}
\| \partial_x^{N-1}\mathbf{P}f^\e\|_{L^2_{x,v}}+
\sum_{2\leq|\a'|\leq N}\|\partial_x^{\a'} E^\e\|_{L^2_x}
\| \partial_x^{\a-\a'}\mathbf{P}f^\e\|_{L^\infty_x L^2_v}\Big]
\|\partial_x^\a(\mathbf{I}-\mathbf{P}) f^\e\|_{L^2_{x,v}}\\
\;&+\e^2\Big[\|\nabla_x E^\e\|_{L^\infty_x}
\| \partial_x^{N-1}f^\e\|_{L^2_{x,v}}+
\sum_{2\leq|\a'|\leq N}\|\partial_x^{\a'} E^\e\|_{L^2_x}
\| \partial_x^{\a-\a'}f^\e\|_{L^\infty_x L^2_v}\Big]
\|\partial_x^\a\mathbf{P} f^\e\|_{L^2_{x,v}}\\
\lesssim \;&\e^2\sqrt{\mathcal {E}_N(t)}\mathcal{D}_N(t).
\esp
\eals
 Therefore, collecting the estimates on the quantities
$\mathcal{J}_{3,i}(i=1,2,3,4)$ in the previous and using  the   \emph{a priori} assumption  \eqref{priori assumption}, we obtain
\bals
\bsp
\mathcal{J}_{3}
\lesssim \;& (\eta+\delta_0)(1+t)^{\sigma_{N,0}}\mathcal{D}_{l_2}^{N,0}(t)+\delta_0\sum_{n\leq N}(1+t)^{\sigma_{n,1}}\mathcal{D}_{l_2}^{n,1}(t)
+\delta_0\mathcal{D}_N(t)
.
\esp
\eals

Applying the analogous argument for the term $\mathcal{J}_{2}$, we can deduce that
\bals
\bsp
\mathcal{J}_{4}
\lesssim\;& \eta(1+t)^{\sigma_{N,0}}\mathcal{D}_{l_2}^{N,0}(t)+\delta_0\sum_{n\leq N}(1+t)^{\sigma_{n,0}}\mathcal{D}_{l_2}^{n,0}(t)+
\delta_0\mathcal{D}_N(t)
.
\esp
\eals
From \eqref{hard gamma1}, the \emph{a priori} assumption  \eqref{priori assumption}  and  the Cauchy--Schwarz inequality, we obtain
\bals
\bsp
\mathcal{J}_{5}
\lesssim\;& \frac{\eta}{\e}
\|\widetilde{w}_{l_2}(\a,0)\partial_x^\a(\mathbf{I}-\mathbf{P}) f^\e\|_{L^2_{x,v}(\nu)}^2
+\e^2{\overline{\mathcal{E}}}_{N-1,l_1}(t)
\sum_{n\leq N}(1+t)^{\sigma_{n,0}}\mathcal{D}_{l_2}^{n,0}(t)
+\sqrt{\mathcal{E}_N(t)}\mathcal{D}_N(t)
\\
\lesssim\;& \eta(1+t)^{\sigma_{N,0}}\mathcal{D}_{l_2}^{N,0}(t)+\delta_0\sum_{n\leq N}(1+t)^{\sigma_{n,0}}\mathcal{D}_{l_2}^{n,0}(t)+
\delta_0\mathcal{D}_N(t)
.
\esp
\eals
In summary, putting all the estimates of $\mathcal{J}_{1}\sim\mathcal{J}_{5}$
into \eqref{with weight 2} and choosing $\eta$ small enough, \eqref{weighted estimate2} follows.

\medskip

 \emph{Step 2. Weighted $L_{x,v}^2$-estimate of  $\partial^\a_\b f^\e$ with the mixed derivative $|\a|+|\b|=N$, $|\b|=m\geq 1$.}\;
To be more precise, we can obtain
\bal
\bsp\label{weighted estimate3}
& \frac{\d}{\d t}
\left[(1+t)^{\sigma_{N,|\b|}}\mathcal{D}_{l_2}^{N,|\b|}(t)\right]
+(1+t)^{\sigma_{N,|\b|}}\mathcal{D}_{l_2}^{N,|\b|}(t)\\
\lesssim\;&\delta_0\sum_{n\leq N}(1+t)^{\sigma_{n,|\b|}}\mathcal{D}_{l_2}^{n,|\b|}(t)
+(1+t)^{\sigma_{N,|\b|}}\mathcal{D}_{l_2}^{N,|\b|-1}(t)
+\delta_0(1+t)^{\sigma_{N,|\b|}}\mathcal{D}_{l_2}^{N,|\b|+1}(t)
\\
\;&
+\delta_0\sum_{n\leq N-1}(1+t)^{\sigma_{n,|\b|-1}}\mathcal{D}_{l_2}^{n,
|\b|-1}(t)
 +\delta_0\sum_{n\leq N-1}(1+t)^{\sigma_{n,|\b|+1}}\mathcal{D}_{l_2}^{n,|\b|+1}(t)+\mathcal{D}_N(t)
 .
\esp
\eal

Indeed, applying $\partial_{\b}^\a$ with $|\a|+|\b|=N, |\b|=m\geq 1$ to (\ref{rVMBweiguan-1}) and then integrating the resulting identity over $\bbR^3_x\times\bbR^3_v$ by multiplying $\widetilde{w}^{2}_{l_2}(\a,\b)\partial_{\b}^\a(\mathbf{I}-\mathbf{P})f^\e$, we  obtain
\bal
\bsp\label{diergeweight:L^2:2}
& \frac{1}{2}\frac{\d }{\d t} \|\widetilde{w}_{l_2}(\a,\b)\partial_{\b}^\a(\mathbf{I}-\mathbf{P})f^\e\|_{L_{x,v}^2}^2
+\frac{q\vartheta}{(1+t)^{1+\vartheta}}
\|\langle v\rangle \widetilde{w}_{l_2}(\a,\b)
\partial_{\b}^\a(\mathbf{I}-\mathbf{P})f^\e\|_{L_{x,v}^2}^2\\
\;&
+\frac{\sigma_0}{\e^2}\| \widetilde{w}_{l_2}(\a,\b)\partial_{\b}^\a(\mathbf{I}-\mathbf{P})f^\e\|_{L_{x,v}^2(\nu)}^2
\\
\le
&\;\frac{\eta}{\e^2}\sum_{|\b'|\leq|\b|}\| \widetilde{w}_{l_2}(\a,\b')\partial_{\b'}^\a(\mathbf{I}-\mathbf{P})f^\e\|_{L_{x,v}^2(\nu)}^2
+\frac{C_{\sigma_0}}{\e^2}\| \partial^\a_x(\mathbf{I}-\mathbf{P})f^\e\|_{L_{x,v}^2(\nu)}^2\\
&\;+{\Big\langle- \frac{1}{\e}C_{\b}^{e_i}\partial_{\b-e_i}^{\a+e_i} (\mathbf{I}-\mathbf{P}) f^\e , \widetilde{w}^{2}_{l_2}(\a,\b)\partial_{\b}^\a(\mathbf{I}-\mathbf{P})f^\e
\Big\rangle_{L_{x,v}^2}}\\
&\;
+{\Big\langle
\partial_{\b}^\a\Big(-v\times B^\e \cdot\nabla_v(\mathbf{I}-\mathbf{P}) f^\e
\Big) , \widetilde{w}^{2}_{l_2}(\a,\b)\partial_{\b}^\a(\mathbf{I}-\mathbf{P})f^\e\Big
\rangle_{L_{x,v}^2}}\\
&\;+{\Big\langle
\partial_{\b}^\a\Big(-\e E^\e \cdot\nabla_v(\mathbf{I}-\mathbf{P}) f^\e+
\frac{\e}{2}v\cdot E^\e (\mathbf{I}-\mathbf{P}) f^\e\Big) , \widetilde{w}^{2}_{l_2}(\a,\b)\partial_{\b}^\a(\mathbf{I}-\mathbf{P})f^\e\Big
\rangle_{L_{x,v}^2}}
\\
&+{\Big\langle\frac{1}{\e} \partial_{\b}^\a\Gamma(f^\e,f^\e), \widetilde{w}^{2}_{l_2}(\a,\b)\partial_{\b}^\a(\mathbf{I}-\mathbf{P})f^\e\Big
\rangle_{L_{x,v}^2}}\\
\;&
+{\Big\langle \partial_{\b}^\a\big([[\mathbf{P},\mathcal{A}_{(E^\e,B^\e)}]]f^\e\big),  \widetilde{w}^{2}_{l_2}(\a,\b)\partial_{\b}^\a(\mathbf{I}-\mathbf{P})f^\e\Big \rangle_{L_{x,v}^2}.}
\esp
\eal
Here, we    have used the spectral inequality \eqref{L coercive3}.
 Now   we classify the last five terms on the
right-hand side of (\ref{diergeweight:L^2:2}) as $\mathcal{J}_6 $ to $\mathcal{J}_{10}$ and   estimate them  term by term.
The H\"{o}lder inequality, the Cauchy--Schwarz inequality with $\eta$ and the interpolation inequality lead  to
\bals
\bsp
\mathcal{J}_6\lesssim \;&\frac{1}{\e}\|\widetilde{w}_{l_2}(\a,\b)\partial_{\b}^\a(\mathbf{I}-\mathbf{P})f^\e\|_{L^2_{x,v}}
\|\langle v \rangle^{\frac{1}{2}}\widetilde{w}_{l_2}(\a+e_i,\b-e_i)\partial_{\b-e_i}^{\a+e_i} (\mathbf{I}-\mathbf{P}) f^\e\|_{L^2_{x,v}} \\
\lesssim \;&\frac{\eta}{\e^2}
\|\widetilde{w}_{l_2}(\a,\b)\partial_{\b}^\a(\mathbf{I}-\mathbf{P})f^\e\|_{L^2_{x,v}}^2
+ C_{\eta}\|\langle v \rangle^{\frac{1}{2}}\widetilde{w}_{l_2}(\a+e_i,\b-e_i)\partial_{\b-e_i}^{\a+e_i} (\mathbf{I}-\mathbf{P}) f^\e\|_{L^2_{x,v}}^2\\
\lesssim \;&
\eta(1+t)^{\sigma_{N,|\b|}}\mathcal{D}_{l_2}^{N,|\b|}(t)
+ C_{\eta}(1+t)^{\frac{1+\vartheta}{2}}(1+t)^{\sigma_{N,|\b|-1}}\mathcal{D}_{l_2}^{N,|\b|-1}(t)\\
\lesssim \;&
\eta(1+t)^{\sigma_{N,|\b|}}\mathcal{D}_{l_2}^{N,|\b|}(t)
+
(1+t)^{\sigma_{N,|\b|}}\mathcal{D}_{l_2}^{N,|\b|-1}(t)
,
\esp
\eals
where we have used \eqref{sigma:def} and the fact  that $\widetilde{w}_{l_2}(\a,\b)=\widetilde{w}_{l_2}(\a+e_i,\b-e_i)\langle v \rangle^{\frac{1}{2}}$.

For the term $\mathcal{J}_7$,  we split it into
\bals
\bsp
\mathcal{J}_7=
\;&
\sum_{|\a'|\leq|\a|,|\b'|=1}\Big\langle
 -\partial_{\b_1}v\times \partial_x^{\a'} B^\e \cdot\nabla_v\partial_{\b-\b'}^{\a-\a'}(\mathbf{I}-\mathbf{P}) f^\e, \widetilde{w}^{2}_{l_2}(\a,\b)\partial_{\b}^\a(\mathbf{I}-\mathbf{P})f^\e
 \Big
\rangle_{L^2_{x,v}}\\
\;&+
\sum_{1\leq|\a'|\leq|\a|}\Big\langle
 -v\times \partial_x^{\a'} B^\e \cdot\nabla_v\partial_{\b}^{\a-\a'}(\mathbf{I}-\mathbf{P}) f^\e, \widetilde{w}^{2}_{l_2}(\a,\b)\partial_{\b}^\a(\mathbf{I}-\mathbf{P})f^\e
 \Big
\rangle_{L^2_{x,v}}\\
\equiv\;&\mathcal{J}_{7,1}+\mathcal{J}_{7,2}.
\esp
\eals
For   $\mathcal{J}_{7,1}$,   we   deduce by the H\"{o}lder inequality, \eqref{NGinequality} and the \emph{a priori} assumption  \eqref{priori assumption}   that
\bals
\bsp
\mathcal{J}_{7,1}
\lesssim\;&\Big[\sum_{|\a'|\leq|\a|-1,|\b'|=1}\|\partial_x^{\a'} B^\e \|_{L^\infty_x} \|\widetilde{w}_{l_2}(\a-\a',\b)\nabla_v\partial_{\b-\b'}^{\a-\a'}(\mathbf{I}-\mathbf{P}) f^\e \|_{L^2_{x,v}}\\
\;&\quad+\|\partial_x^{\a} B^\e \|_{L^3_x} \|\widetilde{w}_{l_2}(1,|\b|)\nabla_v\partial_{\b-\b'}(\mathbf{I}-\mathbf{P}) f^\e \|_{L^6_{x}L^2_{v}}\Big]
\| \widetilde{w}_{l_2}(\a,\b)\partial_{\b}^{\a}(\mathbf{I}-\mathbf{P}) f^\e\|_{L^2_{x,v}}\\
\lesssim \;& \e^2\|\nabla_xB^\e\|_{H^{N-1}_x}\sum_{n\leq N}(1+t)^{\sigma_{n,|\b|}}\mathcal{D}_{l_2}^{n,|\b|}(t)\\
\lesssim \;& \delta_0\sum_{n\leq N}(1+t)^{\sigma_{n,|\b|}}\mathcal{D}_{l_2}^{n,|\b|}(t).
\esp
\eals
Here, we have   used   the fact $|\a|\leq N-1$ because of  $|\b|\geq1$ and $|\a|+|\b|= N$ in this case.
It is worth pointing out that the most subtle part is the estimate of $\mathcal{J}_{7,2}$ which is carried out as
\bal
\bsp\label{Bdekunnan:3}
\mathcal{J}_{7,2}
\lesssim\;&\Big[\|\nabla_x B^\e \|_{L^\infty_x}\sum_{|\a'|=1} \|\langle v\rangle^{\frac{1}{2}} \widetilde{w}_{l_2}(|\a|-1,|\b|+1)\nabla_v\partial_{\b}^{\a-\a'}(\mathbf{I}-\mathbf{P}) f^\e \|_{L^2_{x,v}}
\\
\;&\;\;+\sum_{2\leq|\a'|\leq\max\{|\a|-1,2\}}\|\partial_x^{\a'} B^\e \|_{L^\infty_x} \|\widetilde{w}_{l_2}(|\a-\a'|,|\b|+1)\nabla_v\partial_{\b}^{\a-\a'}(\mathbf{I}-\mathbf{P}) f^\e \|_{L^2_{x,v}}\\
\;&\;\;+\|\partial_x^{\a} B^\e \|_{L^3_x} \|\widetilde{w}_{l_2}(1,\b+1)\nabla_v\partial_{\b}(\mathbf{I}-\mathbf{P}) f^\e \|_{L^6_{x}L^2_{v}}\Big]\| \widetilde{w}_{l_2}(\a,\b)\partial_{\b}^{\a}(\mathbf{I}-\mathbf{P}) f^\e\|_{L^2_{x,v}}\\
\lesssim \;&\e^2\|\nabla_x B^\e \|_{L^\infty_x}^2 \sum_{|\a'|=1}\|\langle v\rangle^{\frac{1}{2}} \widetilde{w}_{l_2}(|\a|-1,|\b|+1)\nabla_v\partial_{\b}^{\a-\a'}(\mathbf{I}-\mathbf{P}) f^\e \|_{L^2_{x,v}}^2\\
\;&+ \e^4
\| B^\e\|_{H^{N}_x}^2\sum_{n\leq N-1}(1+t)^{\sigma_{n,|\b|+1}}\mathcal{D}_{l_2}^{n,|\b|+1}(t)+\frac{\eta}{\e^2} \| \widetilde{w}_{l_2}(\a,\b) \partial_{\b}^{\a}(\mathbf{I}-\mathbf{P}) f^\e \|_{L^2_{x,v}}^2
\\
\lesssim \;&\delta_0 (1\!+\!t)^{\sigma_{N,|\b|}}\mathcal{D}_{l_2}^{N,|\b|+1}\!(t)\!+\!
 \delta_0\!\!\!\sum_{n\leq N-1}\!\!\!(1\!+\!t)^{\sigma_{n,|\b|+1}}\mathcal{D}_{l_2}^{n,|\b|+1}\!(t)
 \!+\! \eta(1\!+\!t)^{\sigma_{N,|\b|}}\mathcal{D}_{l_2}^{N,|\b|}\!(t).
\esp
\eal
Here, we have used \eqref{priori assumption} and the fact, which is derived from the  interpolation inequality and \eqref{sigma:def}, that
\bals
&\sum_{|\a'|=1}\|\langle v\rangle^{\frac{1}{2}} \widetilde{w}_{l_2}(|\a|-1,|\b|+1)\nabla_v\partial_{\b}^{\a-\a'}(\mathbf{I}-\mathbf{P}) f^\e \|_{L^2_{x,v}}^2
\\
\lesssim \;&\!\!\sum_{|\a'|=1}\!\!\|\langle v\rangle \widetilde{w}_{l_2}(|\a|\!-1\!,\!|\b|\!+\!1)\nabla_v\partial_{\b}^{\a-\a'}\!(\mathbf{I}\!-\!\mathbf{P}) f^\e\! \|_{L^2_{x,v}} \| \widetilde{w}_{l_2}(|\a|\!-1\!,\!|\b|\!+\!1)\nabla_v\partial_{\b}^{\a-\a'}\!(\mathbf{I}\!-\!\mathbf{P}) f^\e\! \|_{L^2_{x,v}}\\
\lesssim \;& \e(1+t)^{\frac{1+\vartheta}{2}}
(1+t)^{\sigma_{N,|\b|+1}}\mathcal{D}_{l_2}^{N,|\b|+1}(t)\\
\lesssim \;& \e(1+t)^{{1+\vartheta}}
(1+t)^{\sigma_{N,|\b|}}\mathcal{D}_{l_2}^{N,|\b|+1}(t).
\eals
Consequently, collecting the above bounds on the quantities $\mathcal{J}_{7,1}$ and $\mathcal{J}_{7,2}$, we obtain
\bals
\bsp
 \mathcal{J}_{7}\lesssim\;&
 \delta_0\sum_{n\leq N}(1+t)^{\sigma_{n,|\b|+1}}\mathcal{D}_{l_2}^{n,|\b|+1}(t)+\delta_0\sum_{n\leq N}(1+t)^{\sigma_{n,|\b|}}\mathcal{D}_{l_2}^{n,|\b|}(t)
 + \eta(1+t)^{\sigma_{N,|\b|}}\mathcal{D}_{l_2}^{N,|\b|}(t).
 \esp
 \eals

 For the term $\mathcal{J}_8$, direct calculation gives us that
\bals
\bsp
\mathcal{J}_8=\;&\underbrace{\Big\langle
 -\e E^\e \cdot\nabla_v\partial_{\b}^{\a}(\mathbf{I}-\mathbf{P}) f^\e
,  \widetilde{w}^{2}_{l_2}(\a,\b)\partial_{\b}^\a(\mathbf{I}-\mathbf{P})f^\e\Big
\rangle_{L^2_{x,v}}}_{\mathcal{J}_{8,1}}\\
\;&+\underbrace{\sum_{1\leq|\a'|\leq|\a| }\Big\langle
 - \e\partial_x^{\a'} E^\e \cdot\nabla_v\partial_{\b}^{\a-\a'}(\mathbf{I}-\mathbf{P}) f^\e
,  \widetilde{w}^{2}_{l_2}(\a,\b)\partial_{\b}^\a(\mathbf{I}-\mathbf{P})f^\e\Big
\rangle_{L^2_{x,v}}}_{\mathcal{J}_{8,2}}\\
\;&+\underbrace{\sum_{|\a'|\leq|\a|}\Big\langle
\frac{\e}{2}v\cdot \partial^{\a'}_x E^\e \partial_{\b}^{\a-\a'}(\mathbf{I}-\mathbf{P}) f^\e , \widetilde{w}^{2}_{l_2}(\a,\b)\partial_{\b}^\a(\mathbf{I}-\mathbf{P})f^\e\Big
\rangle_{L_{x,v}^2}}_{\mathcal{J}_{8,3}}
\\
\;&+\underbrace{\sum_{|\a'|\leq|\a|}\Big\langle
\frac{\e}{2}\partial^{\a'}_x E^\e_i \partial_{\b-e_i}^{\a-\a'}(\mathbf{I}-\mathbf{P}) f^\e , \widetilde{w}^{2}_{l_2}(\a,\b)\partial_{\b}^\a(\mathbf{I}-\mathbf{P})f^\e\Big
\rangle_{L_{x,v}^2}}_{\mathcal{J}_{8,4}}.
\esp
\eals
 Notice that
 $\nabla_v \widetilde{w}_{l_2}(\a,\b)\lesssim \langle v\rangle \widetilde{w}_{l_2}(\a,\b)$,
and hence, we have
\bals
\bsp
\mathcal{J}_{8,1}=
\;&\Big\langle
  \e E^\e \cdot\frac{\nabla_v \widetilde{w}_{l_2}(\a,\b)}{\widetilde{w}_{l_2}(\a,\b)}\partial_{\b}^{\a}(\mathbf{I}-\mathbf{P}) f^\e
,
\widetilde{w}^{2}_{l_2}(\a,\b)\partial_{\b}^\a(\mathbf{I}-\mathbf{P})f^\e\Big
\rangle_{L^2_{x,v}}\\
\lesssim\;&\e\| E^\e \|_{L^\infty_x}
\| \langle v\rangle \widetilde{w}_{l_2}(\a,\b)\partial_{\b}^{\a}(\mathbf{I}-\mathbf{P}) f^\e\|_{L^2_{x,v}}
\| \widetilde{w}_{l_2}(\a,\b)\partial_{\b}^{\a}(\mathbf{I}-\mathbf{P}) f^\e\|_{L^2_{x,v}}\\
\lesssim\;&\e^2\|\nabla_xE^\e\|_{H^1_x}(1+t)^{\frac{1+\vartheta}{2}}
 (1+t)^{\sigma_{N,|\b|}}\mathcal{D}_{l_2}^{N,|\b|}(t).
\esp
\eals
Applying the H\"{o}lder inequality, \eqref{NGinequality}, the Sobolev embedding $H^1(\mathbb{R}^3)\hookrightarrow L^3(\mathbb{R}^3)$ and $|\a|\leq N-1$, we derive
\bals
\bsp
\mathcal{J}_{8,2}\lesssim
\;&
\e \|\nabla_x E^\e \|_{L^\infty_x}\!\sum_{|\a'|=1} \! \|\widetilde{w}_{l_2}(|\a|\!-\!1,|\b|\!+\!1)\nabla_v
\partial_{\b}^{\a\!-\!\a'}\!(\mathbf{I}-\mathbf{P}) f^\e\! \|_{L^2_{x,v}}\| \widetilde{w}_{l_2}(\a,\b)\partial_{\b}^{\a}(\mathbf{I}\!-\!\mathbf{P})\! f^\e\|_{L^2_{x,v}}
\\
\;&+\e\Big[\sum_{2\leq|\a'|\leq\max\{|\a|-1,2\}}\|\partial_x^{\a'} E^\e \|_{L^\infty_x} \|\widetilde{w}_{l_2}(|\a-\a'|,|\b|+1)\nabla_v\partial_{\b}^{\a-\a'}(\mathbf{I}-\mathbf{P}) f^\e \|_{L^2_{x,v}}\\
\;&\qquad+\|\partial_x^{\a} E^\e \|_{L^3_x} \|\widetilde{w}_{l_2}(1,|\b|+1)\nabla_v\partial_{\b}(\mathbf{I}-\mathbf{P}) f^\e \|_{L^6_{x}L^2_{v}}\Big]\| \widetilde{w}_{l_2}(\a,\b)\partial_{\b}^{\a}(\mathbf{I}-\mathbf{P}) f^\e\|_{L^2_{x,v}}\\
\lesssim \;&\e^4\|\nabla_xE^\e\|_{L^{\infty}_x}^2\sum_{|\a'|=1}
\|\widetilde{w}_{l_2}(|\a|-1,|\b|+1)\nabla_v\partial_{\b}^{\a-\a'}(\mathbf{I}-\mathbf{P}) f^\e \|_{L^2_{x,v}}^2\\
\;&+
\e^4\|\nabla_xE^\e\|_{H^{N-1}_x}^2\sum_{n\leq N-1}(1+t)^{\sigma_{n,|\b|+1}}\mathcal{D}_{l_2}^{n,|\b|+1}(t)+
\frac{\eta}{\e^2}
\|\widetilde{w}_{l_2}(\a,\b)\partial_{\b}^\a(\mathbf{I}-\mathbf{P})f^\e\|_{L^2_{x,v}}^2
\\
\lesssim \;&
\e^4\|\nabla_xE^\e\|_{L^{\infty}_x}^2
(1+t)^{\sigma_{N,|\b|+1}}\mathcal{D}_{l_2}^{N,|\b|+1}(t)
+\delta_0\sum_{n\leq N-1}(1+t)^{\sigma_{n,|\b|+1}}\mathcal{D}_{l_2}^{n,|\b|+1}(t)\\
\;&
+\eta(1+t)^{\sigma_{N,|\b|}}\mathcal{D}_{l_2}^{N,|\b|}\!(t)
\\
\lesssim \;&
\delta_0
(1+t)^{\sigma_{N,|\b|}}\mathcal{D}_{l_2}^{N,|\b|+1}\!(t)
+\delta_0\!\!\sum_{n\leq N-1}\!\!(1\!+\!t)^{\sigma_{n,|\b|+1}}\mathcal{D}_{l_2}^{n,|\b|+1}\!(t)
+\eta(1\!+\!t)^{\sigma_{N,|\b|}}\mathcal{D}_{l_2}^{N,|\b|}\!(t)
,
\esp
\eals
where we have used the facts that
\bals
\bsp
\;& \widetilde{w}_{l_2}(\a,\b)\lesssim \widetilde{w}_{l_2}(|\a-\a'|,|\b|+1)\langle v\rangle^{-\frac{1}{2}}\quad \text{in the case of}\quad
1\leq|\a'|\leq \max\{|\a|-1,2\},\\
\;&
 \widetilde{w}_{l_2}(\a,\b)\lesssim \widetilde{w}_{l_2}(1,|\b|+1)\quad\quad \quad\quad \qquad \; \text{in the case of}\quad  |\a|\geq3.
\esp
\eals
Then we  obtain from  the H\"{o}lder inequality,  \eqref{NGinequality},  the Sobolev embeddings theory that
\bals
\bsp
\mathcal{J}_{8,3}\lesssim\;&\e\Big[\|E^\e\|_{L^\infty_x}\|\langle v\rangle \widetilde{w}_{l_2}(\a,\b)
\partial_{\b}^\a
 (\mathbf{I}-\mathbf{P}) f^\e\|_{L^2_{x,v}}\\
 \;&\quad+ \sum_{1\leq|\a'|\leq\max\{|\a|-1,1\}}\|\partial^{\a'}_x E^\e\|_{L^\infty_x}
 \|\widetilde{w}_{l_2}(\a-\a',\b)
\partial_{\b}^{\a-\a'}
 (\mathbf{I}-\mathbf{P}) f^\e\|_{L^2_{x,v}} \\
 \;&\quad+\|\partial^{\a}_x E^\e\|_{L^3_x}
 \|\widetilde{w}_{l_2}(1,|\b|)
 \partial_{\b}(\mathbf{I}-\mathbf{P}) f^\e\|_{L^6_{x}L^2_{v}}\Big]
  \| \widetilde{w}_{l_2}(\a,\b)\partial_{\b}^\a(\mathbf{I}-\mathbf{P})f^\e\|_{L^2_{x,v}}\\
 \lesssim \;&\e^2\|\nabla_xE^\e\|_{H^1_x}(1+t)^{\frac{1+\vartheta}{2}}
 (1+t)^{\sigma_{N,|\b|}}
 \mathcal{D}_{l_2}^{N,|\b|}(t)+ \e^4\|\nabla_xE^\e\|_{H^{N-1}_x}^2\!\!\!\!\sum_{n\leq N-1}\!\!\!(1+t)^{\sigma_{n,|\b|}}\!\!\mathcal{D}_{l_2}^{n,|\b|}(t)\\
 \;&+\frac{\eta}{\e^2} \| \widetilde{w}_{l_2}(\a,\b) \partial_{\b}^{\a}(\mathbf{I}-\mathbf{P}) f^\e \|_{L^2_{x,v}}^2\\
\lesssim\;&
 \delta_0\sum_{n\leq N}(1+t)^{\sigma_{n,|\b|}}\mathcal{D}_{l_2}^{n,|\b|}(t)
 + \eta(1+t)^{\sigma_{N,|\b|}}\mathcal{D}_{l_2}^{N,|\b|}(t),
\esp
\eals
where we have used $|\a|\leq N-1$ and the facts that
\bals
\bsp
\;&\langle v\rangle \widetilde{w}_{l_2}(\a,\b)\lesssim \widetilde{w}_{l_2}(\a-\a',\b)\quad\; \text{in the case of}\quad
1\leq|\a'|\leq\max\{|\a|-1,1\},\\
\;&
\langle v\rangle \widetilde{w}_{l_2}(\a,\b)\lesssim \widetilde{w}_{l_2}(1,|\b|)\quad\quad  \quad  \text{in the case of}\quad  |\a|\geq2.
\esp
\eals
For  $\mathcal{J}_{8,4}$, thanks to  $|\a|\leq N-1$ in this case, we derive from the H\"{o}lder inequality and \eqref{NGinequality} that
\bals
\bsp
\mathcal{J}_{8,4}\lesssim\;&\e\Big[\sum_{|\a'|\leq\max\{|\a|-1,1\}}\| \partial^{\a'}_x E^\e\|_{L^\infty_x}
\|\widetilde{w}_{l_2}(\a-\a',\b-e_i)\partial_{\b-e_i}^{\a-\a'}(\mathbf{I}-\mathbf{P}) f^\e\|_{L^2_{x,v}}\\
\;&\quad
+\| \partial^{\a}_x E^\e\|_{L^3_x}
\|\widetilde{w}_{l_2}(1,|\b-e_i|)\partial_{\b-e_i}(\mathbf{I}-\mathbf{P}) f^\e\|_{L^6_{x}L^2_{v}}\Big]
\|\widetilde{w}_{l_2}(\a,\b)\partial_{\b}^\a(\mathbf{I}-\mathbf{P} )f^\e\|_{L^2_{x,v}} \\
\lesssim\;& \e^4 \|\nabla_xE^\e\|_{H^{N-1}_x}^2\sum_{n\leq N-1}(1+t)^{\sigma_{n,|\b|-1}}\mathcal{D}_{l_2}^{n,|\b|-1}(t)
+\frac{\eta}{\e^2}\|\widetilde{w}_{l_2}(\a,\b)\partial_{\b}^\a(\mathbf{I}-\mathbf{P} )f^\e\|_{L^2_{x,v}}^2\\
\lesssim\;& \delta_0\sum_{n\leq N-1}(1+t)^{\sigma_{n,|\b|-1}}\mathcal{D}_{l_2}^{n,|\b|-1}(t)+ \eta(1+t)^{\sigma_{N,|\b|}}\mathcal{D}_{l_2}^{N,|\b|}(t).
\esp
\eals
Thus, the above estimates  of $\mathcal{J}_{8,i}(i=1,2,3,4)$    give us that
\bals
\bsp
\mathcal{J}_{8}
\lesssim
\;&\delta_0
(1+t)^{\sigma_{N,|\b|}}\mathcal{D}_{l_2}^{N,|\b|+1}(t)
+\delta_0\sum_{n\leq N-1}(1+t)^{\sigma_{n,|\b|+1}}\mathcal{D}_{l_2}^{n,|\b|+1}(t)
+\eta(1+t)^{\sigma_{N,|\b|}}\mathcal{D}_{l_2}^{N,|\b|}(t)\\
\;&+\delta_0\sum_{n\leq N}(1+t)^{\sigma_{n,|\b|}}\mathcal{D}_{l_2}^{n,|\b|}(t)+
\delta_0\sum_{n\leq N-1}(1+t)^{\sigma_{n,|\b|-1}}\mathcal{D}_{l_2}^{n,|\b|-1}(t)
,
\esp
\eals

From \eqref{hard gamma1} and the Cauchy--Schwarz inequality with $\eta$, we deduce
\bals
\bsp
\mathcal{J}_{9}
\lesssim
{\overline{\mathcal{E}}}_{N-1,l_1}(t)\sum_{n\leq N}(1+t)^{\sigma_{n,|\b|}}\mathcal{D}_{l_2}^{n,|\b|}(t)+ \eta(1+t)^{\sigma_{N,|\b|}}\mathcal{D}_{l_2}^{N,|\b|}(t).
\esp
\eals
Taking the analogous argument of the term  $\mathcal{Y}_{5}$ in  Proposition \ref{weighted 1},
one can easily find that
\bals
\bsp
\mathcal{J}_{10}
\lesssim
\mathcal{D}_{N}(t)+\e\mathcal{E}_{N}(t) \mathcal{D}_{N}(t)\lesssim
\mathcal{D}_{N}(t)+\delta_0
\mathcal{D}_{N}(t).
\esp
\eals
As a consequence, with the aid of   \eqref{priori assumption}  and  substituting all the above estimates  of $\mathcal{J}_{6}\sim\mathcal{J}_{10}$
into \eqref{diergeweight:L^2:2}  and then choosing $\eta$ small enough yield  the desired estimate  \eqref{weighted estimate3}.

Consequently,
taking \eqref{weighted estimate3}  the summation over $\left\{ |\beta|=m,|\alpha|+|\beta| = N\right\}$ for each given $1 \leq m \leq N $, and then taking combination of those $N$ estimates with properly chosen constant $C_m > 0 $ $(1 \leq m \leq N )$, and then adding the resulting inequality into  \eqref{weighted estimate2},
we conclude \eqref{estimate-weighted-2:N}. This completes the proof of Lemma \ref{weighted 2:N}.
\end{proof}

Next, we provide the lower-order weighted energy estimate   with the weight   $\widetilde{w}_{l_2}(\alpha, \beta)$ for $|\a|+|\beta|=n\leq N-1$  as follows.
\begin{lemma}\label{weighted 2:diyuN}
Assume that  the \emph{a priori} assumption  \eqref{priori assumption} holds true for $\delta_0$ small enough.
Then, there holds
\bal
\bsp\label{estimate-weighted-2:diyuN}
\sum_{n\leq N-1}\frac{\d}{\d t}\widetilde{\mathcal{E}}_{l_2}^{(n)}(t)
+\sum_{n\leq N-1}\widetilde{\mathcal{D}}_{l_2}^{(n)}(t)
\lesssim
\delta_0\mathcal{D}_{N-1, l_1}(t)
+\mathcal{D}_{N}(t).
\esp
\eal
Here, $\widetilde{\mathcal{E}}_{l_2}^{(n)}(t)$ and   $\widetilde{\mathcal{D}}_{l_2}^{(n)}(t)$ are given in \eqref{energy functional 2-N-new}.
\end{lemma}
\begin{proof}
The core idea to prove Lemma \ref{weighted 2:diyuN} is parallel to \emph{Step 2} in the proof of Lemma \ref{weighted 2:N}.
In what follows, we only establish some distinct and critical estimates, and the analogous calculations will be omitted for simplicity.
Indeed, notice that     the estimates of  $\mathcal{J}_{7,2}$ and  $\mathcal{J}_{8,2}$ in the case of $|\a|+|\b|=N$ include the term
\bal\label{weight function 2-kunnan}
\delta_0\sum_{n\leq N-1}(1+t)^{\sigma_{n,|\b|+1}}\mathcal{D}_{l_2}^{n,|\b|+1}(t)=\delta_0\sum_{n\leq N-1}(1+t)^{\sigma_{N,|\b|}}\mathcal{D}_{l_2}^{n,|\b|+1}(t).
\eal
Obviously, this result \eqref{weight function 2-kunnan}  is out of control when $|\a|+|\b|=n<N$ due to \eqref{sigma:def}. Thus, we should to revise      the estimates of  $\mathcal{J}_{7,2}$ and  $\mathcal{J}_{8,2}$  as follows. Direct calculation leads to
\bals
\bsp
\mathcal{J}_{7,2} =\;&
\sum_{0\leq|\a'|\leq N_0}\Big\langle
 -v\times \partial_x^{\a'} B^\e \cdot\nabla_v\partial_{\b}^{\a-\a'}(\mathbf{I}-\mathbf{P}) f^\e, \widetilde{w}^{2}_{l_2}(\a,\b)\partial_{\b}^\a(\mathbf{I}-\mathbf{P})f^\e
 \Big
\rangle_{L^2_{x,v}}\\
\;&
+\sum_{N_0+1\leq|\a'|\leq|\a|\leq n}\Big\langle
 -v\times \partial_x^{\a'} B^\e \cdot\nabla_v\partial_{\b}^{\a-\a'}(\mathbf{I}-\mathbf{P}) f^\e, \widetilde{w}^{2}_{l_2}(\a,\b)\partial_{\b}^\a(\mathbf{I}-\mathbf{P})f^\e
 \Big
\rangle_{L^2_{x,v}}\\
\equiv\;&\mathcal{J}_{7,2}^1+\mathcal{J}_{7,2}^2.
\esp
\eals
Here $N_0\geq 3$ is given in \eqref{hard assumption}.
Applying the H\"{o}lder inequality,  the Cauchy--Schwarz inequality with $\eta$,   the Sobolev embedding $H^1(\mathbb{R}^3)\hookrightarrow L^3(\mathbb{R}^3)$ and \eqref{NGinequality}, we derive
\bals
\bsp
{\mathcal{J}_{7,2}^1}\lesssim \;& \e^2\|\nabla_x B^\e \|_{L^\infty_x}^2 \sum_{|\a'|=1}\|\langle v\rangle^{\frac{1}{2}} \widetilde{w}_{l_2}(|\a|-1,|\b|+1)\nabla_v\partial_{\b}^{\a-\a'}(\mathbf{I}-\mathbf{P}) f^\e \|_{L^2_{x,v}}^2\\
\;&
+\e^2\sum_{2\leq|\a'|\leq N_0-1}\| \partial_x^{\a'}B^\e \|_{L^3_x}^2 \|\langle v\rangle^{\frac{1}{2}} \widetilde{w}_{l_2}(|\a-\a'|+1,|\b|+1)\nabla_v\partial_{\b}^{\a-\a'}(\mathbf{I}-\mathbf{P}) f^\e \|_{L^6_{x}L^2_{v}}^2\\
\;&+\e^2\sum_{|\a'|=N_0}\| \partial_x^{\a'}B^\e \|_{L^2_x}^2 \| \widetilde{w}_{l_2}(|\a-\a'|+1,|\b|+1)\langle v\rangle^{- \frac{7}{2}}\nabla_v\partial_{\b}^{\a-\a'}(\mathbf{I}-\mathbf{P}) f^\e \|_{L^\infty_{x}L^2_{v}}^2\\
\;&+\frac{\eta}{\e^2} \| \widetilde{w}_{l_2}(\a,\b) \partial_{\b}^{\a}(\mathbf{I}-\mathbf{P}) f^\e \|_{L^2_{x,v}}^2
\\
\lesssim \;&\e^3\|\nabla_x B^\e\|_{H^{N_0-1}_x}^2(1+t)^{\frac{1+\vartheta}{2}}
 \sum_{n\leq N-1} (1+t)^{\sigma_{n,|\b|+1}}\mathcal{D}_{l_2}^{n,|\b|+1}(t)
 + \eta(1+t)^{\sigma_{n,|\b|}}\mathcal{D}_{l_2}^{n,|\b|}(t)\\
\lesssim \;&\delta_0
 \sum_{n\leq N-1}(1+t)^{\sigma_{n,|\b|}}\mathcal{D}_{l_2}^{n,|\b|+1}(t)
 + \eta(1+t)^{\sigma_{n,|\b|}}\mathcal{D}_{l_2}^{n,|\b|}(t),
\esp
\eals
where we have used  the \emph{a priori} assumption  \eqref{priori assumption} and the facts that
\bals
\bsp
\;& \langle v\rangle \widetilde{w}_{l_2}(\a,\b)\lesssim \widetilde{w}_{l_2}(|\a-\a'|+1,|\b|+1)\langle v\rangle^{\frac{1}{2}}\quad \quad\quad\text{in the case of}\quad
2\leq|\a'|\leq N_0-1,\\
\;&
\langle v\rangle \widetilde{w}_{l_2}(\a,\b)\lesssim \widetilde{w}_{l_2}(|\a-\a'|+1,|\b|+1)\langle v\rangle^{-\frac{7}{2}}\quad\quad   \;\,\text{in the case of}\quad  |\a'|=N_0.
\esp
\eals
For the term ${\mathcal{J}_{7,2}^2}$, it can be inferred  from the H\"{o}lder inequality,  the Cauchy--Schwarz inequality,    \eqref{NGinequality} and the \emph{a priori} assumption  \eqref{priori assumption}  that
\bals
\bsp
{\mathcal{J}_{7,2}^2}\lesssim \;&\!\!\!\! \sum_{N_0+1\leq|\a'|\leq N}\!\!\|\partial_x^{\a'} B^\e \|_{L^2_x} \|\langle v\rangle \widetilde{w}_{l_2}(\a,
\b)\nabla_v\partial_{\b}^{\a-\a'}(\mathbf{I}-\mathbf{P}) f^\e \|_{L^\infty_{x}L^2_{v}} \| \widetilde{w}_{l_2}(\a,\b) \partial_{\b}^{\a}(\mathbf{I}-\mathbf{P}) f^\e \|_{L^2_{x,v}}
\\
\lesssim \;&\e^2\|\nabla_x B^\e\|_{H^{N-1}_x}^2
\sum_{N_0+1\leq|\a'|\leq N} \|\langle v\rangle \widetilde{w}_{l_2}(\a,
\b)\nabla_v\partial_{\b}^{\a-\a'}\nabla_x(\mathbf{I}-\mathbf{P}) f^\e \|_{H^1_{x}L^2_{v}}^2\\
\;&
+\frac{\eta}{\e^2} \| \widetilde{w}_{l_2}(\a,\b) \partial_{\b}^{\a}(\mathbf{I}-\mathbf{P}) f^\e \|_{L^2_{x,v}}^2
 \\
\lesssim \;&\delta_0
 {\overline{\mathcal{D}}}_{N-1,l_1}(t)
 + \eta(1+t)^{\sigma_{n,|\b|}}\mathcal{D}_{l_2}^{n,|\b|}(t).
\esp
\eals
Here, we have used the inequality
\bals
\bsp
\langle v\rangle \widetilde{w}_{l_2}(\a,\b)
\leq \overline{w}_{l_1}(|\a-\a'|+2, |\b|+1)\langle v\rangle ^{l_2-l_1-4|\a'|+8+\frac{1}{2}|\b|+4}
\leq \overline{w}_{l_1}(|\a-\a'|+2, |\b|+1),
\esp
\eals
which is derived from
$1+l_2-l_1-4|\a'|+8+\frac{1}{2}|\b|+4
\leq
l_2-l_1+9+\frac{1}{2}N- \frac{9}{2}N_0\leq 0
$ because of $|\a'|\geq N_0+1, |\b|=n-|\a'|\leq N-N_0$ and \eqref{hard assumption}.
Thereby, combining  the estimates of  ${\mathcal{J}_{7,2}^1}$ and ${\mathcal{J}_{7,2}^2}$ gives rise to
\bals
\bsp
\mathcal{J}_{7,2}\lesssim \;&
\delta_0
 \sum_{n\leq N-1}(1+t)^{\sigma_{n,|\b|}}\mathcal{D}_{l_2}^{n,|\b|+1}(t)+
\delta_0
 {\overline{\mathcal{D}}}_{N-1,l_1}(t)
 + \eta(1+t)^{\sigma_{n,|\b|}}\mathcal{D}_{l_2}^{n,|\b|}(t).
\esp
\eals
By employing the analogous argument of the estimate of $\mathcal{J}_{7,2}$, we are also able to  deduce
\bals
\bsp
\mathcal{J}_{8,2}\lesssim \;&
\delta_0
 \sum_{n\leq N-1}(1+t)^{\sigma_{n,|\b|}}\mathcal{D}_{l_2}^{n,|\b|+1}(t)+
\delta_0
 {\overline{\mathcal{D}}}_{N-1,l_1}(t)
 + \eta(1+t)^{\sigma_{n,|\b|}}\mathcal{D}_{l_2}^{n,|\b|}(t).
\esp
\eals

Furthermore,  for the rest terms in  \eqref{diyigeweight:L^2:2}  when $|\a|+|\b|=n\leq N-1$,   taking    the   similar estimated manner as \emph{Step 2} in the proof of Lemma \ref{weighted 2:N},  and then combining  the above estimates of $\mathcal{J}_{7,2}$ and $\mathcal{J}_{8,2}$, we   conclude \eqref{estimate-weighted-2:diyuN}. This completes the proof of Lemma \ref{weighted 2:diyuN}.
\end{proof}

Finally, with  Lemma  \ref{weighted 2:N}  and Lemma \ref{weighted 2:diyuN} in hand,   we turn to completing the proof of Proposition \ref{weighted 2:zongguji}.
\begin{proof}[\textbf{Proof of  Proposition \ref{weighted 2:zongguji}}]
From  \eqref{energy functional 2-N} and  \eqref{energy functional 2-N-new},
we deduce
\bals
\widetilde{{\mathcal{E}}}_{N,l_2}(t)=\sum_{n\leq N}\widetilde{{\mathcal{E}}}_{l_2}^{(n)}(t),
\quad\widetilde{{\mathcal{D}}}_{N,l_2}(t)=\sum_{n\leq N}\widetilde{{\mathcal{D}}}_{l_2}^{(n)}(t).
\eals
Therefore,  gathering \eqref{estimate-weighted-2:N}  and\eqref{estimate-weighted-2:diyuN}, and then choosing $\delta_0$ small enough, we conclude \eqref{estimate-weighted-2:zong}. This completes the proof of Proposition \ref{weighted 2:zongguji}.
\end{proof}
\subsection{Energy Estimate in Negative Sobolev Space}
\hspace*{\fill}

This subsection aims to provide an energy estimate in negative Sobolev space  for the VMB system (\ref{rVPB}).
\begin{proposition}\label{negative sobolev energy estimates total}
Assume that  the \emph{a priori} assumption  \eqref{priori assumption} holds true for $\delta_0$ small enough.
Then, there holds
\begin{align}
\begin{split}\label{negative sobolev estimate total}
&\frac{\d}{\d t}\mathcal{E}_{-s}(t)
+\mathcal{D}_{-s}(t)+\varepsilon \frac{\d}{\d t} \mathcal{E}^{-s}_{int}(t)
\lesssim   \mathcal{D}_{N}(t)+\delta_0\mathcal{D}_{N}(t)
+ \delta_0{\overline{\mathcal{D}}}_{N-1,l_1}(t).
\end{split}
\end{align}
Here,   $\mathcal{E}_{-s}(t)$ and   $\mathcal{D}_{-s}(t)$ are given in \eqref{negative sobolev energy} and \eqref{negative sobolev dissipation}, while $\mathcal{E}^{-s}_{int}(t)$   satisfies  \eqref{macro negative sobolev xiao energy}.
\end{proposition}

\begin{proof}[\textbf{Proof of Proposition \ref{negative sobolev energy estimates total}}]
The proof of (\ref{negative sobolev estimate total}) will be split up into two steps.

\medskip

 \emph{Step 1. Energy estimate  of  $\Lambda^{-s} f^\e$.}\;
Precisely speaking, we can obtain
\begin{align}
\begin{split}\label{negative sobolev estimate}
&\frac{\d}{\d t} \left[ \left\|\Lambda^{-s}f^\varepsilon\right\|^2_{L^2_{x,v}}
+\left\|\Lambda^{-s}E^\varepsilon\right\|^2_{L^2_{x}}
+\left\|\Lambda^{-s}B^\varepsilon\right\|^2_{L^2_{x}} \right]
+\frac{\sigma_0}{\varepsilon^2}
\left\|\Lambda^{-s}(\mathbf{I}-\mathbf{P})f^\varepsilon\right\|^2_{L^2_{x,v}(\nu)} \\
\lesssim\;&  \delta_0
\left[ \e ^2\| \Lambda^{\frac{3}{4}-\frac{s}{2}} E^\varepsilon\|^2_{L^2_x}
+ \| \Lambda^{\frac{3}{4}-\frac{s}{2}} f^\varepsilon \|^2_{L^2_{x,v}(\nu)}+\e^2\| \Lambda^{\frac{3}{2}-s} E^\varepsilon\|^2_{L^2_x}+\| \Lambda^{\frac{3}{2}-s} f^\varepsilon\|^2_{L^2_{x,v}}\right]
\\
 \;&
+ \delta_0
\left[ \| \mathbf{P}f^\varepsilon \|^2_{L^2_{x,v}}
+ \| \Lambda^{1-s} \mathbf{P}f^\varepsilon \|^2_{L^2_{x,v}}+\mathcal{D}_N(t)+{\overline{\mathcal{D}}}_{N-1,l_1}(t)\right].
\end{split}
\end{align}
To do this,
applying $\Lambda^{-s}$ to the first equation of \eqref{rVPB} and taking the inner product with $\Lambda^{-s} f^\varepsilon$ over $\mathbb{R}_x^3 \times \mathbb{R}_v^3$, we derive from \eqref{spectL} that
\bal
 \bsp\label{negative sobolev estimate 1}
&\frac{1}{2}\frac{\d}{\d t} \left\| \Lambda^{-s} f^\varepsilon \right\|^2_{L^2_{x,v}}
+\underbrace{\left\langle \Lambda^{-s} E^\varepsilon \cdot v \mu^{1/2} ,\Lambda^{-s} f^\varepsilon \right\rangle_{L^2_{x,v}}}_{\mathcal{I}_1}
+ \frac{\sigma_0}{\varepsilon^2} \left\| \Lambda^{-s} (\mathbf{I}-\mathbf{P})f^\varepsilon \right\|_{L^2_{x,v}(\nu)}^2\\
=\;&\underbrace{\left\langle- \e\Lambda^{-s}\left(E^\varepsilon \cdot \nabla_v f^\varepsilon \right), \Lambda^{-s} f^\varepsilon \right\rangle_{L^2_{x,v}}}_{\mathcal{I}_2}
+\underbrace{\left\langle \frac{\e}{2} \Lambda^{-s}\left(E^\varepsilon \cdot v f^\varepsilon \right), \Lambda^{-s} f^\varepsilon \right\rangle_{L^2_{x,v}}}_{\mathcal{I}_3}\\
\;&+\underbrace{\left\langle- \Lambda^{-s}\left(v\times B^\varepsilon \cdot \nabla_v f^\varepsilon \right), \Lambda^{-s} f^\varepsilon \right\rangle_{L^2_{x,v}}}_{\mathcal{I}_4}
+\underbrace{\frac{1}{\varepsilon} \left\langle \Lambda^{-s}\Gamma(f^\varepsilon, f^\varepsilon), \Lambda^{-s} f^\varepsilon \right\rangle_{L^2_{x,v}}}_{\mathcal{I}_5}.
 \esp
\eal
Now we further estimate $\mathcal{I}_1\sim \mathcal{I}_5$ in \eqref{negative sobolev estimate 1} term by term. It follows from  the second and third equations of \eqref{rVPB} that
\begin{align}\nonumber
\mathcal{I}_1=\left\langle  \Lambda^{-s} E^\e,\Lambda^{-s}\pt_t E^\varepsilon- \nabla_x\times B^\varepsilon \right\rangle_{L^2_x}
=\frac{1}{2}\frac{\d}{\d t}\left(\left\| \Lambda^{-s} E^\varepsilon\right\|^2_{L^2_x}+\left\| \Lambda^{-s}B^\varepsilon\right\|^2_{L^2_x}\right)
.
\end{align}
Owing to $1 < s < \frac{3}{2}$, by using \eqref{negative embed 1}, \eqref{negative embed 2}, the Minkowski inequality \eqref{minkowski}, the H\"{o}lder inequality, the Cauchy--Schwarz inequality with $\eta$ and  the \emph{a priori} assumption  \eqref{priori assumption}, we deduce
\bals
\bsp
\mathcal{I}_2=\;&\e\left\langle   \Lambda^{-s}\left(E^\varepsilon \cdot \nabla_v f^\varepsilon \right), \Lambda^{-s} \mathbf{P}f^\varepsilon \right\rangle_{L^2_{x,v}}
+\e\left\langle   \Lambda^{-s}\left(E^\varepsilon \cdot \nabla_v f^\varepsilon \right), \Lambda^{-s} (\mathbf{I}-\mathbf{P})f^\varepsilon \right\rangle_{L^2_{x,v}} \\
\lesssim\;&\e\| \Lambda^{-s}(E^\varepsilon \mu^{\delta}f^\varepsilon)\|_{L^2_{x,v}}
\| \Lambda^{-s} \mathbf{P}f^\varepsilon \|_{L^2_{x,v}}
+\e\| \Lambda^{-s}(E^\varepsilon \cdot \nabla_v f^\varepsilon\langle v \rangle ^{-\frac{\gamma}{2}})\|_{L^2_{x,v}}
\| \Lambda^{-s} (\mathbf{I}-\mathbf{P})f^\varepsilon \|_{L^2_{x,v}(\nu)} \\
\lesssim\;&\e \left\| E^\varepsilon \mu^{\delta}f^\varepsilon\right\|_{{L^2_v}{L^{\frac{6}{3+2s}}_x}}
\left\| \Lambda^{-s} \mathbf{P}f^\varepsilon \right\|_{L^2_{x,v}}
+\e\| E^\varepsilon \cdot \nabla_v f^\varepsilon\langle v \rangle ^{-\frac{\gamma}{2}}\|_{L^2_v L_x^{\frac{6}{3+2s}}}
\| \Lambda^{-s} (\mathbf{I}-\mathbf{P})f^\varepsilon \|_{L^2_{x,v}(\nu)}  \\
\lesssim\;&\e \left\| E^\varepsilon \right\|_{L^{\frac{12}{3+2s}}_x}
\left\| \mu^{\delta}f^\varepsilon \right\|_{L^2_{v}{L^{\frac{12}{3+2s}}_x}}
\left\| \Lambda^{-s} \mathbf{P}f^\varepsilon \right\|_{L^2_{x,v}}
+\e\| E^\varepsilon\|_{L^{\frac{3}{s}}_x}
\|  \nabla_v f^\varepsilon \|_{L^2_{x,v}}
\left\| \Lambda^{-s} (\mathbf{I}-\mathbf{P})f^\varepsilon \right\|_{L^2_{x,v}(\nu)} \\
\lesssim\;&\delta_0\e \| \Lambda^{\frac{3}{4}-\frac{s}{2}} E^\varepsilon\|_{L^2_{x}}
 \| \Lambda^{\frac{3}{4}-\frac{s}{2}} (\mu^{\delta}f^\varepsilon)\|_{L^2_{x,v}}
+\e\| \Lambda^{\frac{3}{2}-s} E^\varepsilon\|_{L^2_{x}}
\left\|   \nabla_v f^\varepsilon \right\|_{L^2_{x,v}}
\left\| \Lambda^{-s} (\mathbf{I}-\mathbf{P})f^\varepsilon \right\|_{_{L^2_{x,v}(\nu)}} \\
\lesssim\;& \delta_0
\left[ \e ^2\| \Lambda^{\frac{3}{4}-\frac{s}{2}} E^\varepsilon\|^2_{L^2_x}
+ \| \Lambda^{\frac{3}{4}-\frac{s}{2}} f^\varepsilon \|^2_{L^2_{x,v}(\nu)}\right]
+\delta_0\e^4\| \Lambda^{\frac{3}{2}-s} E^\varepsilon\|^2_{L^2_x}
+\frac{\eta}{\e^2} \left\| \Lambda^{-s} (\mathbf{I}-\mathbf{P})f^\varepsilon \right\|^2_{_{L^2_{x,v}(\nu)}} .
\esp
\eals
Similarly, we directly find  that  $\mathcal{I}_3$ has the same upper bound as $\mathcal{I}_2$.

For $\mathcal{I}_4$, we derive from   \eqref{negative embed 1}, \eqref{negative embed 2}, the Minkowski inequality \eqref{minkowski}, the H\"{o}lder inequality, the Cauchy--Schwarz inequality with $\eta$ that
\bals
\bsp
\mathcal{I}_4=\;& \left\langle   \Lambda^{-s}\left(v\times B^\varepsilon \cdot \nabla_v f^\varepsilon \right), \Lambda^{-s} \mathbf{P}f^\varepsilon \right\rangle_{L^2_{x,v}}
+ \left\langle   \Lambda^{-s}\left(v\times B^\varepsilon  \cdot \nabla_v f^\varepsilon \right), \Lambda^{-s} (\mathbf{I}-\mathbf{P})f^\varepsilon \right\rangle_{L^2_{x,v}} \\
\lesssim\;& \| \Lambda^{-1-s}(B^\varepsilon \mu^{\delta}  f^\varepsilon)\|_{L^2_{x,v}}
\| \Lambda^{1-s} \mathbf{P}f^\varepsilon \|_{L^2_{x,v}}
+\| \Lambda^{-s}(v\times B^\varepsilon \cdot \nabla_v f^\varepsilon)\|_{L^2_{x,v}}
\| \Lambda^{-s} (\mathbf{I}-\mathbf{P})f^\varepsilon \|_{L^2_{x,v}(\nu)} \\
\lesssim\;&  \| B^\varepsilon \mu^{\delta}f^\varepsilon \|_{{L^2_v}{L^{\frac{6}{5+2s}}_x}}
 \| \Lambda^{1-s} \mathbf{P}f^\varepsilon  \|_{L^2_{x,v}}
+ \|v\times B^\varepsilon \cdot \nabla_v f^\varepsilon\|_{L^2_v L_x^{\frac{6}{3+2s}}}
\| \Lambda^{-s} (\mathbf{I}-\mathbf{P})f^\varepsilon \|_{L^2_{x,v}(\nu)}  \\
\lesssim\;& \left\| B^\varepsilon \right\|_{L^{\frac{3}{1+s}}_x}
 \| \mu^{\delta}f^\varepsilon  \|_{L^2_{x,v}}
 \| \Lambda^{1-s} \mathbf{P}f^\varepsilon  \|_{L^2_{x,v}}
+ \| B^\varepsilon\|_{L^{\frac{3}{s}}_x}
\| \langle v\rangle \nabla_v f^\varepsilon \|_{L^2_{x,v}}
\| \Lambda^{-s} (\mathbf{I}-\mathbf{P})f^\varepsilon \|_{L^2_{x,v}(\nu)} \\
\lesssim\;&\|\Lambda^{\frac{1}{2}\!-\!s}B^\e \|_{L^2_x}
 \| \mu^{\delta}\!f^\varepsilon\|_{L^2_{x,v}}\| \!\Lambda^{1-s} \mathbf{P}f^\varepsilon  \!\|_{L^2_{x,v}}
\!+ \!\| \Lambda^{\frac{3}{2}-s} B^\varepsilon\|_{L^2_{x}}
\|  \langle v\rangle  \nabla_v f^\varepsilon \|_{L^2_{x,v}}
\| \Lambda^{-s} (\mathbf{I}\!-\!\mathbf{P})f^\varepsilon \|_{_{L^2_{x,v}(\nu)}} \\
\lesssim\;& \delta_0
\left[\| \mu^{\delta}f^\varepsilon\|_{L^2_{x,v}}^2
+ \| \Lambda^{1-s} \mathbf{P}f^\varepsilon \|^2_{L^2_{x,v}}\right]
+\delta_0 \left[ \|  \mathbf{P}f^\varepsilon \|^2_{L^2_{x,v}}
+\|  \langle v\rangle  \nabla_v(\mathbf{I}- \mathbf{P})f^\varepsilon \|_{L^2_{x,v}(\nu)}^2
\right]+\eta\mathcal{D}_{-s}(t).
\esp
\eals
Here, deduced from  the interpolation inequality with respect to the spatial derivatives, the \emph{a priori} assumption  \eqref{priori assumption} and $1 < s < \frac{3}{2}$, the facts utilized are that
\bals
\|\Lambda^{\frac{1}{2}-s}B^\e \|_{L^2_x} \leq\|\Lambda^{ -s}B^\e \|_{L^2_x}+\|\nabla_x B^\e \|_{L^2_x}\lesssim \delta_0,
\quad
\| \Lambda^{\frac{3}{2}-s} B^\varepsilon\|_{L^2_{x}}
\leq\|   B^\varepsilon\|_{H^1_{x}}\lesssim \delta_0.
\eals
Finally, for the term $\mathcal{I}_5$, by the collision invariant property, we arrive at
\begin{align*}
\bsp 
\mathcal{I}_5
\lesssim \;& \frac{1}{\varepsilon}  \|\Lambda^{-s} ( \langle v \rangle ^{-\frac{\gamma}{2}}\Gamma(f^\varepsilon, f^\varepsilon) )  \|_{L^2_{x,v}}
 \| \Lambda^{-s} (\mathbf{I}-\mathbf{P})f^\varepsilon  \|_{L^2_{x,v}(\nu)} \\
\lesssim\;& \frac{1}{\varepsilon} \| \langle v \rangle ^{-\frac{\gamma}{2}}\Gamma(f^\varepsilon, f^\varepsilon) \|_{L^2_{v}L_x^{\frac{6}{3+2s}}}
 \| \Lambda^{-s} (\mathbf{I}-\mathbf{P})f^\varepsilon  \|_{L^2_{x,v}(\nu)} \\
\lesssim\;&\frac{1}{\varepsilon}  \| \langle v \rangle ^{-\frac{\gamma}{2}}\Gamma(f^\varepsilon, f^\varepsilon) \|_{L_x^{\frac{6}{3+2s}}L^2_{v}}
\| \Lambda^{-s} (\mathbf{I}-\mathbf{P})f^\varepsilon \|_{L^2_{x,v}(\nu)} \\
\lesssim\;&\frac{1}{\varepsilon} \| |f^\varepsilon|_{L^\infty_v} |f^\varepsilon|_{L^2_v(\nu)} \|_{L_x^{\frac{6}{3+2s}}}
\| \Lambda^{-s} (\mathbf{I}-\mathbf{P})f^\varepsilon \|_{L^2_{x,v}(\nu)} \\
\lesssim\;&\frac{1}{\varepsilon} \|f^\varepsilon\|_{L_x^{\frac{3}{s}}L^2_v} \|f^\varepsilon\|_{L^2_xH^2_v}
\| \Lambda^{-s} (\mathbf{I}-\mathbf{P})f^\varepsilon \|_{L^2_{x,v}(\nu)} \\
\lesssim\;& \delta_0 \| \Lambda^{\frac{3}{2}-s}f^\varepsilon\|_{L^2_{x,v}}^2
+\frac{\eta}{\varepsilon^2} \left\| \Lambda^{-s} (\mathbf{I}-\mathbf{P}) f^\varepsilon \right\|_{L^2_{x,v}(\nu)}^2,
\esp
\end{align*}
where we have used \eqref{negative embed 1}, \eqref{negative embed 2}, the Minkowski inequality \eqref{minkowski},  the H\"{o}lder inequality, the Cauchy--Schwarz inequality with $\eta$ and \eqref{priori assumption} .

As a result, gathering the above estimates of  $\mathcal{I}_1\sim\mathcal{I}_5$  into
 \eqref{negative sobolev estimate 1}
and choosing $\eta$ small enough yield
 \eqref{negative sobolev estimate}.

\medskip

 \emph{Step 2. Macroscopic dissipation estimate  of  $\Lambda^{1-s} \mathbf{P}f^\e$.}\;

\newcommand{\J}{\mathcal{G}}
In fact, similar to Lemma \ref{macroscopic estimate}, there exists an interactive functional $\mathcal{E}^{-s}_{int}(t)$  defined by
satisfying
\begin{align}\label{macro negative sobolev xiao energy}
\left|\mathcal{E}^{-s}_{int}(t)\right| \lesssim \left\| \Lambda^{-s} f^\varepsilon \right\|^2_{L^2_{x,v}}
+ \left\| \Lambda^{1-s} f^\varepsilon \right\|^2_{L^2_{x,v}}
+ \left\| \Lambda^{1-s} E^\e \right\|^2_{L^2_{x}}+ \left\| \Lambda^{1-s} \nabla_x\times B^\e \right\|^2_{L^2_{x}},
\end{align}
such that
\begin{align}
\begin{split}\label{macro negative sobolev1}
&\varepsilon \frac{\d}{\d t} \mathcal{E}^{-s}_{int}(t) + \left\| \Lambda^{1-s} \mathbf{P}f^\e\right\|^2_{L^2_{x,v}}+\e^2  \left\| \Lambda^{1-s} E^{\varepsilon} \right\|^2_{L^2_{x}}+ \e^2  \left\| \Lambda^{1-s}\nabla_x\times B^{\varepsilon} \right\|^2_{L^2_{x}}  \\
\lesssim\;& \left\| \Lambda^{1-s}(\mathbf{I}-\mathbf{P})f^\varepsilon \right\|^2_{L^2_{x,v}(\nu)}+ \frac{1}{\varepsilon^2}\left\| \Lambda^{-s}(\mathbf{I}-\mathbf{P})f^\varepsilon \right\|^2_{L^2_{x,v}(\nu)}
+ \eta \| \Lambda^{1-s} \nabla_x( a^{\varepsilon}+c^\e)  \|^2_{L^2_{x}}\\
\;&
+\eta\e^2  \left\| \Lambda^{1-s} \nabla_x\times E^{\varepsilon} \right\|^2_{L^2_{x}}
+\e^2\left(\|  \Lambda^{-s}\mathcal{X}_{b^\e}\|_{L_{x,v}^2}^2
+\|  \Lambda^{1-s}\mathcal{X}_{b^\e}\|_{L_{x,v}^2}^2\right)\\
\;&
  +\varepsilon^2\sum_{i=1}^3\sum_{j\neq i}\left\|  \big\langle \Lambda^{-s}( \J _3^i +\J _3^j +\J _4 ^{i,j}+\J _5^i ),\, \zeta\big\rangle_{L^2_{v}}\right\|^2_{L^2_{x}}.
\end{split}
\end{align}
Recalling  \eqref{g define}, \eqref{qkj:esti:f:L^2:12-1} for $g^\varepsilon$ and $\mathcal{X}_{b^\e}$, we obtain from \eqref{hard gamma}  that
\begin{align}
\bsp\label{g estimate--s}
&\e^2\left(\|  \Lambda^{-s}\mathcal{X}_{b^\e}\|_{L_{x,v}^2}^2
\!+\!\|  \Lambda^{1-s}\mathcal{X}_{b^\e}\|_{L_{x,v}^2}^2\right)
  \!+\!\varepsilon^2\sum_{i=1}^3\!\sum_{j\neq i}\!\left\|  \big\langle \Lambda^{-s}( \J _3^i +\J _3^j +\J _4 ^{i,j}+\J _5^i ),\, \zeta\big\rangle_{L^2_{v}}\right\|^2_{L^2_{x}}
\\
\lesssim\;&\left\| \Lambda^{1-s}(\mathbf{I}-\mathbf{P})f^\varepsilon \right\|_{L^2_{x,v}(\nu)}^2
+ \frac{1}{\varepsilon^2}\left\| \Lambda^{-s}(\mathbf{I}-\mathbf{P})f^\varepsilon \right\|_{L^2_{x,v}(\nu)}^2
+ \mathcal{D}_{N}(t)
+\delta_0\left\|\Lambda^{1-s} \mathbf{P}f^\e\right\|^2_{L^2_{x,v}}
\\
\;&
+\delta_0\left[\left\| \Lambda^{-s}(\mathbf{I}-\mathbf{P})f^\varepsilon \right\|_{L^2_{x,v}(\nu)}^2
+\mathcal{D}_{N}(t)\right].
\esp
\end{align}
Hence,
 combining \eqref{macro negative sobolev1} and \eqref{g estimate--s} and choosing $\delta_0$ small enough,   we conclude
\begin{align}
\begin{split}\label{macro negative sobolev2}
&\varepsilon \frac{\d}{\d t} \mathcal{E}^{-s}_{int}(t)  + \left\| \Lambda^{1-s} \mathbf{P}f^\e\right\|^2_{L^2_{x,v}}+\e^2  \left\| \Lambda^{1-s} E^{\varepsilon} \right\|^2_{L^2_{x}}+ \e^2  \left\| \Lambda^{1-s}\nabla_x\times B^{\varepsilon} \right\|^2_{L^2_{x}} \\
\lesssim\;& \left\| \Lambda^{1-s}(\mathbf{I}-\mathbf{P})f^\varepsilon \right\|^2_{L^2_{x,v}(\nu)}+ \frac{1}{\varepsilon^2}\left\| \Lambda^{-s}(\mathbf{I}-\mathbf{P})f^\varepsilon \right\|^2_{L^2_{x,v}(\nu)}
+ \eta \| \Lambda^{1-s} \nabla_x \mathbf{P}f^\varepsilon  \|^2_{L^2_{x.v}}\\
\;&
+\eta\e^2   \left\| \Lambda^{1-s} \nabla_x\times E^{\varepsilon} \right\|^2_{L^2_{x}}+ \mathcal{D}_{N}(t)
+\delta_0\left[
\left\| \Lambda^{-s}(\mathbf{I}-\mathbf{P})f^\varepsilon \right\|_{L^2_{x,v}(\nu)}^2
+\mathcal{D}_{N}(t)\right].
\end{split}
\end{align}

Moreover,
via the interpolation inequality with respect to the spatial derivatives, it is easy to derive that
\begin{align}
\begin{split}\label{interpolation negative}
&\left\| \Lambda^{1-s}(\mathbf{I}-\mathbf{P})f^\varepsilon \right \|^2_{L^2_{x,v}(\nu)}
\lesssim \left\| \Lambda^{-s}(\mathbf{I}-\mathbf{P})f^\varepsilon \right\|_{L^2_{x,v}(\nu)}^2
+\left \|  \nabla_x (\mathbf{I}-\mathbf{P})f^\varepsilon \right\|_{L^2_{x,v}(\nu)}^2,  \\
&\left\| \Lambda^{\frac{3}{4}-\frac{s}{2}}f^\varepsilon \right \|_{ L^2_{x,v}(\nu)}^2
\lesssim \left\| \Lambda^{1-s}\mathbf{P}f^\varepsilon \right\|_{ L^2_{x,v}}^2
+\left\| \Lambda^{-s}(\mathbf{I}-\mathbf{P})f^\varepsilon \right\|_{ L^2_{x,v}(\nu)}^2
+\left\| \nabla_xf^\varepsilon \right\|_{ L^2_{x,v}(\nu)}^2, \\
&\left\| \Lambda^{\frac{3}{2}-s}f^\varepsilon \right\|_{ L^2_{x,v}}^2
\lesssim \left \| \Lambda^{1-s} \mathbf{P}f^\varepsilon \right\|_{ L^2_{x,v}}^2
+\left\| \nabla_x \mathbf{P}f^\varepsilon\right\|_{ L^2_{x,v}}^2
+\left\|  (\mathbf{I}-\mathbf{P})f^\varepsilon\right\|_{H^1_x L^2_v(\nu)}^2,\\
&   \left\| \Lambda^{1-s} \nabla_x\mathbf{P}f^\varepsilon\right\|^2_{L^2_{x,v}}
+ \left\| \mathbf{P}f^\varepsilon \right\|^2_{L^2_{x,v}}
\lesssim  \left\| \Lambda^{1-s} \mathbf{P}f^\varepsilon\right\|_{ L^2_{x,v}}^2
+\left\| \nabla_x \mathbf{P}f^\varepsilon\right\|_{ H^1_{x} L^2_{v}}^2,\\
&\left\| \Lambda^{\frac{3}{4}-\frac{s}{2}} E^\varepsilon \right\|_{ L^2_{x}}^2
+\left\| \Lambda^{\frac{3}{2}-s} E^\varepsilon \right\|_{ L^2_{x}}^2+ \left\| \Lambda^{1-s} \nabla_x\times E^{\varepsilon} \right\|^2_{L^2_{x}}
\lesssim \left\| \Lambda^{1-s} E^\varepsilon \right\|_{ L^2_{x}}^2
+\left\| \nabla_x E^\varepsilon\right \|_{ H^1_{x}}^2,
\end{split}
\end{align}
where we have used $\frac{3}{4}-\frac{s}{2} > 1-s$ for $1 < s < \frac{3}{2}$.
Hence, with the aid of \eqref{interpolation negative},
we   deduce \eqref{negative sobolev estimate total} by a proper linear combination of \eqref{negative sobolev estimate} and \eqref{macro negative sobolev2}. This completes the proof of Proposition \ref{negative sobolev energy estimates total}.
\end{proof}

\section{Global Existence}\label{Global Existence}

In this section, we are devoted to  the proof of Theorem \ref{mainth1}  by first establishing the closed uniform estimates on $X(t)$ defined in \eqref{X define}.
For this purpose, we also need to obtain the time decay estimates of $\e^s\mathcal{E}_{N_0}^0(t)$ and $\e^{1+s}\mathcal{E}_{N_0}^1(t)$.

\subsection{Time Decay Estimate}
\hspace*{\fill}

Our main idea to deduce the time decay estimate   is based on the approach proposed by \cite{GW2012CPDE}. We are now in a position to state the main result of this subsection.
\begin{proposition}\label{k-N decay proposition}
Assume that   the \emph{a priori} assumption  \eqref{priori assumption} holds true for $\delta_0$ small enough.
There exist two  energy functionals $\mathcal{E}^0_{N_0}(t)$, $\mathcal{E}^1_{^{N_0}}(t)$ and the corresponding energy dissipation rate  functionals $\mathcal{D}^0_{N_0}(t)$, $\mathcal{D}^1_{N_0}(t)$ satisfying \eqref{low k energy} and \eqref{low k dissipation} respectively such that
  for   all $0 \leq t \leq T$
\bal
&\frac{\d}{\d t} \mathcal{E}^0_{N_0}(t)+  \mathcal{D}^0_{N_0}(t) \leq 0\label{0-N estimate},\\
&\frac{\d}{\d t} \mathcal{E}^1_{N_0}(t)+   \mathcal{D}^1_{N_0}(t)
\lesssim \|\nabla_x B^\e\|_{H^{N_0-1}_x}\mathcal{D}^0_{N_0}(t).\label{1-N estimate}
\eal
Furthermore, we can obtain that
\bal
\bsp\label{k-N decay}
\sup_{0\leq t\le T}\!\left\{\e^{s}(1+t)^{\frac{s}{2}} \mathcal{E}^0_{N_0}(t)\!+\!\e^{1\!+\!s}(1+t)^{\frac{1\!+\!s}{2}} \mathcal{E}^1_{N_0}(t) \right\}
\lesssim\;&{\mathcal{E}}_{N}(0)
\!+\!\sup_{0\leq t\leq T}\!\left\{{\mathcal{E}}_{N_0\!+\!\frac{1+s}{2}}(t)+
{\mathcal{E}}_{\!-s}(t)\right\}.
\esp
\eal
\end{proposition}

To prove  Proposition \ref{k-N decay proposition},
 we  give a more refined energy estimate for the pure spatial derivatives than  that of    Lemma \ref{energy estimate without weight} for preparations.
\begin{lemma}\label{0-N 1-N lemma}
Assume that   the \emph{a priori} assumption  \eqref{priori assumption} holds true for $\delta_0$ small enough. Then,  for   all $0 \leq t \leq T$, there hold
\begin{align}
 \frac{\d}{\d t} \left[  \left\|   f^\varepsilon \right\|_{H^{N_0}_x\!L^2_v}^2 +  \left\|  E^\varepsilon \right\|_{H^{N_0}_x}^2+ \left\| B^\varepsilon \right\|_{H^{N_0}_x}^2 \right]
+ \frac{\sigma_0}{\varepsilon^2}   \left\|  (\mathbf{I}-\mathbf{P}) f^\varepsilon \right\|^2_{H^{N_0}_x\!L^2_v(\nu)}
\lesssim\;& \delta_0 \mathcal{D}^0_{N_0}(t),  \label{0-N estimate decay}
\end{align}
and
\begin{align}
\bsp
\;&  \frac{\d}{\d t} \left[  \left\|  \nabla_x f^\varepsilon \right\|_{H^{N_0-1}_x\!L^2_v}^2\! + \! \left\|   \nabla_x E^\varepsilon \right\|_{H^{N_0-1}_x}^2\!+\! \left\| \nabla_x B^\varepsilon \right\|_{H^{N_0-1}_x}^2 \right]
+ \frac{\sigma_0}{\varepsilon^2}   \left\|  \nabla_x (\mathbf{I}\!-\!\mathbf{P}) f^\varepsilon \right\|^2_{H^{N_0-1}_xL^2_v(\nu)}\\[1mm]
\lesssim\;& \delta_0 \mathcal{D}^1_{N_0}(t)+\|\nabla_x B^\e\|_{H^{N_0-1}_x}\mathcal{D}^0_{N_0}(t). \label{1-N estimate decay}
\esp
\end{align}
\end{lemma}

\begin{proof}
Applying $\nabla^k_x$ with $1\leq k \leq N_0 $ to  the first equation of \eqref{rVPB}  and taking the inner product with $  \nabla^k_xf^\varepsilon $ over $\mathbb{R}_x^3$ $\times$  $\mathbb{R}_v^3$, we obtain by employing the coercivity of $ L$ in (\ref{spectL}) that
\bal
\bsp\label{pure spatial estimate}
&\frac{1}{2}\frac{\d}{\d t}\left[\left\| \nabla^k_x f^\varepsilon\right\|^2_{L^2_{x,v} } +
\left\| \nabla^k_x E^\varepsilon\right\|^2_{L^2_x}+\left\| \nabla^k_x B^\varepsilon\right\|^2_{L^2_x}\right]
+\frac{\sigma_0}{\varepsilon^2}\left\|\nabla^k_x (\mathbf{I}-\mathbf{P})f^\varepsilon\right\|^2_{L^2_{x,v}(\nu)}  \\
\lesssim\;&\underbrace{\sum_{1\leq j\leq k} \Big\langle\e\nabla_x^j E^{\e} \cdot \nabla_v\nabla_x^{k-j} f^\varepsilon,  \nabla^k_x f^\varepsilon\Big\rangle_{L^2_{x,v} }}_{\mathcal{M}_1}
 +\underbrace{ \frac{\e}{2}\sum_{0\leq j\leq k} \Big\langle v \cdot\nabla^{j}_x E^\e \nabla^{k-j}_x f^\e,\nabla^k_x f^\e \Big\rangle_{L^2_{x,v}}}_{\mathcal{M}_2}
  \\
&+\underbrace{\sum_{1\leq j\leq k} \Big\langle v\times \nabla^{j}_xB^\e \cdot \nabla_v \nabla^{k-j}_x f^\varepsilon, \nabla^k_x f^\varepsilon\Big\rangle_{L^2_{x,v} }}_{\mathcal{M}_3}+\underbrace{\Big\langle\frac{1}{\varepsilon}
\nabla^k_x \Gamma(f^\varepsilon, f^\varepsilon),\nabla^k_x
(\mathbf{I}-\mathbf{P})f^\varepsilon\Big\rangle_{L^2_{x,v} }}_{\mathcal{M}_4} .
\esp
\eal
Here, we have used the collision invariant property.

Now we estimate $\mathcal{M}_1\sim\mathcal{M}_4$ with $1 \leq k \leq N_0$  term by term.
For the term $\mathcal{M}_1$, by employing the H\"{o}lder inequality, the Sobolev inequalities  \eqref{NGinequality}  and the Cauchy--Schwarz inequality with $\eta$, we obtain that for any $2 \leq k \leq N_0$
\begin{align}
\begin{split}\label{M_1 estimate}
\mathcal{M}_1 \lesssim\;& \e \Big[\| \nabla_x E^\varepsilon \|_{L^6_x}
\|\nabla_v \nabla^{k-1}_x  f^\varepsilon\|_{L^3_x L^2_v}
+\!\!\!\!\!
\sum_{2 \leq j \leq N_0-1}\!\!\!\!\! \| \nabla^j_x E^\varepsilon \|_{L^2_x}
\|\nabla_v \nabla^{k-j}_x  f^\varepsilon \!\|_{L^\infty_x L^2_v}\Big]
 \| \nabla^k _xf^\varepsilon\|_{L^2_{x,v}} \\
&\quad+\e \| \nabla^{N_0}_x E^\varepsilon \|_{L^2_x}
\left[\| \mathbf{P} f^\varepsilon \|_{L^\infty_x L^2_v}
+\|  \nabla_v(\mathbf{I} -\mathbf{P} )f^\varepsilon\|_{L^\infty_x L^2_v}\right]\left\| \nabla^{N_0} _xf^\varepsilon\right\|_{L^2_{x,v}} \\
\lesssim\;& {\overline{\mathcal{E}}}_{N-1,l_1}(t)\!\!\! \!\sum_{2 \leq \ell \leq N_0-1 }\!\!\!\e^2 \| \nabla^\ell_x E^\varepsilon \|^2_{L^2_x}
 \!+\!\eta\!\!\!\sum_{2 \leq \ell \leq N_0}\!\!\!\| \nabla^\ell_x  f^\varepsilon \|^2_{L^2_{x,v}}\! +\!  \left[ \mathcal{E}_{N}(t)\!+\!{\overline{\mathcal{E}}}_{N-1,l_1}(t) \right] \mathcal{D}_{N_0}^1(t)
\\
\lesssim \;&\delta_0\mathcal{D}_{N_0}^1(t)+\eta\mathcal{D}_{N_0}^1(t).
\end{split}
\end{align}
Here,   we have used the \emph{a priori} assumption  \eqref{priori assumption} and  the fact, which is deduced from the interpolation inequality \eqref{sobolev interpolation} and the relation of $N-N_0\geq 2$ because of  \eqref{hard assumption}, that
\bals
\;&\e^2\left\| \nabla^{N_0}_x E^\varepsilon \right\|_{L^2_x}^2
\left[\left\| \mathbf{P} f^\varepsilon \right\|_{L^\infty_x L^2_v}^2+
\left\| \nabla_v(\mathbf{I}-\mathbf{P}) f^\varepsilon \right\|_{L^\infty_x L^2_v}^2\right]\\
\lesssim\;&\e^2\left\| \nabla^{N_0-1}_x E^\varepsilon \right\|_{L^2}^{\frac{2(N-N_0)}{N-N_0+1}}
 \left\| \nabla^{N}_x E^\varepsilon \right\|_{L^2_x}^{\frac{2}{N-N_0+1}}
 \Big[ \left\|  \nabla_x \mathbf{P} f^\varepsilon \right\|_{L^2_{x,v}}
\left\|  \nabla_x^2 \mathbf{P} f^\varepsilon \right\|_{L^2_{x,v}}\!+\!\left\|\nabla_v ^3\nabla_x(\mathbf{I}-\mathbf{P})  f^\varepsilon \right\|_{L^2_{x,v}}^{\frac{1}{3}}\\
\;&
\times\left\| \nabla_x (\mathbf{I}-\mathbf{P}) f^\varepsilon \right\|_{L^2_{x,v}}^{\frac{2}{3}}
\left\| \nabla_v\nabla_x^2 (\mathbf{I}-\mathbf{P})  f^\varepsilon \right\|_{L^2_{x,v}}\Big] \\
\lesssim\;&  \left[ \mathcal{E}_{N}(t)+{\overline{\mathcal{E}}}_{N-1,l_1}(t) \right] \mathcal{D}_{N_0}^1(t).
\eals
As for the case of $k=1$, similar to that of \eqref{M_1 estimate},
 one has that
\bals
 \mathcal{M}_1 \lesssim\;&  \e \left\| \nabla_x E^\varepsilon \right\|_{L^6_x} \left\|\nabla_v    f^\varepsilon \right\|_{L^3_{x}L^2_v}
  \left\| \nabla_xf^\varepsilon\right\|_{L^2_{x,v}}\\
     \lesssim\;& \e \| \nabla_x^2 E^\varepsilon \|_{L^2_x}\|\nabla_v ^3  f^\varepsilon \|_{L^2_{x,v}}^{\frac{1}{6}}
     \|\nabla_v^3 \nabla_x  f^\varepsilon \|_{L^2_{x,v}}^{\frac{1}{6}}
         \left\| f^\varepsilon \right\|_{L^2_{x,v}}^{\frac{1}{3}}
  \left\|  \nabla_x   f^\varepsilon \right\|_{L^2_{x,v}}^{\frac{4}{3}}
  \\
     \lesssim\;&\e \left\| \nabla_x^2 E^\varepsilon \right\|_{L^2_x}
     \left\|\nabla_v ^3  f^\varepsilon \right\|_{L^2_{x,v}}^{\frac{1}{6}}
     \|\nabla_v^3 \nabla_x  f^\varepsilon \|_{L^2_{x,v}}^{\frac{1}{6}}
    \|\Lambda^{-1} f^\varepsilon\|_{L^2_{x,v}}^{\frac{2}{9}}
     \| \nabla_x^2f^\varepsilon\|_{L^2_{x,v}}^{\frac{1}{9}}
     \| \Lambda^{-1}f^\varepsilon\|_{L^2_{x,v}}^{\frac{4}{9}}
  \| \nabla_x^2f^\varepsilon\|_{L^2_{x,v}}^{\frac{8}{9}}\\
  \lesssim \;& \big[{\overline{\mathcal{E}}}_{N-1,l_1}(t)+\mathcal{E}_{-s}(t)\big] \e^2 \left\| \nabla_x^2 E^\varepsilon \right\|_{L^2_x}^2+\eta\left\| \nabla_x^2f^\varepsilon\right\|_{L^2_{x,v}}^2\\
   \lesssim \;& \delta_0 \mathcal{D}_{N_0}^1(t)+\eta \mathcal{D}_{N_0}^1(t).
 \eals
In an analogous manner, $\mathcal{M}_2$ has the same upper bound as $\mathcal{M}_1$.

For  $\mathcal{M}_{3}$, utilizing  \eqref{fdefenjie} leads to
\bals
\mathcal{M}_{3}
=\;&\underbrace{\sum_{1\leq j\leq k} \Big\langle v\times \nabla^{j}_xB^\e \cdot \nabla_v \nabla^{k-j}_x(\mathbf{I}-\mathbf{P}) f^\varepsilon, \nabla^k_x\mathbf{P} f^\varepsilon\Big\rangle_{L^2_{x,v}}}_{\mathcal{M}_{3,1}}\\
\;&
+\underbrace{\sum_{1\leq j\leq k} \Big\langle v\times \nabla^{j}_xB^\e \cdot \nabla_v \nabla^{k-j}_x \mathbf{P}f^\varepsilon, \nabla^k_x\mathbf{P} f^\varepsilon\Big\rangle_{L^2_{x,v}}}_{\mathcal{M}_{3,2}}\\
\;&+
\underbrace{\sum_{1\leq j\leq k} \Big\langle v\times \nabla^{j}_xB^\e \cdot \nabla_v \nabla^{k-j}_x f^\varepsilon, \nabla^k_x (\mathbf{I}-\mathbf{P}) f^\varepsilon\Big\rangle_{L^2_{x,v}}}_{\mathcal{M}_{3,3}}
\eals
It follows from the H\"{o}lder inequality, the Sobolev embedding theory and the   Cauchy--Schwarz inequality with $\eta$ that for  $1\leq k\leq N_0$,
\bals
\mathcal{M}_{3,1}
\lesssim\;&\!
\Big[\!\!\sum_{1\leq j\leq N_0-1}\!\!\!\!\!\!\|\nabla^{j}_xB^\e \|_{L^3_x}\| \nabla^{k-j}_x (\mathbf{I}-\mathbf{P}) f^\e\!\|_{L^6_x\!L^2_v}\!
+ \!\|\nabla^{N_0}_xB^\e \|_{L^2_x}\|  (\mathbf{I}\!-\!\mathbf{P})  f^\e\|_{L^\infty_x\!L^2_v}\Big]
\|\nabla^k_x \mathbf{P}f^\varepsilon\|_{L^2_{x,\!v}}\\
\lesssim\;&
\mathcal{E}_{N}(t)\sum_{1\leq \ell \leq N_0}\|\nabla^{\ell}_x(\mathbf{I}-\mathbf{P}) f^\e \|_{L^2_{x,v}(\nu)}^2
+\eta\|\nabla^k_x \mathbf{P}f^\varepsilon\|_{L^2_{x,v}}^2.
\eals
In a similar manner, for  $1\leq k\leq N_0$, we find that
\bals
\mathcal{M}_{3,2}
\lesssim\;&\!
\Big[ \sum_{1\leq j\leq N_0-1} \|\nabla^{j}_xB^\e \|_{L^3_x}\| \nabla^{k-j}_x \mathbf{P} f^\e\!\|_{L^6_x L^2_v}
+ \!\|\nabla^{N_0}_xB^\e \|_{L^2_x}\| \mathbf{P} f^\e\|_{L^\infty_x L^2_v}\Big]
\|\nabla^k_x \mathbf{P}f^\varepsilon\|_{L^2_{x, v}}\\
\lesssim\;&
\|\nabla_x B^\e\|_{H^{N_0-1}_x}\mathcal{D}^0_{N_0}(t).
\eals
Applying the H\"{o}lder inequality  and the   Cauchy--Schwarz inequality with $\eta$, we obtain that
\bals
\mathcal{M}_{3,3}
\lesssim\;&
\Big[\|\nabla_xB^\e \|_{L^6_x}\|\langle v\rangle \nabla_v \nabla^{k-1}_x f^\e\|_{L^3_xL^2_v}
+\sum_{2\leq j\leq N_0-2}\|\nabla^{j}_xB^\e \|_{L^6_x}\|\langle v\rangle \nabla_v\nabla^{k-j}_x  f^\e\|_{L^3_xL^2_v}
\\
\;&
+ \|\nabla^{N_0-1}_xB^\e \|_{L^2_x}\|\langle v\rangle \nabla_v  \nabla_x f^\e\|_{L^\infty_xL^2_v}
+ \|\nabla^{N_0}_xB^\e \|_{L^2_x}\|\langle v\rangle \nabla_v   f^\e\|_{L^\infty_xL^2_v}\Big]
\|\nabla^k_x (\mathbf{I}\!-\!\mathbf{P})f^\varepsilon\|_{L^2_{x,v}}\\
\lesssim\;&
\left[{\overline{\mathcal{E}}}_{N-1,l_1}(t)+\mathcal{E}_{N}(t)\right]\e^2\sum_{3\leq \ell \leq N_0-1}\|\nabla^{\ell}_xB^\e \|_{L^2_x}^2
+\left[ \mathcal{E}_{N}(t)+{\overline{\mathcal{E}}}_{N-1,l_1}(t) +\mathcal{E}_{-s}(t)\right] \mathcal{D}_{N_0}^1(t)\\
\;&+
\frac{\eta}{\e^2}\|\nabla^k_x (\mathbf{I}-\mathbf{P})f^\varepsilon\|_{L^2_{x,v}}^2
\eals
holds for $2\leq k\leq N_0$. Here,   we have used   the facts, which are derived from \eqref{fdefenjie}, the interpolation inequality \eqref{sobolev interpolation1} and the relation of $N-N_0\geq 2$ because of  \eqref{hard assumption}, that
\bals
  \;&\e^2\|\nabla_xB^\e \|_{L^6_x}^2\|\langle v\rangle \nabla_v \nabla^{k-1}_x f^\e\|_{L^3_xL^2_v}^2\\
     \lesssim \;&\e^2\|\nabla_xB^\e \|_{L^6_x}^2\|\langle v\rangle \nabla_v \nabla^{k-1}_x (\mathbf{I}-\mathbf{P}) f^\e\|_{L^3_xL^2_v}^2+
   \e^2\|\nabla_xB^\e \|_{L^6_x}^2\|  \nabla^{k-1}_x \mathbf{P} f^\e\|_{L^3_xL^2_v}^2\\
     \lesssim\;&\e^2 \left\| \Lambda^{-1}B^\varepsilon \right\|_{L^2_x}^{\frac{1}{2}}
     \left\|\nabla_x ^3  B^\varepsilon \right\|_{L^2_{x,v}}^{\frac{3}{2}}
     \|\langle v\rangle^2 \nabla_v \nabla^{k-1}_x (\mathbf{I}-\mathbf{P})f^\e\|_{H^1_{x}L^2_{v}}
       \| \nabla_v^2\nabla^{k-1}_x(\mathbf{I}-\mathbf{P}) f^\e\|_{H^1_{x}L^2_{v}}^{\frac{1}{2}}\\
       \;&\times
       \| \nabla^{k-1}_x (\mathbf{I}-\mathbf{P})f^\e\|_{H^1_{x}L^2_{v}}^{\frac{1}{2}}
      +\e^2\|\nabla_x B^\e\|_{H^{N_0-1}_x}^2\mathcal{D}^0_{N_0}(t)
    \\
  \lesssim \;& \big[{\overline{\mathcal{E}}}_{N-1,l_1}(t)+\mathcal{E}_{-s}(t)\big] \mathcal{D}_{N_0}^1(t)+\e^2\|\nabla_x B^\e\|_{H^{N_0-1}_x}^2\mathcal{D}^0_{N_0}(t),
 \eals
and
\bals
\;&\e^2\|\nabla^{N_0}_xB^\e \|_{L^2_x}\|\langle v\rangle \nabla_v   f^\e\|_{L^\infty_xL^2_v}^2\\
\lesssim\;&\e^2\left\| \nabla^{N_0}_x B^\varepsilon \right\|_{L^2_x}^2
\left[\left\| \mathbf{P} f^\varepsilon \right\|_{L^\infty_x L^2_v}^2+
\left\| \langle v \rangle \nabla_v\nabla_x(\mathbf{I}-\mathbf{P}) f^\varepsilon \right\|_{ L^2_{x,v}}\left\| \langle v \rangle \nabla_v\nabla_x^2(\mathbf{I}-\mathbf{P}) f^\varepsilon \right\|_{L^2_{x,v}}\right]\\
\lesssim\;&\e^2\left\| \nabla^{N_0-1}_x B^\varepsilon \right\|_{L^2}^{\frac{2(N-N_0)}{N-N_0+1}}
 \left\| \nabla^{N}_x B^\varepsilon \right\|_{L^2_x}^{\frac{2}{N-N_0+1}}
 \Big[ \left\|  \nabla_x \mathbf{P} f^\varepsilon \right\|_{L^2_{x,v}}
\left\|  \nabla_x^2 \mathbf{P} f^\varepsilon \right\|_{L^2_{x,v}}
\\
\;&
\qquad+\left\|\nabla_v^3\nabla_x(\mathbf{I}-\mathbf{P}) f^\varepsilon \right\|_{  H^1_{x}L^2_{v}}^{\frac{1}{3}}
\left\| \nabla_x(\mathbf{I}-\mathbf{P}) f^\varepsilon \right\|_{ H^1_{x}L^2_{v}}^{\frac{2}{3}}\left\| \langle v \rangle^2 \nabla_v\nabla_x(\mathbf{I}-\mathbf{P}) f^\varepsilon \right\|_{ H^1_{x}L^2_{v}}\Big]
\\
\lesssim\;&  \left[ \mathcal{E}_{N}(t)+{\overline{\mathcal{E}}}_{N-1,l_1}(t) \right] \mathcal{D}_{N_0}^1(t).
\eals
In the case of $k=1$, by employing similar argument, we can deduce
\bals
\mathcal{M}_{3,3}
\lesssim\;&
\Big[\|\nabla_xB^\e \|_{L^\infty_x}\|\langle v\rangle \nabla_v(\mathbf{I}-\mathbf{P}) f^\e\|_{L^2_{x,v}}
+\|\nabla_xB^\e \|_{L^3_x}\|\mathbf{P} f^\e\|_{L^6_{x}L^2_{v}}\Big]\|\nabla_x (\mathbf{I}-\mathbf{P})f^\varepsilon\|_{L^2_{x,v}}
\\
\lesssim\;&  \big[~{\overline{\mathcal{E}}}_{N-1,l_1}(t)+\mathcal{E}_{-s}(t)\big] \mathcal{D}_{N_0}^1(t) +\|\nabla_x B^\e\|_{H^{N_0-1}_x}\mathcal{D}^0_{N_0}(t).
\eals

Consequently, gathering the above estimates of  $\mathcal{M}_{3,i}(i=1,2,3)$
and using the \emph{a priori} assumption  \eqref{priori assumption} ,
we arrive at the following outcome
\bals
\mathcal{M}_{3}
\lesssim\;& \delta_0 \mathcal{D}^1_{N_0}(t)+\eta \mathcal{D}^1_{N_0}(t) +\|\nabla_x B^\e\|_{H^{N_0-1}_x}\mathcal{D}^0_{N_0}(t).
\eals

For the term $ \mathcal{M}_4$,  making use of the H\"{o}lder inequality, the Sobolev inequality \eqref{NGinequality} and the Cauchy--Schwarz inequality, we obtain
\begin{align*}
 \mathcal{M}_4
\lesssim\;& \frac{1}{\varepsilon}\Big[\! \|  f^\varepsilon \|_{L^\infty_x L^2_v}\| \nabla_x^k  f^\varepsilon \|_{L^2_{x,v}(\nu)}+
\!\sum_{1 \leq j \leq k-1}\!\!
\| \nabla^{k-j}_x  f^\varepsilon \|_{L^3_x L^2_v}\| \nabla^j_x  f^\varepsilon \|_{L^6_x L^2_v(\nu)}\!\Big]
\| \nabla^k _x(\mathbf{I}\!-\!\mathbf{P})f^\varepsilon \|_{L^2_{x,v}(\nu)}\\
\lesssim\;&\mathcal{E}^{1/2}_{N}(t) \Big\{ \sum_{2 \leq \ell \leq N_0} \| \nabla^\ell_x  f^\varepsilon\|_{L^2_{x,v}(\nu)}^2
+\frac{1}{\varepsilon^2} \| \nabla^k _x (\mathbf{I}-\mathbf{P})f^\varepsilon\|_{L^2_{x,v}(\nu)} ^2\Big]\\
\lesssim\;&\delta_0 \mathcal{D}_{N_0}^1(t) \qquad \qquad \qquad \qquad \qquad \qquad \qquad \qquad \qquad   \qquad \qquad \qquad \quad \,\,\,\text{for}\quad k\geq2,
\end{align*}
and
\begin{align*}
 \mathcal{M}_4
\lesssim\;&\frac{1}{\e}
\left\|  f^\varepsilon \right\|_{L^3_x L^2_v}\left\| \nabla_x  f^\varepsilon \right\|_{L^6_x L^2_v(\nu)}
\left\| \nabla _x(\mathbf{I}-\mathbf{P})f^\varepsilon \right\|_{L^2_{x,v}(\nu)}\lesssim\delta_0 \mathcal{D}_{N_0}^1(t)
\quad  \text{for}\quad k=1.
\end{align*}

 Therefore, putting  all the estimates mentioned above into
 \eqref{pure spatial estimate} and summing the resulting inequality over $1\leq k\leq N_0$, we directly conclude \eqref{1-N estimate decay}.
Furthermore, under the \emph{a priori} assumption \eqref{priori assumption}, by combining \eqref{without weight-basic-result} and  \eqref{1-N estimate decay}, the desired estimate  \eqref{0-N estimate decay} follows.
This completes the proof of Lemma \ref{0-N 1-N lemma}.
\end{proof}

Now we complete the Proof of Proposition \ref{k-N decay proposition}.
\begin{proof}[\textbf{Proof of Proposition \ref{k-N decay proposition}}]
We divide it into two steps.

\medskip
\emph{Step 1. The uniform energy  estimate $\mathcal{E}^k_{N_0}(t)$ with $k=0,1$.}\;
Similar to Lemma \ref{macroscopic estimate}, by  considering the \emph{a priori} assumption \eqref{priori assumption}, we obtain that
there exists an  interactive energy functional $\mathcal{E}^{N_0}_{int}(t)$ satisfying
\begin{align}
\begin{split}\nonumber
\left|\mathcal{E}^{N_0}_{int}(t)\right| \lesssim\;&
\sum_{|\alpha| \leq N_0} \left\| \partial^\alpha_x f^{\varepsilon}\right\|^2_{L^2_{x,v}}
+\sum_{|\alpha| \leq N_0} \left\| \partial^\alpha_x E^{\varepsilon}\right\|^2_{L^2_{x}}+\sum_{|\alpha| \leq N_0} \left\| \partial^\alpha_x B^{\varepsilon}\right\|^2_{L^2_{x}},
\end{split}
\end{align}
such that
\begin{align}
\begin{split}\label{0-N macro 1}
&\varepsilon \frac{\d}{\d t} \mathcal{E}^{N_0}_{int}(t)+
\sum_{1 \leq k \leq N_0}\! \! \| \nabla^k_x \mathbf{P}f^\varepsilon\|_{L^2_{x,v}}^2+\e^2\sum_{1 \leq k \leq N_0-1} \| \nabla^k_x E^\varepsilon\|_{L^2_{x}}^2
+\e^2\sum_{2 \leq k \leq N_0-1}\! \!\| \nabla^k_x B^\varepsilon\|_{L^2_{x}}^2\\
\lesssim\;& \frac{1}{\varepsilon^2} \sum_{|\alpha| \leq N_0}\left\| \partial^\alpha_x (\mathbf{I}-\mathbf{P})f^{\varepsilon}\right\|^2_{L^2_{x,v}(\nu)}
+\delta_0 \mathcal{D}^0_{N_0}(t)+\eta \mathcal{D}^0_{N_0}(t).
\end{split}
\end{align}
Therefore, by recalling \eqref{low k energy} for $\mathcal{E}^k_{N_0}(t)$ when $k=0$ and combining the energy estimate \eqref{0-N estimate decay} and the macroscopic estimate \eqref{0-N macro 1},
we conclude \eqref{0-N estimate}.

Further,
applying the analogous argument as Lemma \ref{macroscopic estimate},
we can also derive that
there exists an  interactive energy functional $\mathcal{E}^{1 \rightarrow N_0}_{int}(t)$ satisfying
\begin{align}
\begin{split}\nonumber
\left| \mathcal{E}^{1 \rightarrow N_0}_{int}(t) \right| \lesssim\;& \sum_{1 \leq |\alpha| \leq N_0} \left\| \partial^\alpha_x f^{\varepsilon}\right\|^2_{L^2_{x,v}} +\sum_{1\leq|\alpha| \leq N_0} \left\| \partial^\alpha_x E^{\varepsilon}\right\|^2_{L^2_{x}}+\sum_{1\leq|\alpha| \leq N_0} \left\| \partial^\alpha_x B^{\varepsilon}\right\|^2_{L^2_{x}},
\end{split}
\end{align}
such that
\begin{align}
\begin{split} \label{1-N macro 1}
&\varepsilon \frac{\d}{\d t} \mathcal{E}^{1 \rightarrow N_0}_{int}(t) + \sum_{2 \leq k \leq N_0} \! \!\| \nabla^k_x \mathbf{P}f^\varepsilon\|_{L^2_{x,v}}^2
+\e^2\!\!\sum_{2 \leq k \leq N_0-1}\!\! \| \nabla^k_x E^\varepsilon\|_{L^2_{x}}^2
+\e^2\sum_{3 \leq k \leq N_0-1} \! \!\| \nabla^k_x B^\varepsilon\|_{L^2_{x}}^2\\
\lesssim\;& \frac{1}{\varepsilon^2} \!\!\sum_{1\leq|\alpha| \leq N_0}\!\!\left\| \partial^\alpha_x (\mathbf{I}-\mathbf{P})f^{\varepsilon}\right\|^2_{L^2_{x,v}(\nu)}
+\delta_0 \mathcal{D}^1_{N_0}(t)+\eta \mathcal{D}^1_{N_0}(t).
\end{split}
\end{align}
Consequently,  recalling \eqref{low k energy} for $\mathcal{E}^k_N(t)$ when $k=1$, the desired estimate \eqref{1-N estimate}  follows from the proper linear combination of \eqref{1-N estimate decay} and  \eqref{1-N macro 1}.

\emph{Step 2. The time decay estimate \eqref{k-N decay}.}\;
Indeed,   from Lemma \ref{sobolev interpolation}, we deduce that
\begin{align*}
\left\|   \mathbf{P}f^\varepsilon \right\|_{L^2_{x,v}}
\lesssim\;& \left\| \nabla^{2}_x\mathbf{P}f^\varepsilon \right\|_{L^2_{x,v}}^{\frac{s}{2+s}}
\left\| \Lambda^{-s} \mathbf{P}f^\varepsilon \right\|_{L^2_{x,v}}^{\frac{2}{2+s}},\\
\left\| [ {E}^\varepsilon,\, {B}^\varepsilon] \right\|_{L^2_{x}}
\lesssim\;& \e^{-\frac{s}{2+s}}\left\| \nabla^{2}_x [\e {E}^\varepsilon,\,\e {B}^\varepsilon] \right\|_{L^2_{x}}^{\frac{s}{2+s}}
\left\| \Lambda^{-s} [{E}^\varepsilon,\,{B}^\varepsilon] \right\|_{L^2_{x}}^{\frac{2}{2+s}},\\
\left\|    \nabla_x  {B}^\varepsilon  \right\|_{L^2_{x}}
\lesssim\;& \e^{-\frac{s}{2+s}}\left[\e\left\| \nabla^{2}_x  {B}^\varepsilon  \right\|_{L^2_{x}}\right]^{\frac{s}{ 2+s}}
 \| \Lambda^{1-s/2}  {B}^\varepsilon   \|_{L^2_{x}}^{\frac{2}{ 2+s}},\\
\left\| \nabla_x^{N_0} [ {E}^\varepsilon,\, {B}^\varepsilon] \right\|_{L^2_{x}}
\lesssim\;& \e^{-\frac{s}{2+s}}\left\| \nabla^{N_0-1}_x [\e {E}^\varepsilon,\,\e {B}^\varepsilon] \right\|_{L^2_{x}}^{\frac{s}{2+s}}
 \| \nabla^{N_0+\frac{s}{2}}_x [{E}^\varepsilon,\,{B}^\varepsilon] \|_{L^2_{x}}^{\frac{2}{2+s}}.
\end{align*}
The above inequalities together with the fact that for $0\leq m \leq N_0$
\begin{align}\nonumber
\left\| \nabla^m_x (\mathbf{I}-\mathbf{P})f^\varepsilon \right\|_{L^2_{x,v}}
\lesssim\;& \left( \frac{1}{\varepsilon} \left\| \nabla^m_x (\mathbf{I}-\mathbf{P})f^\varepsilon\right\|_{L^2_{x,v}(\nu)}\right)
^{\frac{ s}{2+s}}
\left( \varepsilon^{\frac{1}{2}}\left\| \langle v \rangle ^{-\frac{\gamma s}{2}}\nabla^m_x (\mathbf{I}-\mathbf{P})f^\varepsilon\right\|_{L^2_{x,v}}\right)
^{\frac{2}{2+s}}
\end{align}
imply
\begin{align} \label{decay:estimate:guocheng:1}
\mathcal{E}^0_{N_0}(t) \leq \e^{-\frac{2s}{2+s}} \left\{{\mathcal{E}}_{N_0+\frac{s}{2}}(t)+{\mathcal{E}}_{-s}(t)\right\}^{\frac{2}{2+s}}
\left\{ \mathcal{D}^0_{N_0}(t) \right\}^{\frac{ s}{2+s}}.
\end{align}
Furthermore, multiplying \eqref{0-N estimate}  by $\e^s(1+t)^{\frac{s}{2}+p}$ and integrating
the resulting inequality with respect to $t$ over $[0,t]$,
we arrive at
\bals
\;&\e^s(1+t)^{\frac{s}{2}+p} \mathcal{E}^0_{N_0}(t)+
\int_0^t
\e^s(1+\tau)^{\frac{s}{2}+p} \mathcal{D}^0_{N_0}(\tau)\dd \tau\\
\lesssim\;&
\e^s\mathcal{E}^0_{N_0}(0)
+\frac{s}{2}\int_0^t
\e^s(1+\tau)^{\frac{s}{2}-1+p} \mathcal{E}^0_{N_0}(\tau)\dd \tau\\
\lesssim\;&\mathcal{E}_{N}(0)+\eta\int_0^t
\e^{s}
(1+\tau)^{\frac{s}{2}+p} \mathcal{D}^0_{N_0}(\tau)\dd \tau
+\sup_{0\leq \tau\leq t}\left\{{\mathcal{E}}_{N_0+\frac{s}{2}}(\tau)+
{\mathcal{E}}_{-s}(\tau)\right\}(1+t)^p.
\eals
Here,   applied
\eqref{decay:estimate:guocheng:1} and the Cauchy-Schwartz inequality with $\eta$,  the fact utilized  is that
\bals
\bsp
\;&\int_0^t\e^{s}(1+\tau)^{\frac{s}{2}+p-1} \mathcal{E}^0_{N_0}(\tau)\dd \tau\\
\lesssim\;&\int_0^t
\left[\e^{s}
(1+\tau)^{\frac{s}{2}+p} \mathcal{D}^0_{N_0}(\tau)\right]^{\frac{s}{2+s}}
(1+\tau)^{(\frac{s}{2}+p)\frac{2}{2+s}-1}
\left\{{\mathcal{E}}_{N_0+\frac{s}{2}}(t)+
{\mathcal{E}}_{-s}(t)\right\}^{\frac{2}{2+s}}
\dd \tau\\
\lesssim\;&\eta\int_0^t
\e^{s}
(1+\tau)^{\frac{s}{2}+p} \mathcal{D}^0_{N_0}(\tau)\dd \tau
+\sup_{0\leq \tau\leq t}\left\{{\mathcal{E}}_{N_0+\frac{1+s}{2}}(\tau)+
{\mathcal{E}}_{-s}(\tau)\right\}\int_0^t(1+\tau)^{-1+p}
\dd \tau\\
\lesssim\;&\eta\int_0^t
\e^{s}
(1+\tau)^{\frac{s}{2}+p} \mathcal{D}^0_{N_0}(\tau)\dd \tau
+\sup_{0\leq \tau\leq t}\left\{{\mathcal{E}}_{N_0+\frac{s}{2}}(\tau)+
{\mathcal{E}}_{-s}(\tau)\right\}(1+t)^p.
\esp
\eals
Consequently, we can obtain
\bal
\bsp\label{decay:estimate:guocheng:2}
\;&\sup_{0\le\tau\le t}\left\{\e^s(1+t)^{\frac{s}{2}} \mathcal{E}^0_{N_0}(t)\right\}+
(1+t)^{-p}\int_0^t
\e^s(1+\tau)^{\frac{s}{2}+p} \mathcal{D}^0_{N_0}(\tau)\dd \tau\\
\lesssim\;&\mathcal{E}_{N}(0)
+\sup_{0\leq \tau\leq t}\left\{{\mathcal{E}}_{N_0+\frac{s}{2}}(\tau)+
{\mathcal{E}}_{-s}(\tau)\right\}.
\esp
\eal

On the other hand, multiply \eqref{1-N estimate}  by $\e^{1+s}(1+t)^{\frac{1+s}{2}+p}$ and integrate
the resulting inequality with respect to $t$ over $[0,t]$ to yield
\bal
\bsp\label{decay:estimate:guocheng:3}
\;&\e^{1+s}(1+t)^{\frac{1+s}{2}+p} \mathcal{E}^1_{N_0}(t)+\e^{1+s} \int_0^t  (1+\tau)^{\frac{1+s}{2}+p} \mathcal{D}^1_{N_0}(\tau)\dd \tau\\
\lesssim \;&\mathcal{E}_{N}(0)\!+\!
\e^{1+s}\!\!\int_0^t\!\!(1\!+\!\tau)^{\frac{s-1}{2}+p} \mathcal{E}^1_{N_0}(\tau)\dd \tau\!+\!
\e^{1+s} \!\!\int_0^t \!\! (1+\tau)^{\frac{1+s}{2}+p}\! \|\nabla_x B^\e(\tau)\|_{H^{N_0\!-\!1}_x}\mathcal{D}^0_{N_0}(\!\tau\!)\dd \tau.
\esp
\eal
Applying the analogous argument as the estimate \eqref{decay:estimate:guocheng:1} gives rise to
\begin{align} \nonumber
\mathcal{E}^1_{N_0}(t) \leq \e^{-\frac{2(1+s)}{3+s}} \left\{{\mathcal{E}}_{N_0+\frac{1+s}{2}}(t)+
{\mathcal{E}}_{-s}(t)\right\}^{\frac{2}{3+s}}
\left\{ \mathcal{D}^1_{N_0}(t) \right\}^{\frac{ 1+s}{3+s}}.
\end{align}
Then the second term on the right-hand side of  \eqref{decay:estimate:guocheng:3} can be inferred  from the above inequality  and the Cauchy-Schwartz inequality with $\eta$  that
\bal
\bsp\label{decay:estimate:guocheng:4}
\;&\e^{1+s}\int_0^t(1+\tau)^{\frac{s-1}{2}+p} \mathcal{E}^1_{N_0}(\tau)\dd \tau\\
\lesssim\;&\int_0^t\! \!
\left[\e^{1+s}
(1+\tau)^{\frac{s+1}{2}+p} \mathcal{D}^1_{N_0}(\tau)\right]^{\frac{1+s}{3+s}}
\! \!(1+\tau)^{(\frac{s+1}{2}+p)\frac{2}{3+s}-1}\! \!\left\{{\mathcal{E}}_{N_0+\frac{1+s}{2}}(t)+
{\mathcal{E}}_{-s}(t)\right\}^{\frac{2}{3+s}}
\dd \tau\\
\lesssim\;&\eta\int_0^t
\e^{1+s}
(1+\tau)^{\frac{1+s}{2}+p} \mathcal{D}^1_{N_0}(\tau)\dd \tau
+\sup_{0\leq \tau\leq t}\left\{{\mathcal{E}}_{N_0+\frac{1+s}{2}}(\tau)+
{\mathcal{E}}_{-s}(\tau)\right\}\! \!\int_0^t(1+\tau)^{-1+p}
\dd \tau\\
\lesssim\;&\eta\int_0^t
\e^{1+s}
(1+\tau)^{\frac{1+s}{2}+p} \mathcal{D}^1_{N_0}(\tau)\dd \tau
+\sup_{0\leq \tau\leq t}\left\{{\mathcal{E}}_{N_0+\frac{1+s}{2}}(\tau)+
{\mathcal{E}}_{-s}(\tau)\right\}(1+t)^p.
\esp
\eal
For the third term on the right-hand side of  \eqref{decay:estimate:guocheng:3}, we obtain from the \emph{a priori} assumption  \eqref{priori assumption}  that
\bal
\bsp\label{decay:estimate:guocheng:5}
 \e^{1+s} \int_0^t  (1+\tau)^{\frac{1+s}{2}+p} \|\nabla_x B^\e(\tau)\|_{H^{N_0-1}_x}\mathcal{D}^0_{N_0}(\tau)\dd \tau
\lesssim\;& \left[X(t)\right]^{\frac{1}{2}}\e^{s}\int_0^t  (1+\tau)^{\frac{s}{2}+p}
\mathcal{D}^0_{N_0}(\tau)\dd \tau\\
\lesssim\;& \delta_0^{\frac{1}{2}}\e^{s}\int_0^t  (1+\tau)^{\frac{s}{2}+p}
\mathcal{D}^0_{N_0}(\tau)\dd \tau.
\esp
\eal
As a consequence,  plugging \eqref{decay:estimate:guocheng:4} and \eqref{decay:estimate:guocheng:5} into  \eqref{decay:estimate:guocheng:3},  and choosing $\eta$ small enough, we obtain
\bal
\bsp\label{decay:estimate:guocheng:6}
\;&\sup_{0\le\tau\le t}\left\{\e^{1+s}(1+t)^{\frac{1+s}{2}} \mathcal{E}^1_{N_0}(t)\right\}+
(1+t)^{-p}\int_0^t
\e^{1+s}(1+\tau)^{\frac{1+s}{2}+p} \mathcal{D}^1_{N_0}(\tau)\dd \tau\\
\lesssim\;& \mathcal{E}_{N}(0)
+\sup_{0\leq \tau\leq t}\left\{{\mathcal{E}}_{N_0+\frac{1+s}{2}}(\tau)+
{\mathcal{E}}_{-s}(\tau)\right\}+\delta_0^{\frac{1}{2}}\e^{s}\int_0^t  (1+\tau)^{\frac{s}{2}+p}
\mathcal{D}^0_{N_0}(\tau)\dd \tau.
\esp
\eal
Therefore, a suitable linear combination of \eqref{decay:estimate:guocheng:2}
 and \eqref{decay:estimate:guocheng:6} yields the desired estimate \eqref{k-N decay}.
This completes the proof of Proposition \ref{k-N decay proposition}.
\end{proof}

\subsection{Proof of Global Existence}
\hspace*{\fill}

With Proposition \ref{result-basic energy estimate}, Proposition \ref{weighted 1}, Proposition \ref{weighted 2:zongguji}, Proposition \ref{negative sobolev energy estimates total} and Proposition \ref{k-N decay proposition} in hand, we  are now  in the position to complete the proof of Theorem
\ref{mainth1} and verify the {\emph{a priori}} assumption  \eqref{priori assumption}.
\begin{proof}[{\bf Proof of Theorem \ref{mainth1}.}] \ 
Collecting \eqref{weight estimate1}, \eqref{estimate-weighted-2:zong} and \eqref{negative sobolev estimate total}, then choosing $\delta_0$ small enough and using the relation $ \frac{1+s}{2}(l_2-l_1) =\frac{1+\vartheta}{2}(N-1)=\sigma_{N-1, N-1} $, we find
\bals
\bsp
\;&\frac{\d}{\d t}\Big[{\overline{\mathcal{E}}}_{N-1,l_1}(t)+ {\widetilde{\mathcal{E}}}_{N,l_2}(t)+\mathcal{E}_{-s}(t)+\varepsilon \mathcal{E}^{-s}_{int}(t)\Big]
+{\overline{\mathcal{D}}}_{N-1,l_1}(t)
+{\widetilde{\mathcal{D}}}_{N,l_2}(t)+\mathcal{D}_{-s}(t)\\
\lesssim\;&
\delta_0(1+t)^{-\sigma_{N-1, N-1}} \frac{1}{\e^2}
 \sum_{|\a|+|\b|\leq N-1,|\a|\geq 1}
 \!\!\!\!
\big\|\widetilde{w}_{l_2}(|\a|-1,|\b|+1)\nabla_v\partial_{\b}^{\a-\a'}(\mathbf{I}-\mathbf{P})
f^\varepsilon\big\|_{L^2_{x,v}}^2\\
\;&+\mathcal{D}_{N}(t)+\eta\e^2(1+t)^{-(1+\vartheta)}\|\nabla_x^N E^{\e}\|_{L^2_x}^2 \\[1mm]
\lesssim\;&
\delta_0\sum_{n\leq N-1}\widetilde{\mathcal{D}}_{l_2}^{(N)}(t)
 +\mathcal{D}_{N}(t)+\eta\e^2(1+t)^{-(1+\vartheta)}\|\nabla_x^N E^{\e}\|_{L^2_x}^2,
\esp
\eals
which further implies
\bal
\bsp\label{estimate-jielun-2}
\;&\frac{\d}{\d t}\Big[{\overline{\mathcal{E}}}_{N-1,l_1}(t)+{\widetilde{\mathcal{E}}}_{N,l_2}(t)+\mathcal{E}_{-s}(t)+\varepsilon \mathcal{E}^{-s}_{int}(t)\Big]
+{\overline{\mathcal{D}}}_{N-1,l_1}(t)
+{\widetilde{\mathcal{D}}}_{N,l_2}(t)+\mathcal{D}_{-s}(t)\\
\lesssim\;&
 \mathcal{D}_{N}(t)+\eta\e^2(1+t)^{-(1+\vartheta)}\|\nabla_x^N E^{\e}\|_{L^2_x}^2.
\esp
\eal
Therefore, a proper  linear combination of \eqref{basic-energy-estimate-result}
and \eqref{estimate-jielun-2} gives us that
\bal
\bsp\label{estimate-jielun-3}
\sup_{0\leq t\leq T}\Big\{ \mathcal{E}_{N}(t)+{\overline{\mathcal{E}}}_{N-1,l_1}(t)+
{\widetilde{\mathcal{E}}}_{N,l_2}(t)+\mathcal{E}_{-s}(t)\Big\}
\lesssim\;&
 \eta\e^2 \sup_{0\leq t\leq T}\left\{ \mathcal{E}_{N}(t)\right\}.
\esp
\eal
Combined     with \eqref{k-N decay}, the estimate (\ref{estimate-jielun-3}) helps us to conclude
\bal
\bsp\label{estimate-jielun-4}
X(t) \lesssim \mathcal{E}_{N}(0)+{\overline{\mathcal{E}}}_{N-1,l_1}(0)+ {\widetilde{\mathcal{E}}}_{N,l_2}(0)+\mathcal{E}_{-s}(0)
\esp
\eal
holds for any $0 \leq t \leq T$, as long as $\mathcal{E}_{N}(0)+{\overline{\mathcal{E}}}_{N-1,l_1}(0)+ {\widetilde{\mathcal{E}}}_{N,l_2}(0)+\mathcal{E}_{-s}(0)$ is sufficiently small.
Furthermore, \eqref{estimate-jielun-4} further  implies the \emph{a priori} assumption  \eqref{priori assumption}.

The rest is to prove the local existence and uniqueness of solutions in terms of the energy norm $$\mathcal{E}_{N}(t)+{\overline{\mathcal{E}}}_{N-1,l_1}(t)+ {\widetilde{\mathcal{E}}}_{N,l_2}(t)+\mathcal{E}_{-s}(t)$$ and the non-negativity of $F^\varepsilon=\mu+\varepsilon \mu^{1/2}f^\varepsilon$. One can use the iteration method with the iterating sequence $\left[f_n^\varepsilon, E^\e_n, B^\e_n\right]$ $(n \geq 0)$ on the system
\begin{align}
	\left\{\begin{array}{l}
\displaystyle \partial_{t} f_{n+1}^{\varepsilon}+\frac{1}{\varepsilon}v \cdot \nabla_{x} {f}_{n+1}^{\varepsilon}-E_{n+1}^{\varepsilon} \cdot v \mu^{1/2} +\dfrac{1}{\varepsilon^{2}} L {f}_{n+1}^{\varepsilon}
\\[3mm]
\qquad\qquad\qquad\qquad= -\left(\e E_{n}^{\varepsilon}+v\times B^\e_n\right) \cdot \nabla_{v} {f}_{n+1}^{\varepsilon}-\frac{\e}{2} E_{n}^{\varepsilon} \cdot v {f}_{n+1}^{\varepsilon}+\frac{1}{\varepsilon} \Gamma\left({f}_{n}^{\varepsilon}, {f}_{n}^{\varepsilon}\right),  \\ [2mm]
\displaystyle  \pt_tE^\varepsilon_{n+1}-\nabla_x\times B^\varepsilon_{n+1}=
    -\int_{\mathbb{R}^3}v \mu^{1/2}f^\varepsilon_{n+1}\d v, \qquad\;
	\nabla_x\cdot E^\varepsilon_{n+1}=\varepsilon \int_{\mathbb{R}^3}  \mu^{1/2}f^{\varepsilon}_{n+1}\d v, \\ [3mm]
 \displaystyle
    \pt_tB_{n+1}^\varepsilon+\nabla_x\times E^\varepsilon_{n+1}=
    0, \qquad\qquad\qquad\quad	\qquad\;\;\;		
	\nabla_x\cdot B^\varepsilon_{n+1}=0, \\ [3mm]
\displaystyle f_{n+1}^\varepsilon(0,x,v)=f^\varepsilon_0(x,v)
\quad E^\varepsilon_{n+1}(0,x)=E^\varepsilon_{0}(x), \quad B^\varepsilon_{n+1}(0,x)=B^\varepsilon_{0}(x).
	\end{array}\right.
\end{align}
The details of proof are omitted for brevity; see \cite{Guo2003, JL2019}. Therefore, the global existence of solutions follows with the help of the continuity argument, and the estimate \eqref{thm1 N l0 decay} and \eqref{thm1 widetilde N l0+l1 decay} hold by the definition of $X(t)$ in \eqref{X define}. This completes the proof of Theorem \ref{mainth1}.
\end{proof}

\section{Limit to the Incompressible NSFM System}\label{Limit section}

In this section, our goal is to derive the  incompressible NSFM system   from the VMB system \eqref{rVPB} as $\varepsilon \to  0$.   The proof   is similar to that one in \cite{JL2019}, but here  we  provide a full proof for completeness.

\subsection{Limits from the Energy Estimate}
\hspace*{\fill}

Based on   Theorem \ref{mainth1},  the system (\ref{rVPB})
admits a global solution $f^\varepsilon \in L^\infty(\mathbb{R}^+; H^N_x L^2_v)$ and $\left[E^\varepsilon,\, B^\varepsilon\right]\in L^\infty(\mathbb{R}^+; H^{N}_x)$, so there exists a positive constant $C$ which is independent of $\varepsilon$, such that
\begin{align}
&\sup_{t\geq 0} \left\{ \left\| f^{\varepsilon}(t) \right\|^2_{H^N_x L^2_v}+ \left\| E^{\varepsilon}(t) \right\|^2_{H^{N}_x}+\left\| B^{\varepsilon}(t) \right\|^2_{H^{N}_x}\right\} \leq  C,
\label{limit:1-1}\\
&\int_0^\infty \left\| (\mathbf{I}-\mathbf{P}){f}^{\varepsilon}(\tau) \right\|^2_{H^{N}_xL^2_v(\nu)}\d \tau\leq C \varepsilon^2. \label{limit:1-2}
\end{align}
From the energy bound \eqref{limit:1-1}, there exist $f\in L^\infty(\mathbb{R}^+; H^N_x L^2_v)$ and $\nabla_x\phi\in L^\infty(\mathbb{R}^+;H^{N+1}_x)$  such that
\begin{align}
f^{\varepsilon} &\to  f \quad \qquad\;\text{~~weakly}\!-\!* ~\text{in}~ L^\infty(\mathbb{R}^+; H^N_x L^2_v), \label{limit:3-1}\\
E^{\varepsilon}  &\to E\qquad \quad\text{~~weakly}\!-\!* ~\text{in} ~L^\infty(\mathbb{R}^+;H^{N}_x), \label{limit:3-2}\\
B^{\varepsilon}  &\to B\qquad \quad\text{~~weakly}\!-\!* ~\text{in} ~L^\infty(\mathbb{R}^+;H^{N}_x) \label{limit:3-3}
\end{align}
as $\varepsilon \to 0$. We still employ the original notation of sequences to denote the subsequences for convenience,
although the limit may hold for some subsequences.
From the dissipation bound \eqref{limit:1-2}, we deduce that
\begin{align}
\label{limit:4}
&(\mathbf{I}-\mathbf{P})f^{\varepsilon} \to  0 \;\;\;\text{~~strongly} {\text{~in}~L^2(\mathbb{R}^+;H^{N}_{x}L^2_v)}
\end{align}
as $\varepsilon \to 0$. We thereby deduce from the convergences \eqref{limit:3-1} and \eqref{limit:4} that
\begin{align}
(\mathbf{I}-\mathbf{P})f=0.
\end{align}
This implies that there exist functions $\rho,u,\theta\in L^\infty(\bbR^+;H^2_{x})$ such that
\bal
\bsp\label{limit:6}
f(t,x,v)=\Big[\rho(t,x)+v \cdot u(t,x) +\frac{|v|^2 -3}{2}\th(t,x)\Big]{\mu}^{1/2}.
\esp
\eal

Next, we introduce the following fluid variables
\bal
\bsp\label{limit:7}
&\rho^\e:=\left\langle f^{\e},\mu^{1/2}\right\rangle_{L^2_{v}},\;\;\;\; u^\e:=\left\langle f^{\e},v\mu^{1/2}\right\rangle_{L^2_{v}},\;\;\;\; \theta^\e :=\Big\langle f^{\e}, \Big(\frac{|v|^2}{3}-1\Big)\mu^{1/2}\Big\rangle_{L^2_{v}}.
\esp
\eal
Multiplying the first equation of \eqref{rVPB} by  $\mu^{1/2}$, $  v\mu^{1/2}$, $(\frac{|v|^2}{3}-1)\mu^{1/2}$ respectively  and integrating  over $\mathbb{R}^3_v$,  we derive that $\rho^\e$, $u^\e$ and $\th^\e$ obey the local conservation laws
\beq
 \left\{
\begin{array}{ll}\label{limit:8-1:1}
\displaystyle
\pt_t \rho^\e+\frac{1}{\e}\nabla_x\cdot u^\e=0,~\\[2mm]
\displaystyle\pt_t u^\e+\frac{1}{\e}\nabla_x(\rho^\e+\theta^\e)-E^\e+\frac{1}{\e}
\nabla_x\cdot \left\langle \widehat{A}(v){\mu}^{\frac{1}{2}}, L(\mathbf{I}-\mathbf{P})f^{\e}\right\rangle_{L^2_{v}}=\e\rho^\e E^{\e}+u_\e\times B^\e,~\\[2mm]
\displaystyle\pt_t \theta^\e+\frac{2}{3\e}\nabla_x\cdot u^\e+\frac{2}{3\e}\nabla_x\cdot\left\langle \widehat{B}(v){\mu}^{\frac{1}{2}}, L(\mathbf{I}-\mathbf{P})f^{\e}\right\rangle_{L^2_{v}}=\frac{2\e}{3}u^\e\cdot E^\e,
\end{array}
\right.
\eeq
where
$
A (v)=v\otimes v -\frac{|v|^2}{3}\mathbb{I}_{3\times 3},~ B (v)=v\Big(\frac{|v|^2}{2}-\frac{5}{2}\Big),
$
$\widehat{A} (v)$ is such that ${L}\big( \widehat{A}(v) \mu^{1/2}\big)= A(v) \mu^{1/2}$ with $\widehat{A}(v) \mu^{1/2} \in \mathrm{N}^{\bot}({L})$,
$\widehat{B} (v)$ is such that ${L}\big( \widehat{B}(v) \mu^{1/2}\big)$ $= B(v) \mu^{1/2}$ with $\widehat{B}(v) \mu^{1/2} \in \mathrm{N}^{\bot}({L})$.

From the second and third equations of \eqref{rVPB}  and the definition of $\rho^\varepsilon, u^\e$ in \eqref{limit:7}, we also find the equations of $E^{\varepsilon}$ and $B^{\varepsilon}$
\begin{align}\label{limit:8-1:4}
	\left\{\begin{array}{l}
\displaystyle \pt_tE^\varepsilon-\nabla_x\times B^\varepsilon=
    -u^\e, \quad\qquad\;\,\,
	\nabla_x\cdot E^\varepsilon=\e\rho^\e, \\ [2mm]
 \displaystyle
    \pt_tB^\varepsilon+\nabla_x\times E^\varepsilon=
    0, \qquad\qquad	\;\;\;		
	\nabla_x\cdot B^\varepsilon=0.
	\end{array}\right.
\end{align}

Further, via the definitions of $\rho^\varepsilon$, $u^\varepsilon$ and  $\theta^\varepsilon$  in \eqref{limit:7} and the uniform energy bound \eqref{limit:1-1}, we obtain
\begin{align}\label{limit:9}
\sup_{t\geq 0}\Big\{\| \rho^\varepsilon(t)\|_{H^N_{x}}+ \| u^\varepsilon(t)\|_{H^N_{x}}+\| \theta^\varepsilon(t)\|_{H^N_{x}}
\Big\}\leq C.
\end{align}
From the convergence \eqref{limit:3-1} and the limit function $f(t,x,v)$ given in \eqref{limit:6}, we thereby deduce the following convergences
\bal
&\rho^{\e}=\left\langle f^{\e},\mu^{1/2}\right\rangle_{L^2_{v}}\to \left\langle f,\mu^{1/2}\right\rangle_{L^2_{v}}
=\rho,~\label{limit:10-1}\\
&u^{\e}=\left\langle f^{\e},v\mu^{1/2}\right\rangle_{L^2_{v}}\to \left\langle f,v\mu^{1/2}\right\rangle_{L^2_{v}}=u,~\label{limit:10-2}\\
&\theta^{\e} =\Big\langle f^{\e}, \Big(\frac{|v|^2}{3}-1\Big){\mu}^{\frac{1}{2}}\Big\rangle_{L^2_{v}}\to\Big\langle f, \Big(\frac{|v|^2}{3}-1\Big){\mu}^{\frac{1}{2}}\Big\rangle_{L^2_{v}}=\theta,\label{limit:10-3}
\eal
$\text{weakly}\!-\!*~ \text{in}~ L^\infty(\mathbb{R}^+,H^N_{x})$ as $\varepsilon \to 0$.

\subsection{Convergences to the Limiting System}
\hspace*{\fill}

In this subsection, we deduce the   incompressible NSFM system  from the local conservation laws \eqref{limit:8-1:1}, \eqref{limit:8-1:4}, the convergences \eqref{limit:3-1}--\eqref{limit:4} and \eqref{limit:10-1}--\eqref{limit:10-3}  obtained in the previous subsection.

\subsubsection{Incompressibility and Boussinesq relation.}
\hspace*{\fill}

From the first equation of \eqref{limit:8-1:1} and the energy uniform bound \eqref{limit:9}, it is easy to deduce
\begin{align}\nonumber
\nabla_x \cdot u^{\varepsilon }=-\varepsilon \partial_t \rho^{\varepsilon}\to 0  \;\;\; \text{~~in the sense of distributions}
\end{align}
as $\varepsilon \to 0$. Combining with the convergence \eqref{limit:10-2}, we get
\begin{align}\label{limit:13}
\nabla_x \cdot u = 0.
\end{align}
The second equation of \eqref{limit:8-1:1}  yields
\begin{align}\nonumber
&\nabla_x(\rho^\varepsilon +\theta^\varepsilon)=-\varepsilon \partial_t u^\varepsilon+\e E^\e
-\nabla_x\cdot \left\langle \widehat{A}(v){\mu}^{\frac{1}{2}}, L(\mathbf{I}-\mathbf{P})f^{\e}\right\rangle_{L^2_{v}}+ \e^2\rho^\e E^{\e}+\e u^\e\times B^\e.
\end{align}

The bound \eqref{limit:9} implies
$$
\varepsilon \partial_t u^\varepsilon \to 0 \;\;\; \text{~~in the sense of distributions}
$$
as $\varepsilon\to 0$.
With the aid of the self-adjointness of ${L}$, we derive from the dissipation bound \eqref{limit:1-2} and the  H{\"{o}}lder inequality that
\begin{align}
\begin{split}\nonumber
\int_0^\infty
\!\Big\|\nabla_x\cdot \left\langle \widehat{A}(v)\sqrt{\mu}, L(\mathbf{I}-\mathbf{P})f^{\e}\right\rangle_{L^2_{v}} \Big\|_{H^{N-1}_x}^2\d \tau
\!=\;&\int_0^\infty
\Big\|\nabla_x\cdot \Big\langle{ {A}}(v)\mu^{1/2}, {(\mathbf{I}-\mathbf{P})f^{\varepsilon} } \Big\rangle \Big\|_{H^{N-1}_x}^2\d \tau\\
\leq \;& C\int_0^\infty
\| (\mathbf{I}-\mathbf{P})f^\varepsilon(\tau)\|_{H^N_{x}L^2_v(\nu)}^2\d \tau\\
\leq\;&C\varepsilon^2.
\end{split}
\end{align}
Therefore, we have
\begin{align}\nonumber
\nabla_x\cdot \left\langle \widehat{A}(v)\sqrt{\mu}, L(\mathbf{I}-\mathbf{P})f^{\e}\right\rangle_{L^2_{v}} \to 0
\;\;\;\text{~~strongly} {\text{~in}~L^2(\mathbb{R}^+;H^{N-1}_{x})}
\end{align}
as $\varepsilon\to 0$.
Moreover, the uniform energy bounds \eqref{limit:1-1} and \eqref{limit:9} reveal that
\begin{align}\nonumber
\sup_{t \geq 0} \left\|\e^2\rho^\e E^{\e}+\e u_\e\times B^\e\right\|_{H^{N}_x}
\leq C \varepsilon ,
\end{align}
which means that
\begin{align}\nonumber
\e^2\rho^\e E^{\e}+\e u_\e\times B^\e\to 0
\;\;\;\text{~~strongly} {\text{~in}~L^\infty(\mathbb{R}^+;H^{N}_{x})}
\end{align}
as $\varepsilon \to 0$.
In summary, we have
\begin{align}\label{limit:15}
\nabla_x(\rho^\varepsilon + \theta^\varepsilon) \to 0 \;\;\; \text{~~in the sense of distributions}
\end{align}
as $\varepsilon \to 0$.
Therefore, combining the convergences \eqref{limit:10-1}, \eqref{limit:10-3} and \eqref{limit:15}, we obtain
\begin{align}\nonumber
\nabla_x(\rho+\theta)=0.
\end{align}
From the integrability of functions, we can deduce the Boussinesq relation
\begin{align}\label{limit:16}
\rho+\theta=0.
\end{align}

\subsubsection{ Convergences of  $\frac{3}{5} \theta^\varepsilon-\frac{2}{5}\rho^\varepsilon$, $\mathcal{P}u^\varepsilon$, $n^\varepsilon$, $E^\varepsilon$, $B^\varepsilon$.}
\hspace*{\fill}

First of all, we introduce the following Aubin--Lions--Simon Theorem, a fundamental result of compactness in the study of nonlinear evolution problems, which can be referred in \cite{BF13, Simon1987}.
\begin{lemma} [Aubin--Lions--Simon Theorem] \label{Aubin--Lions--Simon Theorem}
Let $B_0 \subset B_1 \subset B_2$ be three Banach spaces. We assume that the embedding of $B_1$ in $B_2$ is continuous and the embedding of $B_0$ in $B_1$ is compact. Let $p, r$ be such that $1 \leq p,r \leq \infty$. For $T>0$, we define
\begin{align}\nonumber
E_{p,r}=\Big\{ u \in L^p(0,T; B_0), \partial_t u \in L^r(0,T;B_2) \Big\}.
\end{align}
\begin{itemize}
\setlength{\leftskip}{-6mm}
\item[(1)] If $p < +\infty$, the embedding of $E_{p,r}$ in $L^p(0,T; B_1)$ is compact.
\item[(2)] If $p = +\infty$ and if $r > 1$, the embedding of $E_{p,r}$ in $C(0,T; B_1)$ is compact.
\end{itemize}
\end{lemma}

Now we consider the convergence of $\frac{3}{5}\theta^\varepsilon-\frac{2}{5}\rho^\varepsilon$. It follows from the first equation and the third equation of $\eqref{limit:8-1:1}$ that
\begin{align}\label{rho theta equation}
\partial_t\Big(\frac{3}{5}\theta^\varepsilon -\frac{2}{5}\rho^\varepsilon\Big)+\frac{2}{5\e}\nabla_x\cdot\left\langle \widehat{B}(v){\mu}^{\frac{1}{2}}, L(\mathbf{I}-\mathbf{P})f^{\e}\right\rangle_{L^2_{v}}
=\frac{2}{5}\e u^\e\cdot E^\e.
\end{align}
Thereby, by utilizing the equation \eqref{rho theta equation}, the energy bound \eqref{limit:1-1}, the dissipation bound \eqref{limit:1-2}, one has
\begin{align}
\begin{split}\label{partial rho theta 1}
	&\Big\|\partial_t\Big(\frac{3}{5}\theta^\varepsilon -\frac{2}{5}\rho^\varepsilon\Big)\Big\|_{L^2(0,T;H^{N-1}_x)}\\
	=\;&\Big\|\frac{2}{5}\e u^\e\cdot E^\e-\frac{2}{5\varepsilon} \nabla_x \cdot \langle B(v)\mu^{1/2},(\mathbf I-\mathbf P) f^{\varepsilon} \rangle \Big\|_{L^2(0,T;H^{N-1}_x)}\\
	\leq \;& C \varepsilon \left\|u^\varepsilon \right\|_{L^2(\mathbb{R}^{+};H^{N}_x)} \left\|E^{\varepsilon} \right\|_{L^\infty(\mathbb{R}^{+};H_x^{N})}+\frac{C}{\varepsilon}\left\| (\mathbf{I}-\mathbf{P}){f}^{\varepsilon}\right\|
_{L^2(\mathbb{R}^{+}; H^N_x L^2_v(\nu))}\\
\leq\;& C
\end{split}
\end{align}
for any $T > 0$ and $0 < \varepsilon \leq 1$. Further, we can derive from the energy bound \eqref{limit:9} that
\begin{align}\label{partial rho theta 2}
\Big\|\frac{3}{5}\theta^\varepsilon -\frac{2}{5}\rho^\varepsilon\Big\|_{L^\infty(0,T;H^{N}_x)} \leq C
\end{align}
for any $T > 0$ and $0 < \varepsilon \leq 1$. Notice the fact that
\begin{align}\label{embedding compact}
H^N  \hookrightarrow H^{N-1}_{loc} \hookrightarrow H^{N-1}_{loc},
\end{align}
where the embedding of $H^N$ in $H^{N-1}_{loc}$ is compact by the Rellich--Kondrachov Theorem (see Theorem 6.3 of \cite{AF2003}, for instance) and the embedding of $H^{N-1}_{loc}$ in $H^{N-1}_{loc}$ is naturally continuous. Then, from Aubin--Lions--Simon Theorem in Lemma \ref{Aubin--Lions--Simon Theorem}, the bounds \eqref{partial rho theta 1}, \eqref{partial rho theta 2} and the embedding \eqref{embedding compact}, we deduce that there exists a function $\widetilde{\theta} \in L^\infty(\mathbb{R}^+; H^N_x) \cap C(\mathbb{R}^+; H^{N-1}_{loc})$ such that
\begin{align}\nonumber
\frac{3}{5}\theta^\varepsilon-\frac{2}{5}\rho^\varepsilon \to \widetilde{\theta}
\;\;\;\text{~~strongly} {\text{~in}~C(0,T; H^{N-1}_{loc})}
\end{align}
for any $T > 0$ as $\varepsilon \to 0$. Combining with the convergences \eqref{limit:10-1} and \eqref{limit:10-3}, we deduce that $\widetilde{\theta}=\frac{3}{5}\theta-\frac{2}{5}\rho$. Moreover, it follows that $\widetilde{\theta}=\theta$ from the Boussinesq relation \eqref{limit:16} and $\theta=\frac{3}{5}\theta-\frac{2}{5}\rho+\frac{2}{5}(\rho+\theta)$. As a result, we obtain that
\begin{align}\nonumber
\frac{3}{5}\theta^\varepsilon-\frac{2}{5}\rho^\varepsilon \to \theta
\;\;\;\text{~~strongly} {\text{~in}~C(\mathbb{R}^+; H^{N-1}_{loc})}
\end{align}
as $\varepsilon \to 0$, where $\theta \in L^\infty(\mathbb{R}^+; H^N_x) \cap C(\mathbb{R}^+; H^{N-1}_{loc})$. Noticing that $\theta^\varepsilon=\frac{3}{5}\theta^\varepsilon-\frac{2}{5}\rho^\varepsilon+\frac{2}{5}(\rho^\varepsilon+\theta^\varepsilon)$ and combining with the convergence \eqref{limit:10-3}, one has
\begin{align}\nonumber
\rho^\varepsilon+\theta^\varepsilon \to 0
\;\;\; \text{weakly}\!-\!* \text{~in~} L^\infty(\mathbb{R}^+;H^N_{x}),\;
\text{~~strongly} \text{~in}~ C(\mathbb{R}^+;H^{N-1}_{loc})
\end{align}
as $\varepsilon \to 0$.

Next, we consider the convergence of $\mathcal{P}u^\varepsilon$, where $\mathcal{P}u^\varepsilon$ is the Leray projection operator defined by $\mathcal{P}= I-\nabla_x \Delta_x^{-1} \nabla_x \cdot$. By taking $\mathcal{P}$ on the second equation of \eqref{limit:8-1:1}, we have
\begin{align}\label{Pu equation}
&\partial_t \mathcal{P} u^\varepsilon +\frac{1}{\varepsilon}\mathcal{P}
\nabla_x\cdot \left\langle \widehat{A}(v){\mu}^{\frac{1}{2}}, L(\mathbf{I}-\mathbf{P})f^{\e}\right\rangle_{L^2_{v}} =\mathcal{P}(E^\e+\e\rho^\e E^{\e}+u_\e\times B^\e).
\end{align}

Hence, it follows from the equation \eqref{Pu equation}, the uniform dissipation bound  \eqref{limit:1-2}, the uniform energy bound \eqref{limit:9} and the H{\"{o}}lder inequality that
\begin{align}
\begin{split}\label{partial u 1}
&\left\|\partial_t \mathcal{P} u^\varepsilon \right\|_{L^2(0,T;H^{N-1}_x)}\\
\lesssim\:& \|E^\e\!+\!\e\rho^\e E^{\e}\!+\!u_\e\times B^\e\|_{L^2(0,T;H^{N-1}_x)}
+\Big\| \frac{1}{\varepsilon} \nabla_x \cdot \Big\langle {A}(v)\mu^{1/2}, (\mathbf{I}\!-\!\mathbf{P})f^{\varepsilon}\Big\rangle_{L^2_v} \Big\|_{L^2(0,T;H^{N-1}_x)}\\
\lesssim\:&
\left\|E^\varepsilon \right\|_{L^\infty(\mathbb{R}^+; H^{N}_x)}
+\e\left\|\rho^\varepsilon \right\|_{L^\infty(\mathbb{R}^+; H^{N}_x)} \left\|E^{\varepsilon} \right\|_{L^2(\mathbb{R}^+; H^{N}_x)}
\!+\!\left\|u^\varepsilon \right\|_{L^\infty(\mathbb{R}^+;\! H^{N}_x)} \left\|B^{\varepsilon} \right\|_{L^2(\mathbb{R}^+; \!H^{N}_x)}\\
\;&
+\frac{C}{\varepsilon} \left\|(\mathbf I-\mathbf P) f^{\varepsilon}\right\|_{L^2(\mathbb{R}^+; H^{N}_{x} N^s_\gamma)}\\
\lesssim\;& C
\end{split}
\end{align}
for any $T > 0$ and $0 < \varepsilon \leq 1$. Furthermore, we can derive from the definition of $u^\varepsilon$ in \eqref{limit:7} and the uniform energy bound \eqref{limit:1-1} that for any $T > 0$ and $0 < \varepsilon \leq 1$,
\begin{align}\label{partial u 2}
\left\| \mathcal{P}u^\varepsilon \right\|_{L^\infty (0,T; H^N_x)} \leq C.
\end{align}
Then, from Aubin--Lions--Simon Theorem in Lemma \ref{Aubin--Lions--Simon Theorem}, the bounds \eqref{partial u 1}, \eqref{partial u 2} and the embedding \eqref{embedding compact}, we deduce that there exists a function $\widetilde{u} \in L^\infty(\mathbb{R}^+; H^N_x) \cap C(\mathbb{R}^+; H^{N-1}_{loc})$ such that
\begin{align}\nonumber
\mathcal{P}u^\varepsilon \to \widetilde{u}
 \;\;\;\text{~~strongly} {\text{~in}~ C(0,T; H^{N-1}_{loc})}
\end{align}
for any $T > 0$ as $\varepsilon \to 0$. Moreover, from the convergence \eqref{limit:10-2} and the incompressibility \eqref{limit:13}, we have
\begin{align}\nonumber
\widetilde{u}=\mathcal{P}u=u.
\end{align}
Consequently,
\begin{align}\nonumber
\mathcal{P}u^\varepsilon \to u
\;\;\; \text{weakly}\!-\!*~ \text{in~} L^\infty(\mathbb{R}^+;H^N_{x}),\;
\text{~~strongly} {\text{~in}~C(\mathbb{R}^+; H^{N-1}_{loc})}
\end{align}
as $\varepsilon \to 0$, where $u \in L^\infty(\mathbb{R}^+; H^N_x) \cap C(\mathbb{R}^+; H^{N-1}_{loc})$. Therefore, we have that
\begin{align}\nonumber
\mathcal{P}^{ \bot } u^\varepsilon \to 0
\;\;\; \text{weakly}\!-\!*~ \text{in~} L^\infty(\mathbb{R}^+;H^N_{x})\;
\text{~~strongly} \text{~in}~ C(\mathbb{R}^+;H^{N-1}_{loc})
\end{align}
as $\varepsilon \to 0$. Here, $\mathcal{P}^{ \bot } := I-\mathcal{P}$.

Finally, we consider the convergences of $E^\varepsilon$ and $B^\varepsilon$ . From  \eqref{limit:8-1:4} and the uniform   bounds \eqref{limit:1-1}, \eqref{limit:9}, we have
\begin{align}
\bsp\label{partial nabla phi 1}
\left\| \partial_t E^\varepsilon \right\|_{L^2(0,T;H^{N-1}_x)} \;&=  \left\| \nabla_x\times B^\varepsilon-u^\e \right\|_{L^2(0,T;H^{N-1}_x)} \leq  C,\\
\left\| \partial_t B^\varepsilon \right\|_{L^2(0,T;H^{N-1}_x)} \;&=  \left\| \nabla_x\times E^\varepsilon \right\|_{L^2(0,T;H^{N-1}_x)} \leq  C
\esp
\end{align}
for any $T > 0$ and $0 < \varepsilon \leq 1$. Furthermore, from the bound \eqref{limit:1-1}, we deduce that for any $T > 0$ and $0 < \varepsilon \leq 1$
\begin{align}\label{partial nabla phi 2}
\left\|E^\varepsilon \right\|_{L^\infty(0,T;H^{N}_x)}+
\left\|B^\varepsilon \right\|_{L^\infty(0,T;H^{N}_x)} \leq C.
\end{align}
Then, from Aubin--Lions--Simon Theorem in Lemma \ref{Aubin--Lions--Simon Theorem} and the bounds \eqref{partial nabla phi 1},
\eqref{partial nabla phi 2}, we deduce that
\begin{align}\nonumber
\left[ E^\varepsilon, B^\e \right] \to \left[ E, B  \right]
\;\;\; \text{weakly}\!-\!*~ \text{in~} L^\infty(\mathbb{R}^+;H^{N}_{x}),\;
\text{~~strongly} {\text{~in}~C(\mathbb{R}^+; H^{N-1}_{loc})}
\end{align}
as $\varepsilon \to 0$, where $\left[ E, B  \right]\in L^\infty(\mathbb{R}^+; H^{N}_x) \cap C(\mathbb{R}^+; H^{N-1}_{loc})$.

In summary, we have deduced the following convergences:
\begin{align}\label{limit:17}
\!\!\!\Big( \frac{3}{5}\theta^\varepsilon-\frac{2}{5}\rho^\varepsilon, \mathcal{P}u^\varepsilon, E^\e, B^\e \Big)
\to ( \theta, u,  E, B )
\;\;\; \text{weakly}\!-\!*~ \text{in~} L^\infty(\mathbb{R}^+;H^{N}_{x}),\;
\end{align}
strongly in $C(\mathbb{R}^+; H^{N-1}_{loc})$ as $\varepsilon \to 0$, where $( \theta, u, E, B) \in L^\infty(\mathbb{R}^+; H^N_x) \cap C(\mathbb{R}^+; H^{N-1}_{loc})$  and
\begin{align}\label{limit:18}
\left( \rho^\varepsilon + \theta^\varepsilon, \mathcal{P}^{\bot}u\right) \to (0,0)
\;\;\; \text{weakly}\!-\!*~ \text{in~} L^\infty(\mathbb{R}^+;H^N_{x}),\;
\text{~~strongly} \text{~in}~ C(\mathbb{R}^+;H^{N-1}_{loc})
\end{align}
as $\varepsilon \to 0$.

\subsubsection{Equations of  \texorpdfstring{$\rho$, $u$, $\theta$, $E$} {Lg} and \texorpdfstring{$B$}{Lg}}
\hspace*{\fill}

We first calculate the term
\begin{align}\nonumber
 \frac{1}{\varepsilon} \left\langle\widehat{\Xi}(v) \mu^{1/2}, L(\mathbf{I}-\mathbf{P})f^{\e}\right\rangle_{L^2_{v}} ,
\end{align}
where $\Xi=A$ or $B$.
Following the standard formal derivations of fluid dynamic limits of the Boltzmann equation (see \cite{BGL93} for instance), we obtain
\begin{align}
\label{A(v):weiguan:zuoyong}
\displaystyle\frac{1}{\varepsilon} \Big\langle \widehat{A}(v)\mu^{1/2}, L(\mathbf{I}-\mathbf{P})f^{\e}  \Big\rangle_{L^2_v}
=\;& u^{\varepsilon}\otimes u^{\varepsilon}-\frac{|u^{\varepsilon}|^2}{3}\mathbb{I}_{3 \times 3}-\nu\Sigma(u^{\varepsilon})-R^\varepsilon_A,\\[2mm]
\displaystyle\frac{1}{\varepsilon}\Big\langle \widehat{B}(v)\mu^{1/2}, L(\mathbf{I}-\mathbf{P})f^{\e}\Big\rangle_{L^2_v}
=\;&\frac{5}{2}u^{\varepsilon}\theta^{\varepsilon}-\frac{5}{2}\kappa\nabla_x\theta^{\varepsilon}-R^\varepsilon_B, \label{B(v):weiguan:zuoyong}
\end{align}
where
\begin{align}
\nu := \;&\frac{1}{10}\Big\langle \widehat{A}(v) \mu^{1/2}, A(v)\mu^{1/2} \Big\rangle_{L^2_v}, \label{nu define}\\
\kappa := \;&\frac{2}{15}\Big\langle \widehat{B}(v) \mu^{1/2}, B(v)\mu^{1/2} \Big\rangle_{L^2_v}, \label{kappa define}\\
\Sigma(u^{\varepsilon}) := \;&\nabla_x u^{\varepsilon}+({\nabla_xu^{\varepsilon}})^{T}-\frac{2}{3}\nabla_x\cdot u^{\varepsilon}\mathbb{I}_{3 \times 3}, \nonumber\\
R^\varepsilon_\Xi := \;&\left\langle \hat{\Xi}(v)\sqrt{\mu},\;\;\e\partial_t f^{\e} \right\rangle_{L_v^2}
+\left\langle \hat{\Xi}(v)\sqrt{\mu},\;\;v\cdot\nabla_x(\mathbf{I}-\mathbf{P})f^{\e} \right\rangle_{L_v^2} \nonumber\\
\;&+\Big\langle \hat{\Xi}(v)\sqrt{\mu},\;\; \e \left( \e E^\varepsilon+ v\times B^\e\right)\cdot
\frac{\nabla_v \left({\sqrt{\mu}f^{\e}}\right)}{{\sqrt{\mu}}}\Big\rangle_{L_v^2}
-\left\langle \hat{\Xi}(v)\sqrt{\mu},\;\; \Gamma((\mathbf{I}-\mathbf{P})f^{\e},f^{\e})\right\rangle_{L_v^2} \nonumber\\
\;&-\left\langle \hat{\Xi}(v)\sqrt{\mu},\;\; \Gamma(\mathbf{P}f^{\e},(\mathbf{I}-\mathbf{P})f^{\e})\right\rangle_{L_v^2}
 \nonumber
\end{align}
with $\Xi = A$ or $B$.

The relation \eqref{A(v):weiguan:zuoyong}, the decomposition $u^{\varepsilon}=\mathcal{P}u^{\varepsilon}+\mathcal{P}^{\bot} u^{\varepsilon}$ and the equation \eqref{Pu equation} yield that
\begin{equation}\label{jixian:u:zongjie}
\partial_t\mathcal{P}u^{\varepsilon}+\mathcal{P}\nabla_x \cdot (\mathcal{P}u^{\varepsilon}\otimes\mathcal{P}u^{\varepsilon})
-\nu\Delta_x\mathcal{P}u^{\varepsilon}=\mathcal{P}(E^\e+u^\e\times B^\e)+R^{\varepsilon}_u,
\end{equation}
where $R^\varepsilon_u$ is given by
\begin{equation}\label{jixian:Ru:zongjie}
R^\varepsilon_u := \e\rho^\e E^{\e}+\mathcal{P}\nabla_x\cdot R^\varepsilon_A-\mathcal{P}\nabla_x\cdot
(\mathcal{P}u^{\varepsilon}\otimes\mathcal{P}^\perp u^{\varepsilon}+\mathcal{P}^\bot u^{\varepsilon}\otimes\mathcal{P}u^{\varepsilon}+\mathcal{P}^\perp u^{\varepsilon}\otimes\mathcal{P}^\perp u^{\varepsilon}).
\end{equation}

The relation \eqref{B(v):weiguan:zuoyong}, the decomposition $u^{\varepsilon}=\mathcal{P}u^{\varepsilon}+\mathcal{P}^{\bot} u^{\varepsilon}$ and the equation \eqref{rho theta equation} yield that
\begin{equation}\label{jixian:theta:zongjie}
\partial_t\Big(\frac{3}{5}\theta^{\varepsilon}-\frac{2}{5}\rho^{\varepsilon}\Big)
+\nabla_x \cdot \Big[\mathcal{P} u^\varepsilon \Big(\frac{3}{5}\theta^\varepsilon -\frac{2}{5}\rho^\varepsilon\Big)\Big]
-\kappa \Delta_x \Big(\frac{3}{5}\theta^\varepsilon -\frac{2}{5}\rho^\varepsilon\Big)=R^\varepsilon_\theta,
\end{equation}
where
\begin{align}
\begin{split}\label{jixian:Rtheta:zongjie}
R^\varepsilon_\theta := \;&\frac{2}{5}\nabla_x \cdot R^\varepsilon_B
-\frac{2}{5} \nabla_x \cdot \big[\mathcal{P} u^\varepsilon(\rho^\varepsilon+\theta^\varepsilon)\big]
-\frac{2}{5}\nabla_x \cdot\big[\mathcal{P}^{\bot} u^\varepsilon(\rho^\varepsilon+\theta^\varepsilon) \big]\\
&-\nabla_x \cdot \Big[\mathcal{P}^{\bot} u^\varepsilon \Big(\frac{3}{5}\theta^\varepsilon -\frac{2}{5}\rho^\varepsilon \Big)\Big]
+\frac{2}{5}\kappa \Delta_x(\rho^\varepsilon+\theta^\varepsilon)
+\frac{2}{5}\varepsilon E^\varepsilon\cdot u^{\varepsilon}.
\end{split}
\end{align}

Now, we take the limit from the equation \eqref{jixian:u:zongjie} to obtain the $u$-equation of \eqref{INSFP limit}. For any given $T > 0$, choose a vector-valued test function $\Psi(t,x) \in C^1(0,T;C_0^\infty(\mathbb{R}^3))$ with $\nabla_x \cdot \Psi=0$, $\Psi(0,x)=\Psi_0(x) \in C_0^\infty(\mathbb{R}^3)$ and $\Psi(t,x)=0$ for $t \geq T^\prime$, where $T^\prime < T$. Multiplying \eqref{jixian:u:zongjie} by  $\Psi(t,x)$ and integrating over $(t,x) \in [0,T] \times \mathbb{R}^3$, we have
\begin{align}
\begin{split}\nonumber
&\int_0^T\int_{\mathbb{R}^3}\partial_t\mathcal{P}u^{\varepsilon}\cdot \Psi(t,x) \d x\d t \\	
=\;&-\int_{\mathbb{R}^3}\mathcal{P}u^{\varepsilon}(0,x)\cdot \Psi(0,x) \d x
-\int_0^T\int_{\mathbb{R}^3}\mathcal{P}u^{\varepsilon}\cdot\partial_t \Psi(t,x) \d x \d t\\
=\;&-\int_{\mathbb{R}^3}\mathcal{P} \Big\langle f^{\varepsilon}_{0},v \mu^{1/2} \Big\rangle_{L^2_v} \cdot\Psi_0(x)\d x
-\int_0^T\int_{\mathbb{R}^3}\mathcal{P}u^{\varepsilon}\cdot\partial_t \Psi(t,x) \d x\d t.
\end{split}
\end{align}
The initial conditions in Theorem \ref{mainth3} and the convergence \eqref{limit:17} give us that
\begin{align*}
\int_{\mathbb{R}^3}\mathcal{P}\Big\langle f^{\varepsilon}_{0}, v \mu^{1/2}\Big\rangle_{L^2_v}\cdot \Psi_0(x)\d x
\to \int_{\mathbb{R}^3}  \mathcal{P} \Big\langle f_{0}, v\mu^{1/2} \Big\rangle_{L^2_v}\cdot \Psi_0(x)\d x
=\int_{\mathbb{R}^3} \mathcal{P} u_{0} \cdot \Psi_0(x)\d x
\end{align*}
and
\begin{align*}
\int_0^T\int_{\mathbb{R}^3}\mathcal{P}u^{\varepsilon}\cdot\partial_t \Psi(t,x) \d x\d t
\to \int_0^T\int_{\mathbb{R}^3}u\cdot\partial_t \Psi(t,x) \d x\d t
\end{align*}
as $\varepsilon \to 0$. That is, we deduce
\begin{align}\label{jixian:partial u 1}
\int_0^T\int_{\mathbb{R}^3}\partial_t\mathcal{P}u^{\varepsilon}\cdot \Psi(t,x) \d x\d t
\to-\int_{\mathbb{R}^3}\mathcal{P}u_{0}\cdot \Psi_0(x)\d x-\int_0^T\int_{\mathbb{R}^3}u\cdot\partial_t \Psi(t,x) \d x\d t
\end{align}
as $\varepsilon \to 0$.
With the aid of the bounds \eqref{limit:1-1}, \eqref{limit:1-2}, \eqref{limit:9} and the convergences \eqref{limit:17}--\eqref{limit:18}, we show the following convergences
\begin{equation}
 \left\{
\begin{array}{ll}\label{jixian:partial u 2}
\mathcal{P}\nabla_x\cdot(\mathcal{P}u^{\varepsilon}\otimes\mathcal{P}u^{\varepsilon}) \to
\mathcal{P}\nabla_x\cdot(u\otimes u) \quad\text{strongly in}~ \; C(\mathbb{R}^+;H^{N-2}_{loc}),\\[2mm]
\nu\Delta_x\mathcal{P}u^{\varepsilon}\to
\nu\Delta_x u \;\;\;\qquad\qquad\qquad\quad\quad \text{~~in the sense of distributions},\\[2mm]
\mathcal{P}(E^\e+u^\e\times B^\e)\to
\mathcal{P}(E+u\times B) \;\;\;\;\;\,\,\text{strongly in}~\; C(\mathbb{R}^+;H_{loc}^{N-1}),\\[2mm]
R^\varepsilon_u \to 0 \;\quad\quad\quad\quad\quad\quad\quad\quad\quad\quad\quad\quad\quad \text{~~in the sense of distributions}
\end{array}
\right.
\end{equation}
as $\varepsilon \to 0$. Therefore, the convergences \eqref{jixian:partial u 1}, \eqref{jixian:partial u 2} and the incompressibility \eqref{limit:13} imply that $u \in L^\infty(\mathbb{R}^+;H^{N}_x) \cap C(\mathbb{R}^+; H^{N-1}_{loc})$  subjects to
\begin{equation}
\left\{
\begin{array}{ll}\label{jixian:u equation}
\partial_t u+\mathcal{P}\nabla_x\cdot(u\otimes u)-\nu\Delta_x u =\mathcal{P}(E+u\times B ),\\[2mm]
\nabla_x \cdot u=0
\end{array}
\right.
\end{equation}
with initial data $u(0,x)=\mathcal{P}u_0(x)$.

Next, we take the limit from the equation \eqref{jixian:theta:zongjie} to gain the $\theta$-equation of \eqref{INSFP limit}. For any given $T > 0$, choose a test function $\Phi(t,x)  \in C^1(0,T;C_0^\infty(\mathbb{R}^3))$ with $\Phi(0,x)=\Phi_0(x) \in C_0^\infty(\mathbb{R}^3)$ and $\Phi(t,x)=0$ for $t \geq T^\prime$, where $T^\prime < T$. Then from the initial conditions in Theorem \ref{mainth3} and the convergence \eqref{limit:17}, we have
\begin{align}
\begin{split}\label{jixian:partial theta 1}
&\int_{0}^{T}\int_{\mathbb{R}^3}\partial_t\Big(\frac{3}{5}\theta^{\varepsilon}-\frac{2}{5}\rho^{\varepsilon}\Big) \Phi(t,x) \d x \d t \\
=\;&-\int_{\mathbb{R}^3} \Big\langle f^{\varepsilon}_{0},(\frac{|v|^2}{5}-1) \mu^{1/2} \Big\rangle \Phi_0(x)\d x
-\int_{0}^{T}\int_{\mathbb{R}^3} \Big(\frac{3}{5}\theta^{\varepsilon}-\frac{2}{5}\rho^{\varepsilon}\Big)\partial_t \Phi(t,x)  \d x \d t\\
\to\;&-\int_{\mathbb{R}^3} \Big\langle f_{0}, (\frac{|v|^2}{5}-1) \mu^{1/2} \Big\rangle \Phi_0(x)\d x
-\int_{0}^{T}\int_{\mathbb{R}^3}  \theta(t,x)\partial_t \Phi(t,x)  \d x \d t\\
=\;&-\int_{\mathbb{R}^3}\Big(\frac{3}{5}\theta_0-\frac{2}{5}\rho_0\Big) \Phi_0(x)\d x
-\int_{0}^{T}\int_{\mathbb{R}^3}  \theta(t,x)\partial_t \Phi(t,x)  \d x \d t
\end{split}
\end{align}
as $\varepsilon \to 0$. Furthermore, with the help of the bounds \eqref{limit:1-1}--\eqref{limit:1-2}, \eqref{limit:9} and the convergences \eqref{limit:17}--\eqref{limit:18}, we derive the following convergences
\begin{equation}
 \left\{
\begin{array}{ll}\label{jixian:partial theta 2}
\nabla_x\cdot\Big[\mathcal{P}u^{\varepsilon}\Big(\frac{3}{5}\theta^{\varepsilon}-\frac{2}{5}\rho^{\varepsilon}\Big)\Big]
\to	\nabla_x\cdot(u\theta)\quad \,\,\text{strongly in}~\; C(\mathbb{R}^+;H_{loc}^{N-2}), \\[2mm]
\kappa\Delta_x\Big(\frac{3}{5}\theta^{\varepsilon}-\frac{2}{5}\rho^{\varepsilon}\Big)\to \kappa\Delta_x\theta
\;\;\;\;\,\,\qquad\qquad \text{~~in the sense of distributions},\\[2mm]
R^\varepsilon_\theta \to 0  \,\;\;\;\quad\quad\quad\quad\quad\quad\quad\quad\qquad\quad\quad\quad\text{in the sense of distributions}
\end{array}
\right.
\end{equation}
as $\varepsilon \to 0$.
Therefore, the convergences \eqref{jixian:partial theta 1}, \eqref{jixian:partial theta 2} and the incompressibility \eqref{limit:13} imply that that $\theta \in L^\infty(\mathbb{R}^+;H^{N}_x) \cap C(\mathbb{R}^+; H^{N-1}_{loc})$ obeys
\begin{align}\label{jixian:theta equation}
\partial_t\theta+ u \cdot \nabla_x\theta-\kappa \Delta_x\theta=0
\end{align}
with the initial data $\theta(0,x)=\frac{3}{5}\theta_{0}(x)-\frac{2}{5}\rho_{0}(x)$.

Finally, the convergence \eqref{limit:17}   and the equation
\eqref{limit:8-1:4} show that
\begin{align*}
	\left\{\begin{array}{l}
\displaystyle \pt_tE -\nabla_x\times B =
    -u, \quad\qquad\;\,\,\,
	\nabla_x\cdot E =0 , \\ [2mm]
 \displaystyle
    \pt_tB +\nabla_x\times E =
    0, \qquad\qquad	\;\;\;		
	\nabla_x\cdot B =0
	\end{array}\right.
\end{align*}
with the initial date  $E(0, x)=E_0(x)$ and $B(0, x)=B_0(x)$.

\subsubsection{Summarization}
\hspace*{\fill}

Collecting all above convergence results, we have shown that
\begin{align*}
&\left(\rho, u, \theta, E, B \right) \in L^\infty(\mathbb{R}^+; H^N_{x}) \cap  C(\mathbb{R}^+; H^{N-1}_{loc}),
\end{align*}
 and the following  incompressible NSFM system
\begin{equation} \label{jixian fangchengzong}
\left\{
\begin{array}{ll}
\displaystyle \partial_{t} u+u \cdot \nabla_{x} u-\nu \Delta_{x} u+\nabla_{x}P=E+u\times B,  &\nabla_{x} \cdot u=0,\\[2mm]
\displaystyle \partial_{t} \theta+u \cdot \nabla_{x} \theta-\kappa \Delta_{x} \theta=0, &\rho+\theta=0,\\[2mm]
\displaystyle \partial_{t} E-\nabla_{x}\times B=-u, &\nabla_{x} \cdot E=0,\\[2mm]
\displaystyle \partial_{t} B+\nabla_{x}\times E=0, &\nabla_{x} \cdot B=0\\[2mm]
\end{array} \right.
\end{equation}
with the initial data
\begin{align}
\begin{split}\nonumber
u(0,x)=\mathcal{P}u_0(x),\;\; \theta (0,x)=\frac{3}{5}\theta_0(x)-\frac{2}{5}\rho_0(x),\;\;    E(0,x)=E_0(x),\;\;   B(0,x)=B_0(x).&
\end{split}
\end{align}

Actually, by mollifying the system \eqref{jixian fangchengzong} and combining the local existence theory, the standard energy method and the continuity argument,  we can show that $(\rho, u, \theta, E, B) \in C(\mathbb{R}^{+}; H_x^N) $, cf. Chapter 3.2 in  \cite{MA-90}, and then \eqref{jixian solution space} holds true.  Moreover, by approximating $\mathbb{R}^3_x$ by bounded
region \cite{CK2006} and the unform bounds of functions and their limits proved above, the convergences in \eqref{theorem1.3 4} are thus verified.
This completes the proof of Theorem \ref{mainth3}.

\appendix
\section{Appendix}\label{Appendix}

For convenience, we list the some basic results and inequalities to be used   throughout this paper.

Take that $w_{\ell} \equiv w_{\ell}(v)=\langle v \rangle^\ell e^{\frac{q\langle v\rangle^2}{(1+t)^\vartheta}}$.
The first lemma concerns the coercivity estimate on the linearized operator $L$, proved by \cite{DYZ2002,guo2006}.
\begin{lemma}\label{L coercive estimate result}
Let $0\leq \gamma\leq1$.
\begin{itemize}
\setlength{\leftskip}{-5mm}
\item[(1)] There exists a constant $C\geq 0$ such that the uniform coercive lower bound estimate
\begin{equation}\label{L coercive2}
\left\langle Lf, w_{\ell}^2f \right\rangle_{L^2_v}  \gtrsim \left| w_{\ell}f\right|^2_{L^2_v(\nu)}-C\left| f \right|^2_{L^2_v(\nu)}
\end{equation}
holds for any $\ell \in \mathbb{R}$.
\item[(2)]  For $|\beta| \geq 1$ be a multi-index, we have
\begin{equation}\label{L coercive3}
\left\langle \partial_{\beta}Lf, w_{\ell}^2\partial_{\beta}f \right\rangle_{L^2_v} \gtrsim \left| w_{\ell}\partial_{\beta}f \right|^2_{L^2_v(\nu)}
-\eta \sum_{|\beta^{\prime}| < |\beta|} \left| w_{\ell}\partial_{\beta^{\prime}}f\right|^2_{L^2_v(\nu)}-C_\eta\left| f \right|^2_{L^2_v(\nu)}
\end{equation}
for any $\ell \in \mathbb{R}$ and small $\eta>0$.
\end{itemize}
\end{lemma}
\medskip

The following   lemma  focus on  the estimate  of the nonlinear collision operator $\Gamma$, which can be found in \cite{DYZ2002}.
\begin{lemma}\label{Gamma1}
Let $0\leq \gamma\leq1$.
\begin{itemize}
\setlength{\leftskip}{-6mm}
 \item[(1)] There holds
 \bal\label{hard gamma}
 |\nu^{-\frac{1}{2}} \Gamma (f, g)|_{L^2_v}\lesssim
 |\nu^{\frac{1}{2}} f|_{L^2_v} |g|_{L^2_v}
 + |\nu^{\frac{1}{2}} g|_{L^2_v} |f|_{L^2_v}.
 \eal
\item[(2)]
For  $\a\geq 0$ and $\b\geq 0$ be two multi-indexs. Then, there holds
\bal
\bsp\label{hard gamma1}
\!\!\!  \;& \left\langle \partial_{\beta}^{\alpha} \Gamma (f, g) ,
w_{\ell}^2 \partial_{\beta}^{\alpha} h \right\rangle_{L^2_v}\\
\lesssim \;& \sum \left\{ \left|w_{0}\partial_{\beta_1}^{\alpha_1} f\right|_{L^2_v} \left|w_{\ell} \partial_{\beta_2}^{\alpha_2} g\right|_{L^2_v(\nu)}
+\left|w_{\ell} \partial_{\beta_1}^{\alpha_1} f\right|_{L^2_v}\left|w_{0}\partial_{\beta_2}^{\alpha_2} g\right|_{L^2_v(\nu)}\right\}
\left|w_{\ell} \partial_\beta^\alpha h\right|_{L^2_v(\nu)}.
\esp
\eal
 where the summation $\sum$ is taken over $\alpha_1+\alpha_2=\alpha$ and $\beta_1+\beta_2 \leq \beta$.

\end{itemize}
\end{lemma}
\medskip

In what follows, we present   the following Sobolev interpolation inequalities was shown in \cite{GW2012CPDE}.

\begin{lemma}\label{sobolev interpolation}
 Let $2 \leq p<\infty$ and $k, \ell, m \in \mathbb{R}$. Then for any $f \in C_0^\infty(\mathbb{R}^3)$, we have
\begin{align}\label{sobolev interpolation1}
\|\nabla^k_x f\|_{L^p_x} \lesssim\|\nabla^{\ell}_x f\|_{L^2_x}^\theta\|\nabla^m_x f\|_{L^2_x}^{1-\theta},
\end{align}
where $0 \leq \theta \leq 1$ and $\ell$ satisfies
\begin{align}\nonumber
\frac{1}{p}-\frac{k}{3}=\left(\frac{1}{2}-\frac{\ell}{3}\right) \theta+\left(\frac{1}{2}-\frac{m}{3}\right)(1-\theta).
\end{align}
For the case $p=+\infty$, we have
\begin{align}\label{sobolev interpolation2}
\|\nabla^k_x f\|_{L^{\infty}_x} \lesssim\|\nabla^{\ell}_x f\|_{L^2_x}^\theta\|\nabla^m_x f\|_{L^2_x}^{1-\theta},
\end{align}
where $\ell \leq k+1$, $m \geq k+2$, $0 \leq \theta \leq 1$ and $\ell$ satisfies
\begin{align}\nonumber
-\frac{k}{3}=\left(\frac{1}{2}-\frac{\ell}{3}\right) \theta+\left(\frac{1}{2}-\frac{m}{3}\right)(1-\theta).
\end{align}
\end{lemma}
\medskip

If $s \in (0,3)$, then $\Lambda^{-s}$ is the Riesz potential operator. The
Hardy--Littlewood--Sobolev theorem implies the following inequality for the Riesz potential operator $\Lambda^{-s}$.

\begin{lemma}\label{negative embedding theorem}
Let $0<s<3$, $1< p < q < \infty$, $1/q + s/3=1/p$, then
\begin{align}\label{negative embed 1}
\left\|\Lambda^{-s} f\right\|_{L^q_x} \lesssim\|f\|_{L^p_x}.
\end{align}
\end{lemma}
\medskip

Further, based on Lemma \ref{sobolev interpolation}, we have the following lemma to be used frequently in this article.
\begin{lemma}\label{negative embedding theorem2}
Let $0<s<3/2$. Then there holds
\begin{align}\label{negative embed 2}
\|f\|_{L^{\frac{12}{3+2 s}}_x} \lesssim \|\Lambda^{\frac{3}{4}-\frac{s}{2}} f \|_{L^2_x}, \quad\|f\|_{L^{\frac{3}{s}}_x} \lesssim  \|\Lambda^{\frac{3}{2}-s} f \|_{L^2_x}.
\end{align}
\end{lemma}
\medskip

Finally, we  also use the following Minkowski inequality to interchange the order of a multiple integral, cf. \cite{GW2012CPDE}.
\begin{lemma}\label{minkowski theorem}
For $1 \leq p \leq q \leq \infty$, we have
\begin{align}\label{minkowski}
\|f\|_{L_x^q L_v^p} \leq\|f\|_{L_v^p L_x^q}.
\end{align}
\end{lemma}

\medskip
\noindent\textbf{Acknowledgements.}
The research is supported by NSFC under the grant number 12271179,
Basic and Applied Basic Research Project of Guangdong under the grant number 2022A1515012097,
and Basic and Applied Basic Research Project of Guangzhou under the grant number SL2022A04J01496.

\medskip
\noindent\textbf{Declarations}
\medskip

\noindent\textbf{Conflict of interest}
The authors declare that they have no conflict of interest.

\end{document}